\documentclass[a4paper, 11pt]{amsart}
\usepackage{amsfonts, amsthm, amssymb, amsmath}
\usepackage{mathrsfs,array}
\usepackage{xy}
\usepackage{hyperref}
\usepackage{verbatim}
\usepackage{cancel}
\usepackage[latin1]{inputenc}

\input xy
\xyoption{all}

\newtheorem{theo}{Theorem}[section]
\newtheorem{prop}[theo]{Proposition}
\newtheorem{defi}[theo]{Definition}
\newtheorem{lemm}[theo]{Lemma}
\newtheorem{coro}[theo]{Corollary}

\newtheorem{rema}[theo]{Remark}

\newtheorem{conv}[theo]{Convention}

\newcommand{\wh}{\widehat}
\newcommand{\mb}{\mathbb}
\newcommand{\mc}{\mathcal}
\newcommand{\mf}{\mathfrak}
\newcommand{\mbf}{\mathbf}
\newcommand{\ms}{\mathscr}
\newcommand{\ra}{\rightarrow}

\newcommand{\scl}{\mathscr{L}}

\newcommand{\gl}{\mathrm{GL}_{2}}
\newcommand{\sub}{\subseteq}
\newcommand{\Xinf}{\mc{X}_{\infty}}
\newcommand{\Xinfw}{\mc{X}_{\infty,w}}

\newcommand{\Uinfw}{\mc{U}_{\infty,w}}
\newcommand{\glq}{\mathrm{GL}_{2}(\mathbb{Q}_{p})}
\newcommand{\OC}{\mathcal{O}_{C}}
\newcommand{\oo}{\mathcal{O}}
\newcommand{\Zp}{\mathbb{Z}_{p}}
\newcommand{\Qp}{\mathbb{Q}_{p}}
\newcommand{\Pro}{\mathbb{P}^{1}}

\newcommand{\inj}{\hookrightarrow}
\newcommand{\U}{\mc{U}}
\newcommand{\ctens}{\widehat{\otimes}}
\newcommand{\Xw}{\mathcal{X}_{w}}
\newcommand{\Yw}{\mathcal{Y}_{w}}
\newcommand{\vp}{\varpi}
\newcommand{\Xetw}{X_{\textrm{et}}^{w}}

\newcommand{\ohat}{\widehat{\mathcal{O}}}
\newcommand{\ohatp}{\widehat{\mathcal{O}}^{+}}
\newcommand{\Zpn}{\mathbb{Z}_{p}^{n}}
\newcommand{\wt}{\widetilde}
\newcommand{\surj}{\twoheadrightarrow}
\newcommand{\XC}{X(\mathbb{C})}
\newcommand{\X}{\mathcal{X}}
\newcommand{\Dp}{\Delta_{0}(p)}
\newcommand{\Cp}{\mathbb{C}_{p}}
\newcommand{\dg}{\dagger}
\newcommand{\V}{\mathcal{V}}

\DeclareMathOperator{\coker}{coker}
\DeclareMathOperator{\nc}{nc}
\DeclareMathOperator{\Loc}{Loc}

\DeclareMathOperator{\Hdg}{Hdg}
\DeclareMathOperator{\Sym}{Sym}

\DeclareMathOperator{\Hom}{Hom}

\DeclareMathOperator{\End}{End}

\DeclareMathOperator{\Spa}{Spa}
\DeclareMathOperator{\Spec}{Spec}
\DeclareMathOperator{\MaxSpec}{MaxSpec}
\DeclareMathOperator{\Spf}{Spf}

\DeclareMathOperator{\Frob}{Frob}

\DeclareMathOperator{\GL}{GL}
\DeclareMathOperator{\SL}{SL}
\DeclareMathOperator{\Ker}{Ker}

\DeclareMathOperator{\Sp}{Sp}

\DeclareMathOperator{\Lie}{Lie}
\DeclareMathOperator{\HT}{HT}

\DeclareMathOperator{\univ}{univ}
\DeclareMathOperator{\rig}{rig}
\DeclareMathOperator{\pr}{pr}
\DeclareMathOperator{\gen}{gen}
\DeclareMathOperator{\et}{\acute{e}t}
\DeclareMathOperator{\proet}{pro\acute{e}t}
\DeclareMathOperator{\new}{new}
\DeclareMathOperator{\disc}{disc}
\DeclareMathOperator{\sm}{sm}
\DeclareMathOperator{\ad}{ad}

\title{Overconvergent modular forms and perfectoid Shimura curves}
\date {\today}
\author{Przemys\l aw Chojecki, David Hansen and Christian Johansson}
\email{chojecki@maths.ox.ac.uk}
\email{hansen@math.columbia.edu}
\email{johansson@math.ias.edu}

\begin{document}

\begin{abstract}
We give a new construction of overconvergent modular forms of arbitrary weights, defining them in terms of functions on certain affinoid subsets of Scholze's infinite-level modular curve.  These affinoid subsets, and a certain canonical coordinate on them, play a role in our construction which is strongly analogous with the role of the upper half-plane and its coordinate `$z$' in the classical analytic theory of modular forms.  As one application of these ideas, we define and study an overconvergent Eichler-Shimura map in the context of compact Shimura curves over $\mb{Q}$, proving stronger analogues of results of Andreatta-Iovita-Stevens.
\end{abstract}

\maketitle
\tableofcontents

\section{Introduction}

Let $N\geq5$ be an integer, and let $Y_{1}(N)(\mb{C})$ be the
usual analytic modular curve. A holomorphic modular form $f$ of weight
$k$ and level $N$ admits two rather distinct interpretations, which
one might call the \emph{algebraic }and \emph{analytic }points of
view:

\medskip

\textbf{Algebraic: $f$ }is a global section $\omega(f)$ of the line
bundle $\omega^{\otimes k}=\omega_{E/Y_{1}(N)}^{\otimes k}$ on $Y_{1}(N)(\mb{C})$ (extending to $X_{1}(N)(\mb{C})$).
Equivalently, $f$ is a rule which assigns to each \emph{test object
}- an isomorphism class of triples $(S,E,\eta)$ consisting of a complex
analytic space $S$, a (generalized) elliptic curve $\pi:E\to S$
with level $N$ structure, and $\eta$ an $\mathcal{O}_{S}$-generator
of $\omega_{E/S}:=(\mathrm{Lie}E/S)^{\ast}$ - an element $\mathbf{f}(S,E,\eta)\in\mathcal{O}(S)$
such that $\mathbf{f}$ commutes with base change in $S$ and satisfies $\mathbf{f}(S,E,a\eta)=a^{-k}\mathbf{f}(S,E,\eta)$
for any $a\in\mathcal{O}(S)^{\times}$. These two interpretations
are easily seen to be equivalent (for simplicity we are ignoring cusps): for any test object, there is a
canonical map $s:S\to Y_{1}(N)$ realizing $E/S$ as the pullback
of the universal curve over $Y_{1}(N)$, and one has $s^{\ast}\omega(f)=\mathbf{f}(S,E,\eta)\eta^{\otimes k}$. 

\medskip

\textbf{Analytic: }$f$ is a holomorphic function on the upper half-plane
$\mathfrak{h}$ of moderate growth, satisfying the transformation
rule $f(\frac{az+b}{cz+d})=(cz+d)^{k}f(z)$ for all $\gamma\in\Gamma_{1}(N)$.

\medskip

How do we pass between these points of view? The key is that $\mathfrak{h}$
may be identified with the universal cover $\widetilde{Y_{1}(N)}$
of $Y_{1}(N)$, in the category of complex analytic spaces, and the
pullback of the line bundle $\omega$ to $\widetilde{Y_{1}(N)}$ is
canonically trivialized. Precisely, $\widetilde{Y_{1}(N)}$ consists
of pairs $(E,\beta)$ where $E/\mb{C}$ is an elliptic curve and
$\beta=\{\beta_{1},\beta_{2}\}\in H_{1}(E(\mb{C}),\mb{Z})$
is an oriented basis%
\footnote{One also requires that $\beta_{2}$ generate the $\Gamma_{1}(N)$-structure
under $H_{1}(E(\mb{C}),\mb{Z}/N\mb{Z})\cong E(\mb{C})[N]$%
}. This space admits a left action of $\Gamma_{1}(N)$ by $\{\beta_{1},\beta_{2}\}\mapsto\{a\beta_{1}+b\beta_{2},c\beta_{1}+d\beta_{2}\}$.
Let $p:\widetilde{Y_{1}(N)}\to Y_{1}(N)$ be the natural projection.
Defining the period $z(E,\beta)\in\mb{C}$ by $\int_{\beta_{1}}\eta=z(E,\beta)\int_{\beta_{2}}\eta$,
where $\eta \neq0\in H^{0}(E(\mb{C}),\Omega_{E/\mb{C}}^{1})$
is any nonzero holomorphic one-form on $E$, the map \begin{eqnarray*}
\widetilde{Y_{1}(N)} & \overset{\sim}{\to} & \mathfrak{h}\\
(E,\beta) & \mapsto & z(E,\beta)\end{eqnarray*}
is a $\Gamma_{1}(N)$-equivariant isomorphism of Riemann surfaces.
The pullback $p^{\ast}\omega$ is then trivialized by the differential
$\eta_{\mathrm{can}}$ characterized by $\int_{\beta_{2}}\eta_{\mathrm{can}}=1$.
Defining $f(z)$ by $p^{\ast}\omega(f)=f(z)\eta_{\mathrm{can}}^{\otimes k}$,
a calculation shows that $\eta_{\mathrm{can}}(\gamma z)=(cz+d)^{-1}\eta_{\mathrm{can}}(z)$,
from which the transformation law of $f(z)$ follows.

\medskip

In the \emph{p-}adic setting a priori one only has the algebraic definition valid for integral weights. Given the importance of non-integral weights, it seems natural to hope for a direct algebraic definition
of modular forms with non-integral weights. More precisely, given a continuous character $\kappa:\mb{Z}_{p}^{\times}\to L^{\times}$ for some $L/\mb{Q}_{p}$ finite, one would like to define a {}``\emph{p-}adic modular form of weight $\kappa$ and level $N$'' as a rule which assigns to each\emph{ }test object\emph{ } an isomorphism class of triples $(R,E,\eta)$ consisting of a $p$-adically separated and complete ring $R$, a (generalized) elliptic curve $E/\mathrm{Spec}R$
with level $N$ structure, and $\eta$ a generator of $\omega_{E/R}$
- an element $g(R,E,\eta)\in R\otimes_{\mb{Z}_{p}}L$ such that
$g$ commutes with base change in $R$ and satisfies $g(R,E,r\eta)=\kappa(r)^{-1}g(R,E,\eta)$
for any $r\in R^{\times}$. The problem with this
naive definition is that the expression $\kappa(r)$ does not make sense
in general. In fact, the only characters for which this is unambiguously
defined and functorial in $R$ are the power characters $\kappa:r\to r^{k},k\in\mb{Z}$.
However, Andreatta-Iovita-Stevens and Pilloni (\cite{ais,pil}) discovered (independently)
a remarkable fix to this problem, whereby for a given character $\kappa$
one only allows {}``certain elliptic curves" and, more importantly, ``certain differentials''
in the definition of test objects. The admissible elliptic curves
are those whose Hasse invariant has suitably small valuation. The admissible $\eta$'s are
defined using torsion $p$-adic Hodge theory and the theory of the canonical subgroup; they are not permuted by all of $R^{\times}$, but only by a subgroup of
elements $p$-adically close enough to $\mb{Z}_{p}^{\times}$ so that $\kappa(r)$ is defined.

\medskip

Our first goal in this paper is to develop a $p$-adic analogue of the \emph{analytic} picture above.  Of course, the most pressing question here is: what
are the correct analogues of $\wt{Y_{1}(N)}$, $\mf{h}\sub \Pro_{\mb{C}}$
and $z$?  The following theorem gives a partial answer.

\begin{theo}\label{1.1} Given $N$ and $p \nmid N$ as above, let $$\Xinf \sim \varprojlim_{n} X_{K_1(N)K(p^n)}^{\mathrm{ad}}$$be the infinite-level perfectoid (compactified) modular curve of tame level $N$ (\cite{sch3}), with its natural right action of $\gl(\mathbb{Q}_p)$.  There is a natural family of $K_0(p)$-stable open affinoid perfectoid subsets $\Xinfw \sub \Xinf$ parametrized by rationals $w \in \mathbb{Q}_{>0}$, with $\mc{X}_{\infty,w'} \subseteq \mc{X}_{\infty,w}$ for $w \leq w'$, and there is a canonical global section $\mathfrak{z} \in \mc{O}(\mc{X}_{\infty,w})$ (compatible under changing $w$) such that $\gamma^{\ast}\mathfrak{z}=\frac{a\mathfrak{z}+c}{b\mathfrak{z}+d}$ for all $\gamma \in K_0(p)$.  For any $\kappa$ as above and any $w \gg_{\kappa} 0$, the space $$ M_{\kappa,w}(N) :=  \left\{ f \in \mc{O}(\mc{X}_{\infty,w})\otimes_{\mathbb{Q}_p}L \mid \, \kappa(b\mathfrak{z}+d) \cdot \gamma^{\ast} f = f\,\,\forall \gamma \in K_0(p) \right\}$$ is well-defined, and the module $M^{\dagger}_{\kappa}(N) = \lim_{w\to\infty}M_{\kappa,w}(N)$ is canonically isomorphic with the modules of overconvergent modular forms of weight $\kappa$ and tame level $N$ defined by Andreatta-Iovita-Stevens and Pilloni.

\end{theo}

Thus, in our approach, $\Xinfw$ plays the role of $\wt{Y_{1}(N)}$ and the ``fundamental period" $\mf{z}$ defined below plays the role of $z$. Key to this theorem is that $\omega$ is trivialized over any $\Xinfw$ (by one of the ``fake" Hasse invariants constructed in \cite{sch3}), which is similar to the fact that $\omega$ is trivialized over $\mf{h}$. One reason why the latter is true is because the complex period map is not surjective onto $\Pro_{\mb{C}}$. By contrast, the $p$-adic period morphism $\pi_{\HT}\, :\, \Xinf \ra \Pro$ constructed in \cite{sch3} \emph{is} surjective, but becomes non-surjective when restricted to any $\Xinfw$. This is the reason that our description works for overconvergent modular forms but not for classical modular forms. 

\medskip

The second goal of the paper is to apply our ``explicit" point of view to redefine and analyze the overconvergent Eichler-Shimura map of \cite{oes}, which compares overconvergent modular symbols to overconvergent modular forms. Our perspective gives a short and transparent definition of these maps, and (we believe) clarifies the ideas involved.  We also make use of certain new filtrations on overconvergent distribution modules to obtain a more ``global" point of view on the Eichler-Shimura maps: after first defining them in the setting of Coleman families, we glue them into a morphism of coherent sheaves over the whole eigencurve. 

\medskip

We now describe these results in more detail.

\subsection{Perfectoid modular curves and \emph{w}-ordinarity}

Let $Y_{n}$ denote the modular curve over $\Qp$ with
$Y_{n}(S)$ parametrizing elliptic curves $E/S$ with a point $P\in E(S)[N]$
of exact order $N$ together with an isomorphism \[
\alpha_{n}:(\mb{Z}/p^{n}\mb{Z})^{2}\overset{\sim}{\to}E(S)[p^{n}].\]
Let $X_{n}$ denote the usual compactification of $Y_{n}$, and let
$\mc{Y}_{n}\sub \X_{n}$ be the associated adic spaces
over $\Spa(\Qp,\Zp)$. These form compatible
inverse systems $(\X_{n})_{n\geq1}$, $(\mc{Y}_{n})_{n\geq1}$
with compatible actions of $\gl(\Qp)$. Fundamental to all our considerations is Scholze's construction of the infinite level modular curves $\mc{Y}_{\infty}\sim \varprojlim_{n}\mc{Y}_{n}$, $\Xinf \sim \varprojlim_{n} \X_{n}$ and the $\gl(\Qp)$-equivariant Hodge-Tate period map $\pi_{\HT}\, :\, \Xinf \ra \Pro$ (a morphism of adic spaces over $\Qp$).

\medskip

Let us say a few more words about $\pi_{\HT}$. Let $C/\Qp$
be a complete algebraically closed field extension with ring of integers
$\OC$, and let $A$ be an abelian variety over $C$ of
dimension $g$, with dual $A^{\vee}$. Set $\omega_{A}=H^{0}(A,\Omega_{A/C}^{1})\cong\mathrm{Hom}_{C}(\mathrm{Lie}\,A,C)$,
a $C$-vector space of rank $g$. Then we have a natural linear map
$\mathrm{HT}_{A}:T_{p}A\otimes_{\Zp} C\twoheadrightarrow\omega_{A^{\vee}}$,
the \emph{Hodge-Tate map }of $A$, which fits into a short exact sequence
$$ 0\longrightarrow (\mathrm{Lie}\,A)(1)\overset{\mathrm{HT}_{A^{\vee}}^{\vee}(1)}{\longrightarrow}\mathrm{Hom}(T_{p}A^{\vee},C(1))=T_{p}A\otimes_{\Zp} C\overset{\mathrm{HT}_{A}}{\longrightarrow}\omega_{A^{\vee}}\longrightarrow 0 $$
where $-(1)$ denotes a Tate twist. If $E/C$ is an elliptic curve and $\alpha\, :\, \Zp^{2} \overset{\sim}{\ra} T_{p}E$ is a trivialization, then $(E,\alpha)$ defines a point in $\mc{Y}_{\infty}(C,\OC)$ and $\pi_{\mathrm{HT}}$ sends $(E,\alpha)$ to the line $(\alpha\otimes1)^{-1}(\mathrm{Lie}\, E)\sub C^{2}$.

\medskip

Our first key definition is a new gauge for the ordinarity of an abelian variety, defined in terms of the Hodge-Tate map $\mathrm{HT}_A$.  Let $F_A=\mathrm{HT}_{A}(T_{p}A\otimes_{\Zp}\OC)$. This is an $\OC$-lattice inside $\omega_{A^{\vee}}$.

\begin{defi} Let $w\in\mb{Q}_{>0}$.
\begin{enumerate}
\item An abelian
variety $A/C$ is $w$-\textbf{ordinary} if there is a basis
$b_{1},\dots,b_{2g}$ of $T_{p}A$ such that $\mathrm{HT}_{A}(b_{i})\in p^{w}F_A$
for all $1\leq i\leq g$. 
\item For $A/C$ w-ordinary and $0<n<w+1$, the \textbf{pseudocanonical
subgroup $H_{n}$ of level $n$} is defined to be the kernel of the natural
map
$$ A[p^{n}](C) \to F_A \otimes_{\OC}\OC/p^{\mathrm{min}(n,w)}\OC $$
induced by $\mathrm{HT}_{A}$.
\item For $A/C$ w-ordinary, a trivialization $\alpha:\Zp^{2g}\overset{\sim}{\to}T_{p}A$
is \textbf{strict} if $\alpha(e_{1}),\dots,\alpha(e_{g})\,\mathrm{mod}\, p\in A[p](C)$
form a basis for the pseudocanonical subgroup of level one.
\end{enumerate}
\end{defi}

\medskip

Let $u$ and $v$ be the homogeneous coordinate with respect to the standard basis $e_{1} =\left( \begin{smallmatrix}  1 \\ 0  \end{smallmatrix}\right)$ and $e_{2} =\left( \begin{smallmatrix}  0 \\ 1  \end{smallmatrix}\right)$ of $\Qp^{2}$. We define $z=-v/u$, a coordinate function on $\Pro$. For $w\in\mb{Q}_{>0}$ we let $\Pro_{w} \sub \Pro$ be the locus 
$$ \Pro_{w} = \{z \mid  \mathrm{inf}_{a\in p\mathbb{Z}_{p}}|z-a|\leq|p|^{w} \}.$$
For $\gamma \in K_0(p)$ one sees that $\gamma^{\ast}z=\frac{az+c}{bz+d}$ and that $\Pro_{w}$
is a $K_{0}(p)$-stable affinoid, where $K_{0}(p)$ is the usual Iwahori subgroup of $\gl(\Zp)$. Define $\Xinfw = \pi_{\mathrm{HT}}^{-1}(\Pro_{w})$ and
$\mc{Y}_{\infty,w}=\Xinfw \cap \mc{Y}_{\infty}$. These loci are $K_{0}(p)$-stable. Finally, define the \emph{fundamental period} $\mf{z}=\pi_{\mathrm{HT}}^{\ast}z\in\oo_{\Xinf}^{+}(\Xinfw)$.

\begin{theo}\label{1.3} A point $(E,\alpha)\in\mc{Y}_{\infty}(C,\OC)$
is contained in $\Xinfw$ if and only if $E$ is $w$-ordinary
and $\alpha$ is strict. Furthermore, $\Xinfw$ is
the preimage of a canonical affinoid $\Xw \sub \X_{K_{0}(p)}$,
and $\Xinfw$ is affinoid perfectoid. The $(C,\OC)$-points of $\Yw = \Xw \cap \mc{Y}_{K_{0}(p)}$ are the pairs $(E,H)$ where $E$ is $w$-ordinary and $H$ is the pseudocanonical subgroup of level one, and $(\Xw)_{w}$ is a cofinal set of strict neighbourhoods of the ordinary multiplicative locus in $\X_{K_{0}(p)}$.
\end{theo}

Taking these infinite-level objects as our basic ingredients, we are able to give a short definition (Definition \ref{mainocdef}) of sheaves of overconvergent modular forms $\omega _{\kappa,w}^\dagger$ on $\mc{X}_w$ whose global sections yield the module $M_{\kappa,w}(N)$ of Theorem \ref{1.1}. This construction also works in families of weights. As a guide to our constructions so far, we offer the following table of analogies:

\medskip

\begin{flushleft}
{\footnotesize
\begin{tabular}{|c|c|}
\hline 
$\mb{C}$ & $\mb{Q}_{p}$\tabularnewline
\hline
\hline 
$Y=Y_1(N)(\mb{C})$, a complex analytic space & $\mathcal{Y}_{w}\subset \mc{Y}_{K_0(p)}$, an adic space\tabularnewline
\hline 
$E/\mb{C}$ an elliptic curve & $E/C$ a $w$-ordinary elliptic curve\tabularnewline
\hline 
$\left(E,\beta\right)$, $\beta=\{\beta_{1},\beta_{2}\}\in H_{1}(E(\mb{C}),\mb{Z})$
an oriented basis & $(E,\alpha),$ $\alpha:\mb{Z}_{p}^{2}\overset{\sim}{\to}T_{p}E$
a strict trivialization\tabularnewline
\hline 
$\left(E,\beta\right)\in\tilde{Y}$, the universal cover
of $Y$ & $x=(E,\alpha)\in\mathcal{Y}_{\infty,w}(C)\subset\mathcal{X}_{\infty,w}(C)$\tabularnewline
\hline 
$\Gamma_{1}(N)\circlearrowright\tilde{Y}\overset{\sim}{\to}\mathfrak{h}\subset\mb{P}_{/\mb{C}}^{1}$ & $\pi_{\mathrm{HT}}:\mathcal{X}_{\infty,w}\twoheadrightarrow\mb{P}_{w}^{1}\subset\mb{P}_{/\mb{Q}_{p}}^{1}$$\circlearrowleft K_{0}(p)$\tabularnewline
\hline 
$\tilde{Y}=\mathfrak{h}$ and $\overline{\mathfrak{h}}$ are contractible & $\mathcal{X}_{\infty,w}$ is affinoid perfectoid\tabularnewline
\hline 
$z$, the coordinate on $\mathfrak{h}$ & $\mathfrak{z}\in\mathcal{O}^{+}(\mathcal{X}_{\infty,w})$, the fundamental
period\tabularnewline
\hline 
$z=z(E,\beta)$ characterized by & $\mathfrak{z}(x)\in C$ characterized by \tabularnewline
$\int_{\beta_{1}}\eta=z\int_{\beta_{2}}\eta$  & $\mathrm{HT}_{E}(\alpha(e_{1}))=\mathfrak{z}(x)\mathrm{HT}_{E}(\alpha(e_{2}))$\tabularnewline
\hline 
$\int_{\beta_{2}}\eta_{\mathrm{can}}=1$ & $\mf{s}=\mathrm{HT}_{E}(\alpha(e_{2}))$\tabularnewline
\hline 
$\eta_{\mathrm{can}}(\gamma z)=\frac{1}{cz+d}\eta_{\mathrm{can}}(z)$ & $\gamma^{\ast}\mf{s}=(b\mathfrak{z}+d)\mf{s}$\tabularnewline
\hline 
$(cz+d)^{k}\in\mathcal{O}(\mathfrak{h})$ & $\kappa(b\mathfrak{z}+d)\in\mathcal{O}(\mathcal{X}_{\infty,w})\otimes_{\mb{Q}_{p}}L$
($w\geq w_{\lambda}$)\tabularnewline
\hline 
$M_{k}(N)=\left\{ f\in\mathcal{O}(\mathfrak{h})\mid f(gz)=(cz+d)^{k}f(z)\right\} $ & $M_{\kappa,w}(N)=\left\{ f\in\mathcal{O}(\mathcal{X}_{\infty,w})\otimes_{\mb{Q}_{p}}L\mid\kappa(b\mathfrak{z}+d)\gamma^{\ast}f=f\right\} $\tabularnewline
\hline 
 & $M_{\kappa}^{\dagger}=\lim_{w\to\infty}M_{\kappa,w}$\tabularnewline
\hline
\end{tabular}
Table 1.1: Analogies.}
\par\end{flushleft}

\bigskip

As we have already mentioned, sheaves of overconvergent modular forms have been constructed previously by \cite{ais} and \cite{pil}, and their constructions (which appear slightly different on the surface) are known to give equivalent notions of overconvergent modular forms. We prove (Theorem \ref{pillonicomp}) that our definition is also equivalent to these previous constructions. As emphasized to us by a referee, we remark that our definition of the sheaf $\omega_{\kappa,w}^{\dg}$ is essentially as a line bundle on $\Xinfw$ with a descent datum to $\Xw$. However pro-\'etale descent of vector bundles is not effective in general (this follows for example from \cite[Example 8.1.6]{kl2}). In our case we show with relative ease that the resulting sheaf is indeed a line bundle.

\subsection{Modular curves vs. Shimura curves}

While we have written this introduction so far in the setting of modular curves, we have chosen to work with compact Shimura curves associated with an indefinite quaternion division algebra $B/ \mb{Q}$ split at (our fixed) $p$ in the body of the paper.  There are two reasons for this.  The first reason is that the local $p$-adic geometry of these Shimura curves is entirely analogous to the local $p$-adic geometry of modular curves, but the global geometry is simpler without the presence of boundary divisors, compactifications, and their attendant complications.  We believe that working in a boundaryless setting helps to clarify the point of view adopted in this paper. 

\medskip

Having said this, many of our ideas extend to the case of modular curves. In particular, all the definitions and results in \S \ref{sec:mfs} have exact analogues for classical modular curves, and Theorems \ref{1.1} and \ref{1.3} are true as stated. The techniques we use for Shimura curves work over the open modular curve, and one may extend over the boundary using ``soft" techniques (one does not need the advanced results of \cite[\S 2]{sch3}).

\medskip

The second reason is that our point of view on the overconvergent Eichler-Shimura map does \emph{not} immediately generalize to modular curves.  While we believe that the underlying philosophy should adapt, there are certain technical aspects of the construction (in particular Proposition \ref{prop:diam}) which seem difficult to adapt. As far as we can tell, this is a purely technical issue with the pro-\'etale site as defined in \cite{sch2}. In work in progress (\cite{dt}), Diao and Tan are developing a logarithmic version of the pro-\'etale site, and we believe that our constructions would work well in that setting.

\subsection{The overconvergent Eichler-Shimura map}

From now on we work with Shimura curves attached to an indefinite quaternion division algebra $B/\Qp$ split at $p$ and use the same notation that we previously used for modular curves.

\medskip

After Coleman's construction of Coleman families and the globalization of these to the Coleman-Mazur eigencurve (\cite{clm,cm}), different constructions of families of finite slope eigenforms and eigencurves for ${\rm GL}_{2/\mb{Q}}$ were given by Stevens (\cite{ste}, completed by Bella\"iche \cite{bel}) and Emerton (\cite{eme}). These constructions give the same eigencurve. Roughly speaking, each approach first constructs a $p$-adic Banach/Fr\'echet space (or many such spaces) of "overconvergent" objects interpolating modular forms/cohomology classes, and then creates a geometric object out of this (these) space(s). All spaces have actions of certain Hecke algebras and the fact that these give the same eigencurves amounts to saying that they contain the same finite slope systems of Hecke eigenvalues. However, not much is known about the direct relation between these spaces. One can rephrase the problem in the following way: each construction of the eigencurve remembers the spaces it came from in the form of a coherent sheaf on it, and one may ask if there are relations between these sheaves.

\medskip

In \cite{oes}, Andreatta, Iovita and Stevens study the relationship between overconvergent modular forms and overconvergent modular symbols. While Coleman's overconvergent modular forms $p$-adically interpolate modular forms, Stevens's overconvergent modular symbols interpolate classical modular symbols, i.e. classes in the singular cohomology groups $H^{1}(Y_{1}(N), \Sym^{k-2}\mc{H}^{1})$, where $\mc{H}^{1}$ is the first relative singular cohomogy of the universal elliptic curve. Classical Eichler-Shimura theory, which one may view as an elaboration of Hodge Theory for these particular varieties and coefficient systems, gives a Hecke-equivariant isomorphism
$$  H^{1}(Y_{1}(N), \Sym^{k-2}\mc{H}^{1})\otimes_{\mb{Z}}\mb{C} \cong M_{k} \oplus S_{k} $$
where $M_{k}$ is the space of weight $k$ and level $N$ modular forms and $S_{k}$ is its subspace of cusp forms. Faltings (\cite{fal}) constructed a $p$-adic Hodge-theoretic analogue of this isomorphism, replacing singular cohomology with \'etale cohomology. This construction was then adapted to the overconvergent context in \cite{oes}.

\medskip

Let us describe these ideas and our work in more detail. We refer to the main body of the paper for exact definitions. Recall (\cite{as,han1}) the \emph{overconvergent distribution modules} $\mbf{D}_{\kappa}^{s}$ where $\kappa$ is a character of $\Zp^{\times}$ as above. It may be interpreted as a local system on $X_{K_{0}(p)}(\mb{C})$. The singular cohomology $H^{1}(X_{K_{0}(p)}(\mb{C}),\mbf{D}_{\kappa}^{s})$ is the space of overconvergent modular symbols. Following \cite{han}, we construct a filtration on the integral distribution module $\mbf{D}_{\kappa}^{s,\circ}$ for which the corresponding topology is profinite. From this one gets a sheaf on the pro-\'etale site of $\X_{K_{0}(p)}$ whose cohomology is isomorphic to $H^{1}(X_{K_{0}(p)}(\mb{C}),\mbf{D}_{\kappa}^{s})$, but also carries a Galois action. To compare this to overconvergent modular forms of weight $\kappa$, one introduces a ``fattened" version $\oo\mbf{D}_{\kappa}^{s}$ of $\mbf{D}_{\kappa}^{s}$ which has the explicit description
$$ V \mapsto (\mbf{D}_{\kappa} \ctens \wh{\oo}_{\X_{K_{0}(p)}}(V_{\infty}))^{K_{0}(p)} $$
for $V\in \X_{K_{0}(p),\proet}$ quasi-compact and quasi-separated (qcqs), where $V_{\infty}:= V \times_{\X_{K_{0}(p)}}\Xinf $ and $\wh{\oo}_{\X_{K_{0}(p)}}$ is the completed structure sheaf on $\X_{K_{0}(p),\proet}$. After restricting to $\Xw$ it turns out to be easy to give a morphism to the completed version $\wh{\omega}_{\kappa,w}^{\dg}$: for $V\in \X_{w,\proet}$ qcqs there is a $K_{0}(p)$-equivariant morphism 
$$ \mbf{D}_{\kappa} \ctens \wh{\oo}_{\Xw}(V_{\infty}) \ra \wh{\oo}_{\Xw}(V_{\infty})\otimes_{\Qp}L $$
given on elementary tensors by
$$ \mu \otimes f \mapsto \mu(\kappa(1+\mf{z}x))f. $$
The formula is heavily inspired by a formula of Stevens for the comparison map between overconvergent distributions and polynomial distributions (whose cohomology computes classical modular symbols) which does not seem to be used a lot in the literature (see the paragraph before Definition \ref{def-int}). One then passes to $K_0(p)$-invariants to obtain the desired morphism of sheaves. This is our analogue of the maps denoted by $\delta_{\kappa}^{\vee}(w)$ in \cite{oes}. It induces a map on cohomology groups over $\Cp$ which gives the desired map of spaces using that  $H^{1}_{\proet}(\X_{w,\Cp},\wh{\omega}_{\kappa,w}^{\dg})\cong H^{0}(\X_{w,\Cp}, \omega_{\kappa,w}^{\dg}\otimes_{\oo_{\Xw}}\Omega^{1}_{\Xw})$. The strategy is the same as in \cite{oes}, except that they work with the so-called Faltings site instead of the pro-\'etale site. It is the presence of infinite level Shimura curves in the pro-\'etale sites of finite level Shimura curves that accounts for the clean explicit formulas we obtain: they provide the correct ``local coordinates" for the problem at hand.

\medskip

We analyze these maps by carrying out the above constructions in families of weights (as in \cite{oes}). To define Galois actions one needs to work with families parametrized by certain affine formal schemes instead of the more commonly used affinoid rigid spaces. Whereas the filtrations defined in \cite{oes} only work when the formal scheme is an open unit disc near the center of weight space, our filtrations are defined over arbitrary $\Spf R$ where $R$ is finite over $\Zp[[X_{1},..,X_{d}]]$ for some $d$. This enables us to glue the morphisms for different families of weights into a morphism of sheaves over the whole eigencurve.

\medskip

Denote by $\mc{C}$ the eigencurve and let $\mc{C}_{\Cp}$ be its base change to $\Cp$. It carries coherent sheaves $\mbf{V}$, resp. $\mc{M}^{\dg}$, coming from overconvergent modular symbols resp. forms. We denote by $\mbf{V}_{\Cp}$, resp. $\mc{M}^{\dg}_{\Cp}$, their base changes to $\Cp$, which may also be viewed as sheaves of $\oo_{\mc{C}}\ctens_{\Qp}\Cp$-modules on $\mc{C}$. In the latter point of view, one may think of $\Cp$ as a (rather primitive) period ring. After gluing, the overconvergent Eichler-Shimura map is a morphism $\mc{ES}\, :\, \mbf{V}_{\Cp} \ra \mc{M}^{\dg}_{\Cp}(-1)$ of sheaves.

\begin{theo}[Theorem \ref{thm:coker}, Theorem \ref{theo:main}]  Let $\mc{C}^{\sm}$ be the smooth locus of $\mc{C}$.
\begin{enumerate}
\item $\mbf{V}$ and $\mc{M}^{\dg}$ are locally free over $\mc{C}^{\sm}$, and the kernel $\mc{K}$ and image $\mc{I}$ of $\mc{ES}$ are locally projective sheaves of $\oo_{\mc{C}^{\sm}}\ctens_{\Qp}\Cp$-modules (or equivalently locally free sheaves on $\mc{C}^{\sm}_{\Cp})$. The support of $\mc{M}^{\dg}_{\Cp}(-1)/\mc{I}$ on $\mc{C}_{\Cp}$ is Zariski closed of dimension $0$.

\item Let $\epsilon_{\mc{C}^{\sm}}$ be the character of $G_{\Qp}$ defined by the composition
$$G_{\Qp} \overset{\epsilon}{\longrightarrow}  \Zp^{\times} \overset{\chi_{\mc{W}}}{\longrightarrow} \oo_{\mc{W}}^{\times} \longrightarrow  (\oo_{\mc{C}^{\sm}}^{\times}\ctens_{\Qp}\Cp)^{\times} $$
where $\epsilon$ is the $p$-adic cyclotomic character of $G_{\Qp}$ and $\chi_{\mc{W}}$ is the universal character of $\Zp^{\times}$. Then the semilinear action of $G_{\Qp}$ on the module $\mc{K}(\epsilon_{\mc{C}^{\sm}}^{-1})$ is trivial.

\item The exact sequence
$$ 0 \ra \mc{K} \ra \mbf{V}_{\Cp} \ra \mc{I} \ra 0 $$
is locally split. Zariski generically, the splitting may be taken to be equivariant with respect to both the Hecke- and $G_{\Qp}$-actions, and such a splitting is unique.

\end{enumerate}
\end{theo}

These are stronger analogues of results in \cite{oes}, where the authors prove the analogous results for modular curves in some small (unspecified) open neighbourhood of the set of non-critical classical points. 

\medskip

We believe that our perspective on overconvergent modular forms and the overconvergent Eichler-Shimura map should generalize to higher-dimensional Shimura varieties. In particular, it should be reasonably straightforward to adapt the methods of this paper to (the compact versions of) Hilbert modular varieties. 

\subsection{Notation, conventions and an outline of the paper} 

Throughout this text, we let $p$ denote a fixed prime.

\medskip

For the purposes of this paper, a \emph{small $\Zp$-algebra} is a ring $R$ which is reduced, $p$-torsion-free, and finite as a $\Zp[[X_{1},...,X_{d}]]$-algebra for some (unspecified) $d\geq 0$. Any such $R$ carries a canonical adic profinite topology, and is also complete for its $p$-adic topology. For convenience we will fix a choice of ideal $\mf{a}=\mf{a}_{R}$ defining the profinite topology (a ``canonical" example of such a choice is the Jacobson radical). All constructions made using this choice will be easily verified to not depend upon it.

\medskip
 
We will need various completed tensor product constructions, some of which are non-standard. We will need to take a form of completed tensor product between small $\Zp$-algebras and various Banach spaces or other $\Zp$-modules.  These will always be denoted by an undecorated completed tensor product $\ctens$. We explain our conventions for the unadorned $\ctens$ in Convention \ref{weighttens} and Definitions \ref{defi: ctens} and \ref{mixtens} below. Any adorned $\ctens_{A}$ is a standard one, with respect to the natural topology coming from $A$.

\medskip

We will use Huber's adic spaces as our language for non-archimedean analytic geometry in this paper. In particular, a ``rigid analytic variety" will refer to the associated adic space, and all open subsets and open covers are open subsets resp. open covers of the adic space (i.e. we drop the adjective ``admissible" used in rigid analytic geometry). The pro-\'etale site of \cite{sch2} is key to our constructions; we will freely use notation and terminology from that paper. For perfectoid spaces we use the language of \cite{kl} for simplicity (e.g. we speak of perfectoid spaces over $\Qp$), but any perfectoid space appearing is a perfectoid space is the sense of \cite{perf} (i.e. it lives over a perfectoid field).

\medskip

Let us finish the introduction by briefly outlining the contents of the paper. In \S \ref{sec:mfs} we give our new definitions of sheaves of overconvergent modular forms in families and prove their basic properties, including a comparison with the definitions of \cite{ais,pil}. In \S \ref{sec:3} we recall the basic definitions from the theory of overconvergent modular symbols and define the filtrations mentioned above. We make some technical adjustments when defining slope decompositions. In particular, we do not need the concept of a weak orthonormal basis used in \cite{oes}; all slope decompositions can be defined using standard orthonormal bases and formal operations. In \S \ref{sec:4} we define our overconvergent Eichler-Shimura maps, and \S \ref{sec:5} glues them over the eigencurve and proves the properties stated above.
 
\medskip

The paper concludes with an appendix, collecting various technical results and definitions that are needed in the main text; some of these results may be of independent interest. In \S \ref{sec:6.1} we define our non-standard completed tensor products and prove some basic properties. While we only need this $\ctens$ for small $\Zp$-algebras $R$, it turns out that the ring structure only serves to obfuscate the situation. Accordingly, we define $\ctens$ for a class of $\Zp$-module that we call \emph{profinite flat}. Throughout the text we will need to consider sheaves of rings like $\oo_{X}\ctens R$ and $\wh{\oo}_{X}\ctens R$ on $X$ where $X$ is a rigid space and $R$ is a small $\Zp$-algebra. We prove some technical facts about these sheaves of rings and their modules in \S \ref{sec:6.2}-\ref{sec:6.3}. Finally \S \ref{sec:6.4} discusses quotients of rigid spaces by finite group actions, proving a standard existence result that we were not able to locate in the literature.

\subsection*{Acknowledgements} This collaboration grew out of discussions at the conference in Lyon in June 2013 between P.C. and C.J, and then conversations in Jussieu in October 2013 between P.C. and D.H. We thank the organizers of the Lyon conference for the wonderful gathering. All of the authors benefited from shorter or longer stays at Columbia University and the University of Oxford. D.H. would also like to thank Johan de Jong, Michael Harris and Shrenik Shah for some helpful conversations, and C.J. would like to thank Hansheng Diao for conversations relating to this work and \cite{dt}. The authors would like to thank Judith Ludwig for pointing out a gap in our original proof of Proposition \ref{prop: qspace}, as well as the anonymous referees for their comments and corrections.

\medskip

P.C. was partially funded by EPSRC grant EP/L005190/1. D.H. received funding from the European Research Council under the European Community's Seventh Framework Programme (FP7/2007-2013) / ERC Grant agreement no. 290766 (AAMOT). C. J. was supported by EPSRC Grant EP/J009458/1 and NSF Grants 093207800 and DMS-1128155 during the work on this paper.

\section{Overconvergent modular forms}\label{sec:mfs}

\subsection{Weights and characters}
In section we recall some basic notions about weights. We may define weight space as the functor from complete affinoid $(\Zp,\Zp)$-algebras $(A,A^{+})$ to abelian groups given by
$$ (A,A^{+}) \mapsto \Hom_{cts}(\Zp^{\times},A^{\times}). $$
This functor is representable by $(\Zp[[\Zp^{\times}]],\Zp[[\Zp^{\times}]])$. The proof is well known. The key fact that makes the arguments work in this generality is that $A^{\circ}/A^{\circ\circ}$ is a reduced ring of characteristic $p$, where $A^{\circ \circ}$ is the set of topologically nilpotent elements in $A$. We define $\mf{W}=\Spf(\mathbb{Z}_p[[\mathbb{Z}_p^{\times}]])$ and let  $\mc{W}=\Spf(\mathbb{Z}_p[[\mathbb{Z}_p^{\times}]])^{\mathrm{rig}}$ be the associated rigid analytic weight space, with its universal character $\chi_{\mc{W}}:\mathbb{Z}_p^{\times} \to \mc{O}(\mc{W})^{\times}$. 

\medskip

We embed $\mb{Z}$ into $\mc{W}$ by sending $k$ to the character $\chi _k(z) = z^{k-2}$.

\begin{defi}\label{defi:smallw}
\begin{enumerate}
\item A small weight is a pair $\mc{U}=(R_{\mc{U}},\chi_{\mc{U}})$ where $R_{\mc{U}}$ is a small $\Zp$-algebra and $\chi_{\mc{U}}\, :\, \Zp^{\times} \ra R^{\times}_{\mc{U}}$ is a continuous character such that $\chi_{\mc{U}}(1+p)-1$ is topologically nilpotent in $R_{\mc{U}}$ with respect to the \emph{$p$-adic} topology.

\item An affinoid weight is a pair $\mc{U}=(S_{\mc{U}},\chi_{\mc{U}})$ where $S_{\mc{U}}$ is a reduced Tate algebra over $\Qp$ topologically of finite type and $\chi_{\mc{U}}\, :\, \Zp^{\times} \ra S^{\times}_{\mc{U}}$ is a continuous character.

\item A weight is a pair $\mc{U}=(A_{\mc{U}},\chi_{\mc{U}})$ which is either a small weight or an affinoid weight.

\end{enumerate}
\end{defi}

We shall sometimes abbreviate $A_{\mc{U}}$ by $A$ when $\mc{U}$ is clear from context. In either case, we can make $A_{\U}[\frac{1}{p}]$ into a uniform $\Qp$-Banach algebra by letting $A_{\U}^{\circ}$ be the unit ball and equipping it with the corresponding spectral norm. We will denote this norm by $|\cdot |_{\U}$. Note then that there exists a smallest integer $s \geq 0$ such that $|\chi_{\mc{U}}(1+p)-1|_{\U}<p^{-\frac{1}{p^{s}(p-1)}}$. We denote this $s$ by $s_{\mc{U}}$. When $\mc{U}$ is small, the existence uses that $\chi_{\mc{U}}(1+p)-1$ is \emph{$p$-adically} topologically nilpotent. We will also make the following convention:

\begin{conv}\label{weighttens} Let $\mc{U}$ be a weight and let $V$ be a Banach space over $\Qp$. We define $V \ctens A_{\mc{U}}$ as follows:
\begin{enumerate} 
\item If $\mc{U}$ is small, then $V \ctens R_{\mc{U}}$ is a \emph{mixed completed tensor product} in the sense of the appendix.

\item If $\mc{U}$ is affinoid, then $V \ctens S_{\mc{U}}:=V \ctens_{\Qp} S_{\mc{U}}$.

\end{enumerate}
\end{conv}

\begin{rema} Many (though not all) of the results in this paper involving a choice of some weight $\mc{U}$ make equally good sense whether $\mc{U}$ is small or affinoid.  In our proofs of these results, we typically give either a proof which works uniformly in both cases, or a proof in the case where $\mc{U}$ is small, which is usually more technically demanding.
\end{rema}

\medskip

When $\mc{U}=(R_{\mc{U}},\chi_{\mc{U}})$ is a small weight the universal property of weight space gives us a canonical morphism
$$ \Spf(R_{\mc{U}}) \ra \mf{W}, $$
which induces a morphism
$$ \Spf(R_{\mc{U}})^{\rig} \ra \mc{W}. $$
When $\mc{U}=(S_{\mc{U}},\chi_{\mc{U}})$ is an affinoid weight we get an induced morphism
$$ \Spa(S_{\mc{U}},S_{\mc{U}}^{\circ}) \ra \mc{W}. $$
We make the following definition:

\begin{defi}

\begin{enumerate}
\item A small weight $\U=(R_{\U},\chi_{\U})$ is said to be \emph{open} if $R_{\U}$ is normal and the induced morphism $ \Spf(R_{\U})^{\rig} \ra \mc{W} $ is an open immersion.

\item An affinoid weight $\U=(S_{\U},\chi_{\U})$ is said to be \emph{open} if the induced morphism $  \Spa(S_{\U},S_{\U}^{\circ}) \ra \mc{W} $ is an open immersion.

\item A weight $\U=(A_{\U},\chi_{\U})$ is said to be \emph{open} if it is either a small open weight or an affinoid open weight.

\end{enumerate}
\end{defi}

Note that if $\U$ is a small weight such that $\U^{rig} \ra \mc{W}$ is an open immersion, then the normalization of $\U$ is a small open weight.
\medskip

Let $B$ be any uniform $\mb{Q}_p$-Banach algebra. Let us say that a function $f:\Zp\to B$
is $s$-analytic for some nonnegative integer $s$ if, for any fixed $a\in \Zp$, there is
some $\varphi_{f,a}\in B\left\langle T\right\rangle $ such that $\varphi_{f,a}(x)=f(p^{s}x+a)$
for all $x\in\Zp$. In other words, $f$ can be expanded
in a convergent power series on any ball of radius $p^{-s}$. This
is naturally a Banach space which we denote by $\mc{C}^{s-\mathrm{an}}(\Zp,B)$.

\begin{theo}[Amice]\label{theo: amice} The polynomials $e_{j}^{s}(x)=\left\lfloor p^{-s}j\right\rfloor !\left(\begin{array}{c}
x\\
j\end{array}\right)$ form an orthonormal basis of $\mc{C}^{s-\mathrm{an}}(\Zp,B)$
for any uniform $\Qp$-Banach algebra $B$. Furthermore, we have $e_{j}^{s}(\Zp+p^{s} B^{\circ})\sub B^{\circ}$.
\end{theo}
\begin{proof}
This is well known, see e.g. \cite[Theorem 1.7.8]{col} and its proof.
\end{proof}

We can now prove that characters extend over a bigger domain.

\begin{prop}\label{prop-extend}
Let $\U=(A_{\U},\chi_{\U})$ be a weight and let $B$ be any uniform $\Qp$-Banach algebra. Then for any $s\in \mb{Q}_{>0}$ such that $s \geq s_{\U}$, $\chi_{\U}$ extends canonically to a character \[ \chi_{\U}: B_{s}^{\times} \ra (A_{\U}^{\circ}\ctens_{\Zp}B^\circ)^{\times} \subset (A_{\U} \ctens B)^{\times},\]where $B_s ^{\times} := \Zp^\times \cdot (1+p^{s+1}B^{\circ}) \sub (B^{\circ})^{\times}$ and $p^{s+1}B^{\circ}$ is shorthand for $\{b\in B^{\circ} \mid |b|\leq p^{-s-1} \}$ (where $|-|$ is the spectral norm on $B$).
\end{prop}
\begin{proof}
Without loss of generality we may assume that $s$ is an integer (e.g by replacing $s$ with $\left \lfloor s \right \rfloor$). We may decompose any $b\in B_{s}^{\times}$ uniquely as $b=\omega(b)\left\langle b \right\rangle$, with $\omega(b)\in \mu_{p-1}$ and $\left\langle b \right\rangle \in 1+p \Zp + p^{s+1}B^{\circ}$. We will show that for any $s\geq s_{\mc{U}}$ and $b\in B_{s}^{\times}$, the individual terms of the series
\begin{eqnarray*}
f(b) & = & \chi_{\U}(\omega(b))\sum_{j=0}^{\infty}\left(\chi_{\U}(1+p)-1\right)^{j}\left(\begin{array}{c}
\frac{\log\left\langle b\right\rangle }{\log(1+p)}\\
j\end{array}\right)\\
 & ``=" & \chi_{\U}(\omega(b))\cdot\chi(1+p)^{\frac{\log\left\langle b\right\rangle }{\log1+p}}\end{eqnarray*}
lie in $A_{\U}^{\circ}\otimes_{\Zp}B^\circ$ and tend to zero $p$-adically, so this series converges to an element of $A_{\U}^{\circ}\ctens_{\Zp}B^\circ$ and \emph{a fortiori} to an element of $A_{\U} \ctens B$.  We claim this series defines a canonical extension of $\chi_{\U}$.
\medskip

\noindent Using the well known formula for the $p$-adic valuation of factorials we see that $v_{p}(\left\lfloor p^{-s}j\right\rfloor !)\leq\frac{j}{p^{s}(p-1)}$.
In particular, writing 
$$ \chi_{\U}(\omega(b))\sum_{j=0}^{\infty}\left(\chi_{\U}(1+p)-1\right)^{j}\left(\begin{array}{c}
x\\
j\end{array}\right)=\chi_{\U}(\omega(b))\sum_{j=0}^{\infty}\frac{\left(\chi_{\U}(1+p)-1\right)^{j}}{\left\lfloor p^{-s}j\right\rfloor !}e_{j}^{s}(x), $$
our assumption on $s$ implies that $\left|\chi_{\U}(1+p)-1\right|_{\U}<p^{-\frac{1}{p^{s}(p-1)}}$, so we see that this series converges to an element of $A_{\U}^{\circ}\ctens_{\Zp}B^\circ$ for any $x \in B$ such that $e^{s}_{j}(x) \in B^{\circ}$ for all $j$. By Theorem \ref{theo: amice} it then suffices to verify that if $b\in B_{s}^{\times}$ then $x=\frac{\log \left\langle b \right\rangle }{\log1+p}\in \Zp + p^{s}B^{\circ}$. But we know that the function $\eta(b)=\frac{\log\left\langle b\right\rangle }{\log1+p}$
defines a homomorphism from $1+p \Zp + p^{s+1}B^{\circ}$ to $\Zp+p^{s}B^{\circ}$, so we are done. The character property follows by calculating directly from the definitions. Finally, to show that this character extends $\chi_{\U}$, note that for $b\in\mu_{p-1}\times(1+p)^{\mb{Z}_{\geq 0}}$,
$f(b)$ becomes a finite sum which equals $\chi_{\U}(b)$ by
the binomial theorem, so $f|_{\Zp^{\times}}=\chi_{\U}$ by
continuity.
\end{proof}

\subsection{Shimura curves}

We write $C$ for an algebraically closed field containing $\mb{Q}_p$, complete with respect to a valuation $v:C\to \mb{R}\cup \{+\infty \}$ with $v(p)=1$ (so $v$ is nontrivial), and we write $\mc{O}_{C}$ for the valuation subring. Fix an embedding $\Cp \sub C$. We fix a compatible set of $p^{n}$-th roots of unity in $\Cp$ and use this choice throughout to ignore Tate twists (over any $C$). We let $B$ denote an indefinite non-split quaternion algebra over $\mb{Q}$, with discriminant $d$ which we assume is \emph{prime to} $p$. We fix a maximal order $\oo_{B}$ of $B$ as well as an isomorphism $\oo_{B} \otimes_{\mb{Z}}\Zp \cong M_{2}(\Zp)$, and write $G$ for the algebraic group over $\mb{Z}$ whose functor of points is
$$ R \mapsto (\oo_{B} \otimes_{\mb{Z}}R)^{\times}, $$
where $R$ is any ring.  We fix once and for all a neat compact open subgroup $K^{p}\sub G(\wh{\mb{Z}}^{p})$ such that  $K^{p}=\prod_{\ell\neq p}K_{\ell}$ for compact open subgroups $K_{\ell}\sub \gl(\mb{Z}_{\ell})$ and (for simplicity) $\det(K^{p})=(\wh{\mb{Z}}^{p})^{\times}$. 

\medskip

Recall (e.g. from \cite[\S1]{buz1}) that a \emph{false elliptic curve} over a $\mb{Z}[\frac{1}{d}]$-scheme $S$ is a pair $(A/S,i)$ where $A$ is an abelian surface over $S$ and $i\, :\, \oo_{B} \inj \End_{S}(A)$ is an injective ring homomorphism. We refer to \cite{buz1} for more information and definitions regarding false elliptic curves, in particular the definition of level structures. Let $X_{\gl(\Zp)}$ be the moduli space of false elliptic curves with $K^{p}$-level structure as a scheme over $\Zp$. We denote by $\mc{X}_{\gl(\Zp)}$ the Tate analytification of its generic fibre, viewed as an adic space over $\Spa(\Qp,\Zp)$. For any compact open subgroup $K_{p}\sub \gl(\Zp)$ we use a subscript $-_{K_{p}}$ to denote the same objects with a $K_{p}$-level structure added. We will mostly use the standard compact open subgroups $K_{0}(p^{n})$ or $K(p^n)$, for $n\geq 1$. Since we will mostly work with the Shimura curves with $K_{0}(p)$-level structure, we make the following convention:

\begin{conv} We define $X:=X_{K_{0}(p)}$, $\mc{X}=\mc{X}_{K_{0}(p)}$, et cetera. A Shimura curve with no level specified has $K_{0}(p)$-level at $p$.
\end{conv}

\medskip

 The following striking theorem of Scholze is key to all constructions in this paper.

\begin{theo}[Scholze] \label{theo: scholze}There exist a perfectoid space $\mc{X}_{\infty}$ over $\Spa(\mb{Q}_p,\mb{Z}_p)$ such that
$$\mc{X}_{\infty} \sim \varprojlim _{n} \mc{X}_{K(p^n)}.$$
It carries an action of $\gl(\mb{Q}_{p})$ and there exists a $\gl(\mb{Q}_{p})$-equivariant morphism 
$$\pi_{\HT}\, :\, \Xinf \ra \mb{P}^1 $$
of adic spaces over $\Spa(\mb{Q}_p,\mb{Z}_p)$. Let $\Pro = V_{1} \cup V_{2}$ denote the standard affinoid cover. Then $\mc{V}_{1}=\pi_{\HT}^{-1}(V_{1})$ and $\mc{V}_{2}=\pi_{\HT}^{-1}(V_{2})$ are both affinoid perfectoid, and there exists an $N$ and affinoid opens $S_{1},S_{2}\sub \mc{X}_{K(p^{N})}$ such that $\mc{V}_{i}$ is the preimage of $S_{i}$. Moreover we have $\omega = \pi_{\HT}^{\ast}\mc{O}(1)$ on $\Xinf$, where $\omega$ is obtained by pulling back the usual $\omega$ (defined below) from any finite level $\mc{X}_{K(p^n)}$.
\end{theo}
 
\medskip

A few remarks are in order. For the definition of $\sim$ we refer to \cite[Definition 2.4.1]{sw}. This theorem is essentially a special case of \cite[Theorem IV.1.1]{sch3} except for the difference in base field and the target of $\pi_{\HT}$; there one obtains a perfectoid space over some algebraically closed complete $C/\Qp$ and $\pi_{\HT}$ takes values in a larger (partial) flag variety. The version here is easily deduced in the same way; we now sketch the argument. The tower $(\X_{K(p^{n})})_{n}$ embeds into the tower of Siegel threefolds (over $\Qp$), and the same argument as in the proof of \cite[Theorem IV.1.1]{sch3} gives the existence of $\Xinf$ and a map $\pi_{\HT}$ which takes values in the partial flag variety $\mc{F}l$ of ${\rm GSp}_{4}$ with respect to the Siegel parabolic. Using the $M_{2}(\Zp)$-action (see below) one sees that it takes values in $\Pro \sub \mc{F}l$. Since the first version of this paper, such results have appeared in the case of general Hodge type Shimura varieties; see \cite[Theorems 2.1.2 and 2.1.3]{cs}. Finally, one easily sees that the standard affinoid opens of $\Pro$ come by pullback from standard affinoid opens of $\mb{P}^{5}$ via the embeddings $\mb{P}^{1} \sub \mc{F}l \sub \mb{P}^{5}$, where $\mc{F}l \sub \mb{P}^{5}$ is the Pl\"ucker embedding.

\medskip

Let us now discuss some standard constructions and define the sheaf $\omega$ mentioned in Theorem \ref{theo: scholze}. For any false elliptic curve $A$ over some $\Zp$-scheme $S$ the $p$-divisible group $A[p^{\infty}]$ carries an action $\oo_{B} \otimes_{\mb{Z}}\Zp\cong M_{2}(\Zp)$. Put $\mc{G}_{A}=eA[p^{\infty}]$, where $e\in M_{2}(\Zp)$ is an idempotent that we will fix throughout the text (take e.g. $\left( \begin{smallmatrix} 1 & 0 \\ 0 & 0 \end{smallmatrix} \right)$). This is a $p$-divisible group over $S$ of height $2$ and we have $A[p^{\infty}] \cong \mc{G}_{A}^{\oplus 2}$ functorially; we will fix this isomorphism. For all purposes $\mc{G}_{A}$ behaves exactly like the $p$-divisible group of an elliptic curve and we may use it to define ordinarity, supersingularity, level structures et cetera. We will often just write $\mc{G}$ instead of $\mc{G}_{A}$ if the false elliptic curve $A$ is clear from the context. The line bundle $\omega$ is the dual of $e(\Lie(A^{\univ}/X))$, where $A^{\univ}$ is the universal false elliptic curve. We will also write $\mc{G}^{\univ}=\mc{G}_{A^{\univ}}$. The same definitions and conventions apply to the adic versions, and to other level structures. 

\medskip

Next, we specify the right action of $\gl(\mb{Q}_{p})$ on $(C,\mc{O}_{C})$-points on both sides of the Hodge-Tate period map $\Xinf \ra \Pro$. First we consider $\Pro$: $g\in \glq$ acts from the left on $C^2$ (viewed as column vectors) and a line $L\sub C^2$ is sent by $g$ to $g^{\vee}(L)$, where $g\mapsto g^{\vee}$ is the involution

$$ g=\left( \begin{matrix} a & b \\ c & d \end{matrix} \right) \mapsto g^{\vee}=\det(g)g^{-1}=\left( \begin{matrix} d & -b \\ -c & a \end{matrix} \right). $$

This defines a right action. A $(C,\OC)$-point of $\Xinf$ consists of a false elliptic curve $A/C$ and an isomorphism $\alpha\, : \, \Zp^2 \ra T_{p}\mc{G}$ (and the $K^p$-level structure which we ignore). Let $g\in \glq$ and fix $n\in \mb{Z}$ such that $p^{n}g\in M_{2}(\Zp)$ but $p^{n-1}g\notin M_{2}(\Zp)$. For $m\in \mb{Z}_{\geq 0}$ sufficiently large the kernel of $p^{n}g^{\vee}$ modulo $p^m$ stabilizes and we denote the corresponding subgroup of $\mc{G}[p^m]$ under $\alpha$ by $H$. We define $(A,\alpha).g$ to be $(A/H^{\oplus 2},\beta)$, where $\beta$ is defined as the composition
$$  \Zp^2 \overset{p^{n}g}{\longrightarrow} \Qp^2 \overset{\alpha}{\longrightarrow} V_{p}\mc{G} \overset{(f^{\vee})^{-1}_{\ast}}{\longrightarrow} V_{p}(\mc{G}/H). $$
Here $H^{\oplus 2}$ is viewed as a subgroup scheme of $A[p^{\infty}]$ via the functorial isomorphism $A[p^{\infty}]\cong \mc{G}^{\oplus 2}$, $V_{p}(-)$ denotes the rational Tate module and $(f^{\vee})_{\ast} :\, V_{p}(\mc{G}/H) \ra V_{p}\mc{G}$ is the map induced from the dual of the natural isogeny $f\, :\, \mc{G} \ra \mc{G}/H$ (note that $\beta$ is isomorphism onto $T_{p}(\mc{G}/H)$). In particular, if $g\in \gl(\mb{Z}_p)$, then $(A,\alpha).g=(A,\alpha\cdot g)$ where $(\alpha \cdot g)(e_1)=a\alpha(e_1)+c\alpha(e_2)$, $(\alpha \cdot g)(e_2)=b\alpha(e_1)+d\alpha(e_2)$. Here and everywhere else in the text $e_1$ and $e_2$ are the standard basis vectors $\left( \begin{smallmatrix} 1 \\ 0 \end{smallmatrix} \right)$ and $\left( \begin{smallmatrix} 0 \\ 1 \end{smallmatrix} \right)$ of $\Zp^2$.

\medskip

\subsection{\emph{w}-ordinary false elliptic curves}

Let $H$ be a finite flat group scheme over $\OC$ killed by $p^n$. We let $\omega_{H}$ denote the dual of $\Lie(H)$. It is a torsion $\OC$-module and hence isomorphic to $\bigoplus_{i} \OC/a_{i}\OC$ for some finite set of $a_{i}\in \OC$. The degree $\deg(H)$ of $H$ is defined to be $\sum_{i}v(a_{i})$. The Hodge-Tate map $\HT_{H}$ is the morphism of \emph{fppf} abelian sheaves $H \ra \omega_{H^{\vee}}$ over $\OC$ defined on points by 
$$ f\in H=(H^{\vee})^{\vee} \mapsto f^{\ast}(dt/t)\in \omega_{H^{\vee}} $$
where we view $f$ as a morphism $f\, :\, H^{\vee} \ra \mu_{p^n}$, $dt/t\in \omega_{\mu_{p^n}}$ is the invariant differential and $-^{\vee}$ denotes the Cartier dual. We will often abuse notation and use $\HT_{H}$ for the map on $\OC$-points, and there one may identify the $\OC$-points of $H$ with the $C$-points of its generic fibre. 

\medskip

Now let $G$ be a $p$-divisible group over $\OC$. Taking the inverse limit over the Hodge-Tate maps for the $G[p^n]$ we obtain a morphism $\HT_{G}\, :\, T_{p}G \ra \omega_{G^{\vee}}$, which we will often linearize by tensoring the source with $\OC$. Taking this morphism for $G^{\vee}$ and dualizing it we obtain a morphism $\Lie(G) \ra T_{p}G \otimes_{\Zp}\OC$. Putting these morphisms together we get a sequence
$$ 0 \ra \Lie(G) \ra T_{p}G \otimes _{\Zp}\OC \ra \omega_{G^{\vee}} \ra 0 $$
which is in fact a complex with cohomology groups killed by $p^{1/(p-1)}$ (\cite[Th\'eor\`eme II.1.1]{fgl}).

\medskip

Let $A/C$ be a false elliptic curve. Then $A$ has good reduction and we will denote its unique model over $\OC$  by $\mc{A}$. We have the Hodge-Tate sequence of $\mc{G}_{\mc{A}}[p^{\infty}]$:
$$ 0 \ra \Lie(\mc{G}_{\mc{A}}) \ra T_{p}\mc{G} \otimes_{\Zp} \OC \ra \omega_{\mc{G}_{\mc{A}}^{\vee}} \ra 0. $$
Here we have dropped the subscript $-_{\mc{A}}$ in the notation of the Tate module for simplicity; this should not cause any confusion. We will write $\HT_{A}$ for $\HT_{\mc{G}_{\mc{A}}[p^\infty]}$. The image and kernel of $\HT_{A}$ are free $\OC$-modules of rank $1$ that we will denote by $F_{A}$ and $F^{1}_{A}$ respectively. Note that $p^{1/(p-1)}\omega_{\mc{G}_{\mc{A}}^{\vee}} \sub F_{A} \sub \omega_{\mc{G}_{\mc{A}}^{\vee}}$.

\medskip

Recall that $e_{1}$ and $e_{2}$ are the standard basis vectors of $\Zp^2$ and let $w$ be a positive rational number.

\begin{defi} Let $A/C$ be a false elliptic curve with model $\mc{A}/\OC$. Let $w\in \mb{Q}_{>0}$.
\begin{enumerate}
\item Let $\alpha$ be a trivialization of $T_{p}\mc{G}$. We say that $\alpha$ is $w$-ordinary if $\HT_{A}(\alpha(e_{1})) \in p^{w}F_{A}$.
\item $A$ is called $w$-ordinary if there is a $w$-ordinary trivialization of $T_{p}\mc{G}$.
\end{enumerate}
\end{defi}

\medskip

Note that if $A$ is $w$-ordinary, then it also $w^{\prime}$-ordinary for all $w^{\prime}<w$. Note also that $A$ is ordinary (in the classical sense) if and only if it is $\infty$-ordinary (i.e. $A$-ordinary for all $w>0$).

\begin{defi}
Let $A/C$ be a $w$-ordinary false elliptic curve and assume that $n\in \mb{Z}_{\geq 1}$ is such that $n<w+1$. Then the kernel of the morphism $\mc{G}[p^n](C) \ra F_{A}/p^{{\rm min}(n,w)}F_{A}$ induced by $\HT_{A}$ is an \'etale subgroup scheme $H_{n}$ of $\mc{G}_{A}[p^n]$ isomorphic to $\mb{Z}/p^{n}\mb{Z}$ which we call the pseudocanonical subgroup of level $n$. 
\end{defi}

\medskip

We will use the notation $H_{n}$ to denote the pseudocanonical subgroup of level $n$ (when it exists) whenever the false elliptic curve $A$ is clear from the context. When there are multiple false elliptic curves in action we will use the notation $H_{n,A}$. Since $H_{n}$ is naturally equipped with an inclusion into $\mc{G}_{A}[p^{n}]$ we may take its schematic closure inside $\mc{G}_{\mc{A}}[p^{n}]$. This is a finite flat group scheme of rank $p^n$ over $\OC$ with generic fibre $H_{n}$ and we will abuse notation and denote it by $H_{n}$ as well.

\medskip

When $n=1$ we will refer to $H_{1}$ simply as the pseudocanonical subgroup and drop "of level $1$". Note that if $\alpha\, :\, T_{p}\mc{G} \ra \Zp^2$ is a $w$-ordinary trivialization with $n-1 <w \leq n$ then $(\alpha^{-1}\, {\rm mod}\, p^{n})|_{\mb{Z}/p^{n}\mb{Z} \oplus 0}$ trivializes the pseudocanonical subgroup. We record a simple lemma:

\medskip

\begin{lemm}
Let $A/C$ be a false elliptic curve and let $\alpha$ be a $w$-ordinary trivialization of $T_{p}\mc{G}$. Assume that $w>n\in \mb{Z}_{\geq 1}$ and let $m\leq n$ be a positive integer. Then $A/H_{m,A}^{\oplus 2}$ is $(w-m)$-ordinary, and for any $m^{\prime}\in \mb{Z}$ with $m< m^{\prime} \leq n$, $H_{m^{\prime}-m,A/H_{m}^{\oplus 2}}=H_{m^{\prime},A}/H_{m,A}$. 
\end{lemm}

\begin{proof}
Let $g\in \glq$ denote the matrix $ \left( \begin{smallmatrix} 1 & 0 \\ 0 & p^{m} \end{smallmatrix} \right)$. Then $(A,\alpha).g=(A/H_{m,A}^{\oplus 2},\beta)$ where $\beta$ is defined by this equality. Let $f$ denote the natural isogeny $\mc{G} \ra \mc{G}/H_{m}$. From the definitions we get a commutative diagram 
\[
\xymatrix{\Zp^2\ar[d]^{g^{\vee}}\ar[r]^{\alpha} & T_{p}\mc{G}\ar[d]^{f_{\ast}}\ar[r]^{\HT_{A}} & F_{A}\ar[d]^{(f^{\vee})^{\ast}}\\
\Zp^2\ar[r]^{\beta} &T_{p}(\mc{G}/H_{m})\ar[r]^{\quad \HT_{A/H_{m}^{\oplus 2}}} & F_{A/H_{m}^{\oplus 2}}}
\] 
and direct computation gives that $p^{m}\HT_{A/H_{m}^{\oplus 2}}(\beta(e_{1}))=(f^{\vee})^{\ast}\HT_{A}(\alpha(e_{1}))$. Since $\HT_{A}(\alpha(e_{1})) \in p^{w}F_{A}$ we see that $\HT_{A/H_{m}^{\oplus 2}}(\beta(e_{1})) \in p^{w-m}F_{A/H_{m}^{\oplus 2}}$ which proves the first assertion. For the second assertion, observe that by definition $H_{m^{\prime}-m,A/H_{m}^{\oplus 2}}$  and $H_{m^{\prime},A}/H_{m,A}$ are generated by $\beta(e_{1}) \, {\rm mod}\, p^{m^{\prime}-m}$ and $f(\alpha(e_{1})) \, {\rm mod} \, p^{m}$ respectively, and that these are equal. 
\end{proof}

\medskip

\begin{rema}
The commutativity of the diagram in the proof above is also what essentially proves the $\glq$-equivariance of the Hodge-Tate period map $\pi_{\HT}$, and the first assertion may be viewed more transparently as a direct consequence of this equivariance for the element $g$. Note also that the second assertion of the Lemma mirrors properties of the usual canonical subgroups of higher level.
\end{rema}

\medskip

Next we recall some calculations from Oort-Tate theory which are recorded in \cite[\S 6.5 Lemme 9]{far} (we thank an anonymous referee for pointing out this reference). For the last statement, see  Proposition 1.2.8 of \cite{kas} (for further reference see Remark 1.2.7 of \emph{loc.cit} and \S 3 of \cite{buz}; note that these references treat elliptic curves but the results carry over \emph{verbatim}).

\medskip

\begin{lemm}
Let $H$ be a finite flat group scheme over $\OC$ of degree $p$. Then $H$ is isomorphic to $\Spec(\OC[Y]/(Y^{p}-aY))$ for some $a\in \OC$ and determined up to isomorphism by $v(a)$, and the following holds:
\begin{enumerate}
\item $\omega_{H}=(\OC/a\OC).dY$ and hence $\deg(H)=v(a)$.
\item The image of the (linearized) Hodge-Tate map $\HT_{H^{\vee}}\, :\, H^{\vee}(C)\otimes \OC \ra  \omega_{H}$ is equal to $(c\OC/a\OC).dY$, where $v(c)=(1-v(a))/(p-1)$. 
\end{enumerate}
Moreover, if $A/C$ is a false elliptic curve such that $H\sub \mc{G}_{\mc{A}}[p]$ and $\deg(H)>1/(p+1)$, then $H$ is the canonical subgroup of $\mc{G}_{\mc{A}}$.
\end{lemm}

\medskip

We will use these properties freely in this section. Using this we can now show that the pseudocanonical subgroup coincides with the canonical subgroup for sufficiently large $w$ (as a qualitative statement this is implicit in \cite{sch3}, cf. Lemma III.3.8).

\medskip

\begin{lemm}
Let $A/C$ be a $w$-ordinary false elliptic curve and assume that $p/(p^{2}-1)<w\leq 1$. Then $H_{1}$ is the canonical subgroup of $\mc{G}_{\mc{A}}$.
\end{lemm}

\begin{proof}
Consider the commutative diagram 
\[
\xymatrix{0  \ar[r]& H_{1}(C)\ar[d]_{\HT_{H_{1}}}\ar[r] & \mc{G}_{\mc{A}}[p](C)\ar[d]^{\HT_{\mc{G}_{\mc{A}}[p]}} \\
0 \ar[r] & \omega_{H_{1}^{\vee}} \ar[r] & \omega_{\mc{G}_{\mc{A}}[p]^{\vee}}}
\] 
with exact rows. We have $\omega_{\mc{G}_{\mc{A}}[p]^{\vee}}=\omega_{\mc{G}_{\mc{A}}^{\vee}}/p\omega_{\mc{G}_{\mc{A}}^{\vee}}$. Note that $H_{1}$ is an Oort-Tate group scheme and hence isomorphic to $\Spec(\OC[Y]/(Y^{p}-aY))$ for some $a\in \OC$ with $v(a)=\deg(H_{1})$ and $H_{1}^{\vee}$ is isomorphic to $\Spec(\OC[Y]/(Y^{p}-bY))$ with $ab=p$. Fix a generator $s\in H_{1}(C)$. By choosing generators the inclusion $\omega_{H_{1}^{\vee}} \ra \omega _{\mc{G}_{\mc{A}}[p]^{\vee}}$ may be written as $\OC/b\OC \ra \OC/p\OC$ where the map is multiplication by $a$, and $\HT_{H_{1}^{\vee}}(s)$ has valuation $v(a)/(p-1)$. Since $A$ is $w$-ordinary we have $a\HT_{H_{1}^{\vee}}(s)=\HT_{\mc{A}[p]}(s)\in p^{w}\omega_{\mc{A}[p]^{\vee}}$ and hence $pv(a)/(p-1) \geq w$, i.e. $\deg(H_{1}) \geq (p-1)w/p$. By our assumption on $w$ we deduce $\deg(H_{1})>1/(p+1)$ and hence that $H_{1}$ is the canonical subgroup. 
\end{proof}

\begin{rema}
As emphasized by a referee, one may also bound the Hodge height of $\mc{G}_{A}$, at least under stronger assumptions on $w$ (recall that the Hodge height is the valuation of the Hasse invariant, truncated by $1$). For example, using \cite[Th\'eor\`eme II.1.1]{fgl} and \cite[Proposition 3.1.2]{aip}, one sees that if $\tfrac{1}{2}+\tfrac{1}{p-1}<w\leq 1$ and $p\geq 5$, then the Hodge height of $\mc{G}_{A}$ is $\leq 1-w+\tfrac{1}{p-1}$.
\end{rema}

We will now, somewhat overdue, discuss the interpretation of $w$-ordinarity in terms of the Hodge-Tate period map. Define a coordinate $z$ on $\Pro$ by letting $z$ correspond to the line spanned by $\left( \begin{smallmatrix} x \\ y \end{smallmatrix} \right)$ with $z=-y/x$ and for $w\in\mb{Q}_{>0}$ define $U_{w}\sub \Pro$ to be the locus $|z|\leq |p^{w}|$. If $w\leq n$, then this locus is  $K_{0}(p^{n})$-stable. It is a rational subset of $V_{1}$. Now define 
$$\Uinfw=\pi_{\HT}^{-1}(U _{w}).$$ 
This is a $K_{0}(p^{n})$-stable open subspace of $\Xinf$. In fact, $\Uinfw$ is affinoid perfectoid since $U_{w}$ is a rational subset of $V_{1}\sub \Pro$ and therefore $\Uinfw$ is a rational subset of $\mc{V}_{1}$ since $\pi_{\HT}$ is adic. The conclusion follows from \cite[Theorem 6.3(ii)]{perf}. Note that, directly from the definition, the $(C,\OC)$-points of $\Uinfw$ are exactly the pairs $(A,\alpha)$ for which $\alpha$ is $w$-ordinary.

\medskip

In this article we will want to vary $w$ in order to capture all weights when defining overconvergent modular forms. A minor disadvantage of the loci $\Uinfw$ is that they are stable under different compact open subgroups as $w$ varies. We will instead define related loci that have the advantage that they all stable under the action of $K_{0}(p)$ (independent of $w$). 

\begin{defi}
Let $w\in \mb{Q}_{>0}$ and let $A$ be a $w$-ordinary false elliptic curve over $C$. Let $\alpha$ be a trivialisation of $T_{p}\mc{G}$. We say that $\alpha$ is \emph{strict} if $(\alpha \mod\, p)|_{\mb{Z}/p\mb{Z} \oplus 0}$ trivializes the pseudocanonical subgroup.

\end{defi}

Note that if $\alpha$ is $w$-ordinary then it is strict. Given $(A,\alpha)$ with $A$ $w$-ordinary and $\alpha$ strict, the orbit of $(A,\alpha)$ in $\Xinf$ under $K_{0}(p)$ consists exactly of the pairs $(A,\beta)$ for which $\beta$ is strict. We define $\Pro_{w}\sub V_{1}$ to be the rational subsets whose $(C,\OC)$-points correspond to

$$ \{z\in \OC \mid \exists s\in p\Zp\, :\, |z-s|\leq|p^{w}| \}. $$ 

\smallskip

One checks directly that if $\gamma\in K_{0}(p)$ and $|z|,|s|<1$, then $|\gamma^{\vee}z-\gamma^{\vee}s|=|z-s|$. Thus $\Pro_{w}$ is stable under $K_{0}(p)$ and moreover $\Pro_{w}=U_{w}.K_{0}(p)$. We then define 
$$\Xinfw=\pi_{\HT}^{-1}(\Pro_{w}).$$ 
By the discussion above these are $K_{0}(p)$-stable and a $(C,\OC)$-point $(A,\alpha)$ of $\Xinf$ is in $\Xinfw$ if and only if $A$ is $w$-ordinary and $\alpha$ is strict.

\begin{theo}\label{th1} Let $w\in \mb{Q}_{>0}$.
\begin{enumerate}
\item There is a unique affinoid open rigid subspace $\Xw \sub \X$ whose $(C,\OC)$-points are exactly the pairs $(A,H)$ with $A$ a $w$-ordinary false elliptic curves (with $K^{p}$-level structure) and $H$ its pseudocanonical subgroup.

\item Let $q\, :\, \Xinf \ra \X$ denote the projection map. Then $\Xinfw = q^{-1}(\mc{X}_{w})$.

\item The sets  $(\Xw)_{w}$ form a cofinal set of open neighbourhoods of the closure of the ordinary-multiplicative locus in $\X$.
\end{enumerate}
\end{theo}
\begin{proof}
Write, for any level $K_{p}\sub \gl(\Zp)$, $q_{K_{p}}$ for the projection map $\Xinf \ra \X$. We prove (1) and (2) together. Indeed, we define $\Xw := q(\Xinfw)$. Since $q$ is pro-\'etale, $\Xw$ is open and quasicompact (\cite{sch2}, Lemma 3.10(iv)) and the characterization of the $(C,\OC)$-points of $\Xw$ follows directly from the characterization of the $(C,\OC)$-points of $\Xinfw$. Note that uniqueness is clear since quasicompact open subsets of quasiseparated rigid analytic varieties (viewed as adic spaces) are determined by their classical points (cf. \cite[Theorem 2.21]{perf}). This also finishes the proof (1), except that we need to show that $\Xw$ is affinoid. 

\medskip

\noindent By the definition of the inverse limit topology we may find $m$ and a quasicompact open subset $W\sub \mc{X}_{K(p^m)}$ such that $q_{K(p^m)}^{-1}(W)=\Xinfw$. We claim that $W=q_{K(p^{m})}(\Xinfw)$. Indeed, both sets are determined by their $(C,\OC)$-points and $q_{K(p^m)}$ is surjective on $(C,\OC)$-points. Thus it remains to see that if $r$ is the natural map $\mc{X}_{K(p^m)} \ra \X$, then $r^{-1}(\Xw)=W$, which we can check on $(C,\OC)$-points. This follows from the surjectivity of $q_{K_{p}}$ on $(C,\OC)$-points for arbitrary $K_{p}$ and the fact that any trivialization that maps down to the pseudocanonical subgroup is automatically strict (in the sense that one/any lift to a trivialization of the whole Tate module is strict). This proves (2). Finally we finish the proof of (1) by showing that $\Xw$ is affinoid. To see this we use $\Xinfw = q^{-1}(\Xw)$. Since $\Xinfw$ is a rational subset of $\mc{V}_{1}$ we may find some large $N$ and an affinoid $S\sub \mc{X}_{K(p^{N})}$ such that $\Xinfw=q_{K(p^{N})}^{-1}(S)$ (this follows from Theorem \ref{theo: scholze} and the fact that rational subsets come from finite level). Then $\Xw$ is the quotient of $S$ by the finite group $K_{0}(p)/K(p^{N})$ and hence affinoid by Corollary \ref{coro: quotients}.

\medskip

\noindent We now prove (3). As a qualitative result it follows by general topology arguments using the constructible topology, cf. \cite[Lemma III.3.8]{sch3}. Here we deduce a quantitative version using the Lubin-Katz theory of the canonical subgroup. For a statement of the results we need phrased in terms of the degree function see Proposition 1.2.8 of \cite{kas} (as before this reference treats elliptic curves but the results carry over \emph{verbatim}). Without loss of generality we restrict our attention to $w=1,2,3,...$. On $\mc{X}_{1}$, $H_{1}$ is the canonical subgroup and $\deg(H_{1})\geq (p-1)/p$. Let $h= \left( \begin{smallmatrix} p & 0 \\ 0 & 1 \end{smallmatrix} \right) \in \glq$. Then $h(\Xinfw)=\mc{X}_{\infty,w+1}$ for all $w$. Moreover, if $(A,\alpha)\in \mc{X}_{K(p^{\infty}),1}(C,\OC)$ and $(A,\alpha).h=(A/H^{\oplus 2},\beta)$, then $H$ is anticanonical and hence of degree $(1-\deg(H_{1}))/p$. Hence $\deg(\mc{G}[p]/H)=1-(1-\deg(H_{1}))/p>1/2$ and $\mc{G}[p]/H$ is the canonical subgroup. Write $\delta_{n}$ for the degree of the canonical subgroup of $(h^n).A$. Then we get the recurrence relation $\delta_{n+1}=(p-1)/p+\delta_{n}/p$ and from $\delta_{1}\geq (p-1)/p$ we deduce that
$$ \delta_{n} \geq \delta_{n}^{\prime}:= \frac{p-1}{p} \sum_{i=0}^{n-1}\frac{1}{p^{i}}=1-\frac{1}{p^{n}}. $$
We have therefore proved that if $(A,H)\in \mc{X}_{n}(C,\OC)$ then $deg(H)\geq \delta^{\prime}_{n}$. As $\delta^{\prime}_{n} \ra 1 $ as $n\ra \infty$, the result follows since the loci $\{ (A,H) \mid deg(H) \geq r \}$ for $r \in (0,1)$ form a cofinal set of open neighbourhoods of the closure of the ordinary-multiplicative locus (which is the locus $\{ (A,H) \mid deg(H)=1 \}$) . 
\end{proof}

\medskip

\subsection{The fundamental period and a non-vanishing section}\label{sec-period} Recall the coordinate $z$ on $\Pro$ defined earlier. The action of $g = \left( \begin{smallmatrix} a & b \\ c & d \end{smallmatrix} \right) \in \gl(\Qp)$ is

$$ z.g=g^{\vee}z = \frac{az+c}{bz+d}.$$

\smallskip

Below, whenever we have a matrix $g \in \gl$ we will use $a,b,c,d$ as above to denote its entries. Note that $z$ defines a function in $H^{0}(V_{1},\mc{O}_{\Pro}^{+})$ and by composing with $\pi_{\HT}$ we obtain a function $\mf{z}\in H^{0}(\mc{V}_{1},\mc{O}_{\Xinf}^{+})$ which we will call the \textbf{fundamental period}. We will use the same notation to denote its restriction to $\Xinfw$ for any $w>0$; by definition $\mf{z}\in H^{0}(\Xinfw, p\Zp+p^{w}\oo_{\Xinf}^{+})$. Note that if $\gamma = \left( \begin{smallmatrix} a & b \\ c & d \end{smallmatrix} \right) \in K_{0}(p)$ then 
$$ \gamma^{\ast}\mf{z} = \frac{a\mf{z}+c}{b\mf{z}+d} $$
as functions on $\Xinfw$. Next we wish to trivialize $\mc{O}(1)$ over $S=\{z\neq \infty\}\sub \Pro$ by defining a non-vanishing section. Everything we write in this section is standard but we repeat it since we will need some explicit formulas. The section we are after is algebraic so we will momentarily work in the realm of algebraic geometry. Explicitly it is given by the following formula: The line bundle $\mc{O}(1)$ has a geometric incarnation with total space $(\gl \times \mb{A}^{1})/B$ where $B$ denotes the lower triangular Borel in $\gl$ and acts on $\mb{A}^{1}$ by $g.x = d^{-1}x$ (this equivariant structure may differ from the canonical one, but it is the one corresponding to the Hodge-Tate sequence on $\Xinf$). Global sections correspond to functions $ f\, :\, \gl \ra \mb{A}^{1}$ satisfying $f(gh)=d(h)^{-1}f(g)$ for $g\in \gl$, $h \in B$. There is a morphism of equivariant vector bundles
$$ (\gl /B)\times \mb{A}^{2} \ra (\gl \times \mb{A}^{1})/B, $$ 
given by
$$ \left(g, \left( \begin{matrix} x \\ y \end{matrix} \right) \right) \mapsto \left(g, dx-by \right). $$
Here the left hand side has the $\gl$-action $g.(h,v)=(gh,(g^{\vee})^{-1}v)$. Then the global section of $(\gl/B)\times \mb{A}^{2}$ given by the constant function $g \mapsto e_{2}$ maps to global section of $\mc{O}(1)$ given by the function  
$$ s\, : \, \gl \ra \mb{A}^{1}; $$
$$ s(g)=-b(g). $$
Note that $s=0$ if and only $g\in B$ so $s$ is invertible on $S$. Let us now return to the rigid analytic world and let $W\sub S$ be an open subset such $h^{\vee}(W) \sub W$ for some $h\in \gl(C)$. Then by direct calculation 
$$ (sh^{\vee})(g)= d(g)b(h)-b(g)d(h). $$
Thus
$$ \frac{s\circ h^{\vee}}{s}(g) = \frac{d(h)b(g)-b(h)d(g)}{b(g)}. $$
This is a function on $W$. Taking $g=g_{z}= \left( \begin{smallmatrix} 0 & -1 \\ 1 & z \end{smallmatrix} \right)$ we see that 
$$ \frac{s\circ h^{\vee}}{s}(z) = b(h)z+d(h). $$
We may then pull $s$ back via $\pi_{\HT}$ to get (compatible) non-vanishing sections $\mf{s}\in H^{0}(\Xinfw, \omega)$ for all $w$ satisfying
$$ \frac{\gamma^{\ast}\mf{s}}{\mf{s}} = b\mf{z}+d $$
for all $\gamma \in K_{0}(p)$. We remark that $\mf{s}$ is one of the "fake" Hasse invariants constructed in \cite{sch3}.

\medskip

\subsection{A perfectoid definition of overconvergent modular forms}

We will give definitions of sheaves of overconvergent modular forms with prescribed small or affinoid weight. Recall the forgetful morphism $q: \Xinfw \ra \Xw \sub \X$.

\begin{defi}\label{mainocdef}
Let $\U$ be a weight and let $w\in \mb{Q}_{>0}$ be such that $w\geq 1+s_{\U}$. We define a sheaf $\omega_{\U,w}^{\dg}$ on $\Xw$ by
$$ \omega_{\U,w}^{\dagger}(U)=\left\{ f \in \oo _{\Xinfw}(q^{-1}(U)) \ctens A_{\U} \mid \gamma^{\ast}f =  \chi_{\U}(b\mf{z}+d)^{-1}f  \; \forall \gamma \in K_0(p)\right\} $$
where $U \sub \Xw$ is a qcqs open subset.
\end{defi}

We remark that, since $b\mf{z}+d\in \Zp^{\times}(1+p^{w}\oo_{\Xinfw}^{+})$, the assumption $w\geq 1+s_{\U}$ ensures that $\chi_{\U}(b\mf{z}+d)$ is well defined (by Proposition \ref{prop-extend}). Define
$$\Delta_0(p) = \{  (\begin{smallmatrix} a & b \\ c & d \end{smallmatrix}) \in {\rm M}_{2}(\Zp)\cap \gl(\Qp) \mid c \in p\mb{Z}_p, \ d \in \mb{Z}_p ^{\times} \}.$$
This is a submonoid of $\gl(\Qp)$ containing $K_{0}(p)$, and it stabilizes $\Xinfw$ for all $w$. We can form a sheaf $\mc{F}$ on $\Xinfw$ by $\mc{F}(U_{\infty})=\oo_{\Xinfw}(U_{\infty})\ctens A_{\U}$ for $U_{\infty}\sub \Xinfw$ qcqs (\emph{not} necessarily of the form $q^{-1}(U)$), and we may equip it with a $\Delta_{0}(p)$-equivariant structure by the rule
$$ \gamma \cdot_{\U} - \, :\, \ \mc{F}(U_{\infty}) \ra \mc{F}(\gamma^{-1}U_{\infty}); $$
$$ \gamma \cdot_{\U}f = \chi_{\U}(b\mf{z}+d)\gamma^{\ast}f $$
for $\gamma\in \Delta_{0}(p)$. Then $\omega_{\U,w}^{\dg}=(q_{\ast}\mc{F})^{K_{0}(p)}$ and we obtain an action of the double cosets $K_{0}(p)\gamma K_{0}(p)$  for $\gamma\in \Delta_{0}(p)$ on $\omega_{\U,w}^{\dg}$, given by the following standard formula:
$$ [K_{0}(p)\gamma K_{0}(p)]\, :\, \omega_{\U,w}^{\dg}(V) \ra \omega_{\U,w}^{\dg}(U);$$
$$ [K_{0}(p)\gamma K_{0}(p)].f = \sum_{i}\chi_{\U}(b_{i}\mf{z}+d_{i})\gamma_{i}^{\ast}f $$
where $U\sub \Xw$ is qcqs, $V\sub \Xw$ is the image of $U$ under $[K_{0}(p)\gamma K_{0}(p)]$ viewed as a correspondence on $\Xw$, and $K_{0}(p)\gamma K_{0}(p) = \coprod_{i}\gamma_{i}K_{0}(p)$ is a coset decomposition. This gives us the Hecke action at $p$ for $\omega_{\U,w}^{\dg}$. As in the complex case, we may also view this action as an action by correspondences. We give a few remarks on this for Hecke operators at primes $\ell\neq p$ below; the action at $p$ is similar. The statements analogous to Proposition \ref{prop:hecke} that are needed follow from Proposition \ref{prop:omega}.

\medskip

For the Hecke actions at $\ell \neq p$, let us momentarily introduce the tame level into our notation, writing $\omega^{\dagger}_{\U,w,K^p}$ for the sheaf on $\mc{X}_{w,K^p}\subset \mc{X}_{K_0(p) K^p}$.  The construction of Hecke operators is then immediate from the following proposition.

\begin{prop}\label{prop:hecke}
\begin{enumerate}
\item For any $g \in G(\mathbb{A}_{f}^{p})$ with $g: \mc{X}_{w,K^p} \overset{\sim}{\to} \mc{X}_{w,g^{-1} K^p g}$ the associated isomorphism of Shimura curves, there is a canonical isomorphism $g^{\ast} \omega^{\dagger}_{\U,w,g^{-1} K^p g} \cong \omega^{\dagger}_{\U,w,K^p}$.
\item For any inclusion $K_2^p \subset K_1^p$ of tame levels with $\pi: \mc{X}_{w,K_2^p} \to \mc{X}_{w,K_1^p}$ the associated finite \'etale projection of Shimura curves, there is a canonical isomorphism \[\pi_{\ast}\omega^{\dagger}_{\U,w,K_2^p} \cong \omega^{\dagger}_{\U,w,K_1^p} \otimes_{\mc{O}_{\mc{X}_{w,K_1^p}}} \pi_{\ast} \mc{O}_{\mc{X}_{w,K_2^p}}\]  In particular there is a canonical $\mc{O}_{\mc{X}_{w,K_1^p}}$-linear trace map \[ \pi_{\ast}\omega^{\dagger}_{\U,w,K_2^p} \to \omega^{\dagger}_{\U,w,K_1^p} \]
\end{enumerate}
\end{prop}
\begin{proof}
(1) follows immediately from the definition of the sheaves together with the fact that $g^{\ast}\mf{z}=\mf{z}$; the latter follows from \cite[Theorem IV.1.1(iv)]{sch3}.

\medskip

For (2) we consider the pullback diagram
\[
\xymatrix{\mathcal{X}_{\infty,w,K_{2}^{p}}\ar[r]^{q_{2}}\ar[d]^{\pi_{\infty}} & \mathcal{X}_{w,K_{2}^{p}}\ar[d]^{\pi}\\
\mathcal{X}_{\infty,w,K_{1}^{p}}\ar[r]^{q_{1}} & \mathcal{X}_{w,K_{1}^{p}}
}
\]
of adic spaces. Given a rational subset $U\subset\mathcal{X}_{w,K_{1}^{p}}$,
we then have
\begin{eqnarray*}
\mathcal{O}(q_{2}^{-1}\pi_{1}U) & = & \mathcal{O}(\pi_{\infty}^{-1}q_{1}^{-1}U)\\
 & \cong & \mathcal{O}(q_{1}^{-1}U)\widehat{\otimes}_{\mathcal{O}(U)}\mathcal{O}(\pi^{-1}U)\\
 & \cong & \mathcal{O}(q_{1}^{-1}U)\otimes_{\mathcal{O}(U)}\mathcal{O}(\pi^{-1}U)
\end{eqnarray*}
where in the final line we use the fact that $\mathcal{O}(\pi^{-1}U)$
is a finite projective $\mathcal{O}(U)$-Banach module. Applying $-\widehat{\otimes}A_{\mathcal{U}}$
and making use of Lemma \ref{lemm: commute}, we get a canonical isomorphism
\[
\mathcal{O}(q_{2}^{-1}\pi_{1}^{-1}U)\widehat{\otimes}A_{\mathcal{U}}\cong\left(\mathcal{O}(q_{1}^{-1}U)\widehat{\otimes}A_{\mathcal{U}}\right)\otimes_{\mathcal{O}(U)}\mathcal{O}(\pi^{-1}U).
\]
Since $\pi_{\infty}^{\ast}\mathfrak{z}=\mathfrak{z}$, this isomorphism
is equivariant for the $\chi_{\mathcal{U}}$-twisted action of $K_{0}(p)$.
Passing to $K_{0}(p)$-invariants for the twisted action, the left-hand
side then becomes $\omega_{\mathcal{U},w,K_{2}^{p}}^{\dagger}(\pi^{-1}U)=(\pi_{\ast}\omega_{\mathcal{U},w,K_{2}^{p}}^{\dagger})(U)$,
while the right-hand side becomes $\omega_{\mathcal{U},w,K_{1}^{p}}^{\dagger}(U)\otimes_{\mathcal{O}(U)}(\pi_{\ast}\mathcal{O}_{\mathcal{X}_{w,K_{2}^{p}}})(U)$,
so (2) follows.
\end{proof}

\medskip

We now define our spaces of overconvergent modular forms.

\begin{defi} Let $\U$ be a weight and let $w\in \mb{Q}_{>0}$ be such that $w\geq 1+s_{\U}$.
\begin{enumerate}
\item We define the space of $w$-overconvergent modular forms $\mc{M}^{\dg,w}_{\U}$ of weight $\U$ by
$$ \mc{M}^{\dg,w}_{\U} = H^{0}(\Xw, \omega^{\dg}_{\U,w} \otimes_{\oo_{\Xw}}\Omega_{\Xw}^{1}). $$

\item We define the space of overconvergent modular forms $\mc{M}^{\dg}_{\U,w}$ of weight $\U$ by
$$ \mc{M}_{\U}^{\dg}= \varinjlim_{w}\mc{M}^{\dg,w}_{\U} . $$
\end{enumerate}
\end{defi}
The spaces $\mc{M}^{\dg,w}_{\U}$ are $\Qp$-Banach spaces, and $\mc{M}_{\U}^{\dg}$ is an LB-space. Using the functoriality of $\Omega_{\X}^{1}$ one may define Hecke operators on $\mc{M}^{\dg}_{\U,w}$ in the same way that we did for $\omega_{\U,w}^{\dg}$.

\begin{rema}
The reason we have formulated the definition in this way is because it fits naturally with our discussion of the overconvergent Eichler-Shimura map in \S \ref{sec:4}-\ref{sec:5}. It would be more fitting with our desire to be ``explicit" to consider the space
$$ H^{0}(\Xw, \omega^{\dg}_{\U+2,w}) $$
where $\U+2$ is the weight $(R_{\U}, z \mapsto \chi_{\U}(z)z^{2})$. Note that this shift by $2$ is forced upon us by our convention to associate $k\in \mb{Z}$ with the character $z \mapsto z^{k-2}$. Since $\omega$ is trivialized over $\Xinfw$ by $\mf{s}$ one sees that $\omega_{\U+2,w}^{\dg} \cong \omega_{\U,w}^{\dg}\otimes_{\oo_{\Xw}}\omega^{2}$. By the Kodaira-Spencer isomorphism $\Omega_{\X}^{1}\cong \omega^{2}$ we then have that $\omega^{\dg}_{\U+2,w}\cong \omega_{\U,w}^{\dg} \otimes_{\oo_{\Xw}} \Omega_{\Xw}^{1}$. Thus the two definitions give the same spaces. However, it is well known that the Kodaira-Spencer isomorphism fails to be equivariant for the \emph{natural} Hecke actions on both sides. Indeed it is customary in the theory of modular forms to renormalize the natural Hecke action on $\omega^{\otimes k}$ by multiplying $T_{\ell}$ by $\ell^{-1}$ (more precisely one multiplies the action of any double coset by $\det^{-1}$). The Kodaira-Spencer isomorphism is then Hecke-equivariant. Thus, by defining $\mc{M}^{\dg}_{\U,w}$ the way we have we can use the natural Hecke actions; there is no need to normalize. 

\medskip

For a discussion of how the normalized Hecke action on $\omega^{\otimes k}$ corresponds to the natural Hecke action on $\omega^{\otimes k-2}\otimes_{\oo_{\X}}\Omega_{\X}^{1}$, see \cite[p. 257-258]{fc}.
\end{rema}

\subsection{Locally projective of rank one}

In this subsection we prove that $\omega_{\mc{U},w}^{\dagger}$ is locally projective of rank one as a sheaf
of $\mc{O}_{\mc{X}_w} \widehat{\otimes} A_{\mc{U}}$-modules.

\medskip

First we prove a general lemma. Let $A_{\infty}$ be a uniform $\Qp$-Banach algebra equipped with an action of a profinite group $G$ by continuous homomorphisms. Let $A=A_{\infty}^{G}$. We record the following easy facts:

\begin{prop}
$A$ is a closed subalgebra of $A_{\infty}$, hence carries an induced structure of a uniform $\Qp$-Banach algebra, and $A^{\circ} = (A_{\infty}^{\circ})^{G}$.
\end{prop}

We then have:

\begin{lemm}\label{lemm: inv1}
Keep the above notations and assumptions.
\begin{enumerate}
\item Let $M$ be any profinite flat $\Zp$-module (in particular $M$ could be a small $\Zp$-algebra). We equip $M$ with the trivial $G$-action. Then $(A_{\infty} \ctens M)^{G} = A \ctens M$.

\item Let $V$ be a Banach space over $\Qp$ (in particular $V$ could be a reduced affinoid $\Qp$-algebra topologically of finite type). Equip $V$ with the trivial $G$-action. Then $(A_{\infty} \ctens_{\Qp} V)^{G} = A \ctens_{\Qp} V$.

\end{enumerate}
\end{lemm}

\begin{proof}
To prove (1), we choose a pseudobasis $(e_{i})_{i\in I}$ of $M$ and follow the computation in Proposition \ref{prop: appflat} to see that the natural map $A^{\circ} \ctens M \ra A^{\circ}_{\infty} \ctens M $ is the inclusion 
$$ \prod_{i\in I} A^{\circ}e_{i} \sub \prod_{i\in I} A^{\circ}_{\infty}e_{i} $$
with the action of $G$ on the right hand side being coordinate-wise. Thus the natural map $A \ctens M \ra A_{\infty} \ctens M$ is the inclusion of bounded sequences in $A$ indexed by $I$ into bounded sequences in $A_{\infty}$ indexed by $I$, with the $G$-action on the latter being coordinate-wise. The desired statement now follows from the definition of $A$ as $A_{\infty}^{G}$. 

The proof of (2) is similar, using an orthonormal basis instead of a pseudobasis.
\end{proof}

We record the following fact, which is certainly implicit in \cite{sch2} (our main reason for recording it is that some assumption on the sheaf $\mc{F}$ seems to be needed).

\begin{lemm}\label{newinv}
Let $Y$ be a rigid analytic variety, let $G$ be a profinite group and assume that $Y_{\infty} \in Y_{\proet}$ is a Galois $G$-cover of $Y$. Let $U \in Y_{\proet}$ be quasicompact and quasiseparated and set $U_{\infty}:=U \times_{Y}Y_{\infty}$; this is a Galois $G$-cover of $U$. Let $\mc{F}$ be a sheaf on $Y_{\proet}$ that comes via pullback from $Y_{\et}$. Then $\mc{F}(U)=\mc{F}(U_{\infty})^{G}$.
\end{lemm}

\begin{proof}
Recall that, by the definitions, we may find a system of open normal subgroups $G_{j}$ of $G$ such that $G\cong \varprojlim_{j} G/G_{j}$, and a compatible system of Galois $G/G_{j}$-covers $Y_{j}$ of $Y$ such that $Y_{\infty} = \varprojlim_{j}Y_{j}$ is a pro-\'etale presentation. Pick a pro-\'etale presentation $U=\varprojlim_{i}U_{i}$; we get $U_{\infty}=\varprojlim_{i,j}U_{i}\times_{Y}Y_{j}$. Note that $U_{\infty}$ is quasicompact and quasiseparated by \cite[Proposition 3.12(v)]{sch2}. We may then compute
\begin{eqnarray*}
\mc{F}(U_{\infty})^{G} & = & \left( \varinjlim_{i,j} \mc{F}(U_{i}\times_{Y}Y_{j}) \right)^{G} \\
 & = & \varinjlim_{i,j} \left( \mc{F}(U_{i}\times_{Y}Y_{j})^{G} \right)  \\
 & = & \varinjlim_{i,j} \mc{F}(U_{i}) =  \mc{F}(U). \end{eqnarray*}
Here we have used \cite[Lemma 3.16]{sch2} for the first and fourth equality and that $U_{i}\times_{Y}Y_{j} \ra U_{i}$ is a Galois $G/G_{j}$-cover for the third equality. For the second equality we use that $\mc{F}(U_{i}\times_{Y}Y_{j}) \ra \mc{F}(U_{i^{\prime}}\times_{Y}Y_{j^{\prime}})$ is injective for large enough $i$, since $U_{i^{\prime}}\times_{Y}Y_{j^{\prime}} \ra U_{i}\times_{Y}Y_{j}$ is an \'etale cover for large enough $i$ (by the definition of a pro-\'etale presentation), and that direct limits commute with taking invariants if the maps in the direct system are eventually injective.
\end{proof}

\begin{rema}\label{cartanleray}
An elaboration of this argument allows one to deduce the full ``Cartan--Leray spectral sequence"
$$ E_{2}^{pq}=H^{p}_{cts}(G, H^{q}(U_{\infty}, \mc{F})) \implies H^{p+q}(U, \mc{F}) $$
from the \v{C}ech-to-derived functor spectral sequence, under the same assumptions on $U$ and $\mc{F}$ (here $H^{q}(U_{\infty},\mc{F})$ is given the discrete topology). This is implicit in \cite{sch2} and used repeatedly there (though it should be noted that it is the \v{C}ech-to-derived functor spectral sequence itself that is referred to as the Cartan--Leray spectral sequence in \cite{sch2}, following SGA 4).
\end{rema}

\begin{lemm} \label{lemm: inv2}
Let $K_{p}\sub \gl(\Zp)$ be an open compact subgroup and let $U\sub \mc{X}_{K_{p}}$ be an open subset. Put $U_{\infty}=q_{K_{p}}^{-1}(U)\sub \Xinf$. Then $\oo_{\mc{X}_{K_{p}}}^{+}(U)=\oo_{\Xinf}^{+}(U_{\infty})^{K_{p}}$, and hence $\oo_{\mc{X}_{K_{p}}}(U)=\oo_{\Xinf}(U_{\infty})^{K_{p}}$.
\end{lemm}

\begin{proof}
We work on the pro-\'etale site of $\mc{X}_{K_{p}}$. We may assume that $U$ is quasicompact; the general case follows by gluing. $U_{\infty} \ra U$ is a perfectoid object of $\mc{X}_{K_{p},\proet}$ which is Galois with group $K_{p}$, and Lemma \ref{newinv} implies that  $(\oo_{\mc{X}_{K_{p}}}^{+}/p^{m})(U)=(\oo_{\X_{K_{p}}}^{+}/p^{m})(U_{\infty})^{K_{p}}$ for all $m$. Taking the inverse limit we get, by definition, that $\wh{\oo}_{\X_{K_{p}}}^{+}(U)=\wh{\oo}_{\X_{K_{p}}}^{+}(U_{\infty})^{K_{p}}$. By \cite[Corollary 6.19]{sch2} we have $\wh{\oo}_{\mc{X}_{K_{p}}}(U)=\oo_{\mc{X}_{K_{p}}}(U)$. It then follows from \cite[Lemma 4.2(ii),(v)]{sch2} that $\wh{\oo}^{+}_{\mc{X}_{K_{p}}}(U)=\oo^{+}_{\mc{X}_{K_{p}}}(U)$, which finishes the proof (note that $\wh{\oo}_{\X_{K_{p}}}^{+}(U_{\infty})=\oo^{+}_{\Xinf}(U_{\infty})$ ).
\end{proof}

Now let $U\sub \mc{X}_{w}$ be any rational subset, $U_{n}$
its preimage in $\mc{X}_{K(p^{n})}$ and put $U_{\infty}=q^{-1}(U)\sub \Xinfw$.
Note that $U_{n} \ra U$ is finite \'etale and so $\oo_{\mc{X}_{K(p^{n})}}(U_{n})$ is a finite projective $\oo_{\X}(U)$-module since $\oo_{\X}(U)$ is Noetherian. Suppose that $\omega|_{U}$ is free,
and choose a nowhere vanishing section (i.e. generator) $\eta_{U}\in H^{0}(U,\omega)$.
Define $t_{U}\in\oo_{\Xinf}(U_{\infty})$ by the equality
\[\mf{s}=t_{U}\cdot q^{\ast}\eta_{U}.\]

\begin{prop}\label{prop: approx} We have $\gamma^{\ast}t_{U}=(b\mf{z}+d)t_{U}$
for any $\gamma\in K_{0}(p)$, and $t_{U}$ is a unit. Furthermore,
for any $m\in\mb{Z}_{\geq1}$ we may choose some large $n=n(m)$
and elements\begin{eqnarray*}
t_{U}^{(n)} & \in & 1+p^{m}\oo_{\Xinf}^{+}(U_{\infty}),\\
s_{U,n} & \in & \oo_{\mc{X}_{K(p^{n})}}(U_{n})^{\times},\end{eqnarray*}
such that $t_{U}=t_{U}^{(n)}s_{U,n}$.
\end{prop}

\begin{proof} Since neither $\mf{s}$ nor $q^{\ast}\eta_{U}$ vanish, $t_{U}$ is a unit, and 
$$(b\mf{z}+d)t_{U}q^{\ast}\eta_{U}= (b\mf{z}+d)\mf{s}=\gamma^{\ast}\mf{s}= (\gamma^{\ast}t_{U})q^{\ast}\eta_{U} $$ 
since $\gamma^{\ast}q^{\ast}\eta_{U}=q^{\ast}\eta_{U}$, hence $\gamma^{\ast}t_{U}=(b\mf{z}+d)t_{U}$. For the second sentence, first choose $M,N\in \mb{Z}_{\geq 0}$  such that $|p^{M}|\leq |p^{N}t_{U}| \leq 1$ (possible by quasicompactness of $U_{\infty}$ and invertibility of $t_{U}$). Since $\oo_{\Xinf}^{+}(U_{\infty}) = (\varinjlim \oo_{\mc{X}_{K(p^{n})}}^{+}(U_{n}))^{\wedge}$ ($p$-adic completion) we may find an $n=n(m)$ and $s\in \oo_{\mc{X}_{K(p^{n})}}^{+}(U_{n})$ such that $p^{N}t_{U}-s\in p^{M+m}\oo_{\Xinf}^{+}(U_{\infty})$. Then we set $s_{U,n}=p^{-N}s$ and $t_{U}^{(n)}=t_{U}/s_{U,n}$ and these do the job.
\end{proof}

This has several consequences - in particular, since $s_{U,n}$ is
fixed by $K(p^{n})$ and $t_{U}^{(n)}$ is $p$-adically close to
$1$, the element $\chi_{\mc{U}}(t_{U}^{(n)})\in (\oo_{\Xinf}(U_{\infty}) \ctens A_{\mc{U}})^{\times}$ is well-defined (for $m$
large enough) and satisfies $\gamma^{\ast}(\chi_{\U}(t_{U}^{(n)}))=\chi_{\U}(b\mf{z}+d)\chi_{\mc{U}}(t_{U}^{(n)})$
for all $\gamma\in K(p^{n})$.
\medskip
Recall that 
$$ \omega_{\U,w}^{\dagger}(U):=\left\{ f\in\oo_{\Xinf}(U_{\infty})\ctens A_{\U} \mid \gamma^{\ast}f=\chi_{\U}(b\mf{z}+d)^{-1}f\,\forall\gamma\in K_{0}(p)\right\} . $$
Thus, given any $f\in\omega_{\mc{U},w}^{\dagger}(U)$, the element $f\cdot\chi_{\mc{U}}(t_{U}^{(n)})\in\mathcal{O}_{\Xinf}(U_{\infty})\ctens A_{\mc{U}}$
lies in $(\oo_{\Xinf}(U_{\infty})\ctens A_{\mc{U}})^{K(p^{n})}$, which is equal to $\oo_{\mc{X}_{K(p^{n})}}(U_{n}) \ctens A_{\U}$ by Lemma \ref{lemm: inv1} and Lemma \ref{lemm: inv2}.

\medskip

Furthermore, for any $\gamma\in G_{n}:=K_{0}(p)/K(p^{n})$, we have
\begin{eqnarray*}
\gamma^{\ast}(f\cdot\chi_{\U}(t_{U}^{(n)})) & = & \chi_{\U}(b\mf{z}+d)^{-1}\cdot f\cdot\chi_{\U}(\gamma^{\ast}t_{U}^{(n)})\\
 & = & f\cdot\chi_{\U}(t_{U}^{(n)})\cdot\chi_{\U}\left(\frac{\gamma^{\ast}t_{U}^{(n)}}{(b\mf{z}+d)t_{U}^{(n)}}\right)\\
 & = & f\cdot\chi_{\U}(t_{U}^{(n)})\cdot\chi_{\U}\left(\frac{s_{U,n}}{\gamma^{\ast}s_{U,n}}\right)\\
 & = & f\cdot\chi_{\U}(t_{U}^{(n)})\cdot\chi_{\U}(j_{U,n}(\gamma))^{-1}\end{eqnarray*}
where $j_{U,n}$ is the cocycle \begin{eqnarray*}
j_{U,n}(\gamma):G_{n} & \to & \Zp^{\times}\cdot\left(1+p^{\mathrm{min}(w,m)}\mathcal{O}_{\X_{K(p^{n})}}(U_{n})^{\circ}\right)\\
\gamma & \mapsto & \frac{\gamma^{\ast}s_{U,n}}{s_{U,n}}.\end{eqnarray*}
In summary, we find that the map $f\mapsto f\cdot\chi_{\U}(t_{U}^{(n)})$
defines an $\oo_{\X}(U)$-\emph{module} \emph{isomorphism }of $\omega_{\U,w}^{\dagger}(U)$
onto the subspace of functions $f_{0}\in\oo_{\mc{X}_{K(p^{n})}}(U_{n})\ctens A_{\U}$
fixed by the twisted action \begin{eqnarray*}
\gamma:\oo_{\mc{X}_{K(p^{n})}}(U_{n})\ctens A_{\mc{U}} & \to & \oo_{\mc{X}_{K(p^{n})}}(U_{n})\ctens A_{\mc{U}}\\
f_{0} & \mapsto & \chi_{\mc{U}}(j_{U,n}(\gamma))\cdot\gamma^{\ast}f_{0}\end{eqnarray*}
of $G_{n}$. Since $G_{n}$ is a finite group, the usual idempotent
$$ e_n = \frac{1}{|G_{n}|}\sum_{\gamma\in G_{n}} \chi_{\U}(j_{U,n}(\gamma))^{-1}\gamma $$
in the group algebra $(\oo_{\mc{X}_{K(p^{n})}}(U_{n}) \ctens A_{\U})[G_{n}]$ defines an $\oo_{\X}(U)\ctens A_{\U}$-module splitting of the inclusion $\omega_{\U,w} ^{\dg} (U)\sub \oo_{\mc{X}_{K(p^{n})}}(U_{n})\ctens A_{\U}$.
This realizes $\omega_{\U,w}^{\dagger}(U)$ as a direct summand of a finite
projective $\oo_{\X}(U) \ctens A_{\U}$-module, and therefore $\omega_{\U,w}^{\dagger}(U)$
is finite projective over $\oo_{\X}(U) \ctens A_{\U}$ as desired. We may now prove the main result of this section:

\begin{theo} \label{theo:rankone}
We have $\omega_{\mc{U},w} ^{\dg}=\Loc(\omega_{\U,w}^{\dg}(\Xw))$ and $\omega_{\U,w}^{\dg}(\Xw)$ is a finite projective $\oo_{\X}(\Xw) \ctens A_{\U}$-module of rank $1$. Moreover, $\omega_{\U,w}^{\dg}$ is \'etale locally free.
\end{theo}

\begin{proof}
To prove that $\omega_{\mc{U},w}^{\dg}=\Loc(\omega_{\U,w}^{\dg}(\Xw))$  and that it is locally projective of finite rank it suffices, by Theorem \ref{theo: thmA}, to prove that $\omega_{\U,w}^{\dg}$ is a coherent $\oo_{\X} \ctens A_{\U} $-module (it is then locally projective by the above). To do this, we work locally using the $U$ above. We wish to show that for $V\sub U$ (without loss of generality assume $V$ rational) the natural map 
$$(\oo_{\X}(V) \ctens A_{\U} ) \otimes_{(\oo_{\X}(U) \ctens A_{\U} )} \omega_{\U,w}^{\dg}(U) \ra \omega_{\U,w}^{\dg}(V)$$
is an isomorphism. To see this, note that by applying Lemma \ref{lemm: commute} twice the natural map
$$ (\oo_{\X}(V) \ctens A_{\U} ) \otimes_{(\oo_{\X}(U) \ctens A_{\U} )} (\oo_{\mc{X}_{K(p^{n})}}(U_{n}) \ctens A_{\U}) \ra \oo_{\mc{X}_{K(p^{n})}}(V_{n}) \ctens A_{\U} $$
is an isomorphism (note that $\oo_{\X}(V) \otimes_{\oo_{\X}(U)} \oo_{\mc{X}_{K(p^{n})}}(U_{n}) \cong \oo_{\mc{X}_{K(p^{n})}}(V_{n}) $). This isomorphism matches up the idempotents on both sides, giving us the desired isomorphism. To prove that the rank is one note that $\oo_{\X}(U) \ra \oo_{\mc{X}_{K(p^{n})}}(U_{n})$ is Galois with group $G_{n}$ and by applying Lemma \ref{lemm: commute} twice we may deduce that $\oo_{\X}(U) \ctens A_{\U} \ra \oo_{\mc{X}_{K(p^{n})}}(U_{n}) \ctens A_{\U}$ is also Galois with group $G_{n}$. Moreover, the twisted action whose invariants give $\omega_{\U,w}^{\dg}(U)$ is a Galois descent datum, so we see that $\omega_{\U,w}^{\dg}(U)$ is the descent of a rank $1$ free module, hence rank $1$ projective as desired. Finally, note that this also proves that $\omega_{\U,w}$ becomes trivial over $U_{n}$, which gives the final statement of the theorem.
\end{proof}

The techniques of this proof also yield the following result.

\begin{lemm}\label{lemm:fibre} Let $\U=(A_{\U},\chi_{\U})$ be a weight, and let $i: A_{\U} \to A_{\mc{Z}}$ be a surjection such that $\mc{Z}=(A_{\mc{Z}},\chi_{\mc{Z}}=i\circ \chi_{\U})$ is also a weight and such that $\ker(i)$ is generated by a regular element $x\in A_{\U}$. Then we have a natural exact sequence of sheaves 
$$ 0 \to \omega_{\U,w}^{\dg} \overset{\cdot x}{\to} \omega_{\U,w}^{\dg} \to \omega_{\mc{Z},w}^{\dg}\to 0$$ on $\mc{X}_{w}$, and an exact sequence of global sections $$0\to H^0(\mc{X}_{w},\omega_{\U,w}^{\dg}) \overset{\cdot x}{\to} H^0(\mc{X}_w,\omega_{\U,w}^{\dg})\to H^0(\mc{X}_w,\omega_{\mc{Z},w}^{\dg})\to 0.$$

\end{lemm}

\begin{proof} There is certainly a (not necessarily exact) sequence of sheaves $$ 0 \to \omega_{\U,w}^{\dg} \overset{\cdot x}{\to} \omega_{\U,w}^{\dg} \to \omega_{\mc{Z},w}^{\dg}\to 0,$$and we check exactness of this sequence on a basis of suitably small open subsets $U \subset \mc{X}_w$ as in the proof of Theorem \ref{theo:rankone}.  By assumption we have a short exact sequence 
$$0 \to A_{\U} \overset{\cdot x}{\to} A_{\U} \to A_{\mc{Z}} \to 0$$ 
of $A_{\U}$-modules. Tensoring this sequence over $A_{\U}$ with $A_{\U} \ctens \oo_{\mc{X}_{K(p^n)}}(U_n)$, we obtain by the flatness of $A_{\U} \ctens \oo_{\mc{X}_{K(p^n)}}(U_n)$ over $A_{\U}$  (proved in the same way as Lemma \ref{lemm:flat}(2)) a short exact sequence 
$$0 \to A_{\U} \ctens \oo_{\mc{X}_{K(p^n)}}(U_n) \overset{\cdot x}{\to} A_{\U} \ctens \oo_{\mc{X}_{K(p^n)}}(U_n) \to A_{\mc{Z}} \ctens \oo_{\mc{X}_{K(p^n)}}(U_n) \to 0.$$
Applying the idempotent $e_n$ as above gives a short exact sequence $$ 0 \to \omega_{\U,w}^{\dagger}(U) \overset{\cdot x}{\to} \omega_{\U,w}^{\dg}(U) \to \omega_{\mc{Z},w}^{\dg}(U)\to 0$$ as desired.

Taking cohomology, we note that $H^1(\mc{X}_w, \omega_{\U,w}^{\dg})=0$ by Theorem \ref{theo:rankone} and Proposition \ref{prop:loc}, and the lemma follows.
\end{proof}

\subsection{Comparison with other definitions of overconvergent modular forms}

In this section we take $\V=(S_{\V},\chi_{\V})$ to be an open affinoid weight, and we will compare our definition of overconvergent modular forms of weight $\V$ with those in the literature (all known to be equivalent). More specifically we will compare it to that of \cite{pil}, trivially modified to our compact Shimura curves. 

\medskip

We now recall the definition of the "Pilloni torsor". This is the object denoted by $F_{n}^{\times}$ in \cite{pil} but we will use the notation $\mc{T}(n,v)$. For any $n \geq 1$ and any $v<\frac{p-1}{p^n}$, $\mc{T}(n,v)$ is a rigid space equipped with a smooth surjective morphism $\pr:\mc{T}(n,v)\to \mc{X}(v)$, where $\mc{X}(v)\sub \mc{X}$ is the locus where the Hodge height is $\leq v$.  For any point $x \in \mc{X}(v)(C,\mc{O}_C)$, the $(C,\mc{O}_C)$-points in the fiber $\pr^{-1}(x)$ consists of the differentials $\eta \in \omega_{\mc{G}_{\mc{A}}}$ that reduce to an element in the image of the Hodge-Tate map $(H_{n}^{\vee})^{\gen} \ra \omega_{H_{n}}$, where $H_{n}$ is the canonical subgroup of $\mc{G}$ of level $n$ and $(H_{n}^{\vee})^{\gen}\sub H_{n}^{\vee}$ denotes the subset of generators. The set $\pr^{-1}(x)$ is a torsor for the group $\Zp^{\times}(1+p^{n-\frac{p^{n}-1}{p-1}\Hdg(x)}\mc{O}_{C})$, and we think of $\mc{T}(n,v)$ as an open subspace of $\mc{T}$, where $\mc{T}$ is the total space of $e(0^{\ast}\Omega_{\mc{A}^{\univ} / \mc{X}(v)}^{1})$ (here $0\, :\, \mc{X}(v) \ra \mc{A}^{\univ}$ is the zero section; this is a line bundle on $\mc{X}(v)$). We remark that the canonical action of $\Zp^{\times}$ on $\mc{T}$ preserves $\mc{T}(n,v)$ (as should be clear from the description of the fibres above). This induces an action on functions.

\medskip

If we fix $w=n-\frac{p^{n}-1}{p-1}v$, then $\mc{X}(v)\sub \mc{X}_{w}$; this follows for example from \cite[Proposition 3.2.1]{aip}. We will compare our construction of $w$-overconvergent modular forms over $\mc{X}(v)$ to the notion defined in \cite{pil}. For the purpose of the comparison, a \emph{Pilloni form} of weight $\V$ over a quasi-compact open $U\sub \mc{X}(v)$ is defined to be an element $\mbf{f}\in \oo_{\mc{T}(n,v)}(\pr^{-1}(U)) \ctens_{\Qp} S_{\V}$ (i.e a function on $\pr^{-1}(U)\times \V \sub \mc{T}(n,v)\times \mc{W}$. Here we abuse notation and write $\V$ also for the image of the natural map into $\mc{W}$ induced by $\V$) such that $z.\mbf{f}=\chi_{\V}^{-1}(z)\mbf{f}$ for all $z\in \Zp^{\times}$.

\medskip

Now fix $U\sub \mc{X}(v)$ as used in the previous section, such that we have a nowhere vanishing $\eta_{U}\in \omega(U)$. We freely use the notation of the previous section with one exception: we use $r$ instead of the already used letter $n$. Thus we have functions $t_{U}^{(r)}\in 1+p^{m}\oo_{\Xinf}^{+}(U_{\infty})$ and $s_{U,r}\in \oo_{\mc{X}_{K(p^{r})}}(U_{r})^{\times}$ such that 
$$ \mf{s}=t_{U}^{(r)}s_{U,r}\eta_{U}.$$
Let $\mc{T}_{U,r}(n,v)=\mc{T}(n,v) \times_{U}U_{r}$. It is an open subset of $\mc{T}_{U,r}:=\mc{T} \times_{U}U_{r}$. It inherits commuting actions of $G_{r}:=K_{0}(p)/K(p^{r})$ and $\Zp^{\times}$.
\begin{lemm} \label{lemm: triv}
Assume that $m\geq n$ and $r\geq n$. Then $s_{U,r}\eta_{U}$ trivializes $\mc{T}_{U,r}(n,v)$.
\end{lemm}

\begin{proof}
We think of $s_{U,r}\eta_{U}$ as a section $U_{r} \ra \mc{T}_{U,r}$ and we wish to show that the image lands inside $\mc{T}_{U,r}(n,v)$, for which it is enough to argue on geometric points. Take a $(C,\OC)$-point $\bar{x}$ and lift it to a $(C,\OC)$-point $x$ of $\Xinf$. Then we see that 
$$(s_{U,r}\eta_{U})(\bar{x})=((t_{U}^{(r)})^{-1}\mf{s})(x).$$
If $\bar{x}=(A,\bar{\alpha})\in \mc{Y}_{K(p^{r})}^{gd}$ we put $x=(A,\alpha)$. Then $\mf{s}(A,\alpha)=\HT_{A}(\alpha(e_{2}))$ by definition and this maps to a generator of $H_{n}^{\vee}$ via the canonical map $\mc{G}[p^{n}]\cong \mc{G}[p^{n}]^{\vee} \ra H_{n}^{\vee}$, so $\mf{s}(A,\alpha)$ lies in the fibre of $\mc{T}_{U,r}(n,v)$ over $(A,\bar{\alpha})$. This proves that $\mf{s}(x)$ lies in the fibre of $\mc{T}_{U,r}(n,v)$ over $\bar{x}$ for all $x$. The result follows since $t_{U}^{(r)}$ is small.
\end{proof}

The assumption $m\geq n$ and $r\geq n$ will be in force throughout the rest of this section so that the Lemma applies. Recall what we proved in the process of proving Theorem \ref{theo:rankone}: $w$-overconvergent modular forms $f$ of weight $\V$ over $U$ identifies, via the map $f\mapsto f_{0}:=\chi_{\V}(t_{U}^{(r)})f$, with functions $f_{0}\in  \oo_{\mc{X}(p^{r})}(U_{r}) \ctens_{\Qp} S_{\V}$ such that 
$$ \gamma^{\ast}f_{0}=\chi_{\V}(j_{U,r}(\gamma))^{-1}f_{0} $$
for all $\gamma\in G_{r}$.

\begin{prop}\label{prop: ident}
\begin{enumerate}

\item The space of Pilloni forms of weight $\V$ over $U$ is isomorphic, via pullback, to the space of functions $\mbf{g}$ on $\mc{T}_{U,r}(n,v)\times \V$ such that $z.\mbf{g}=\chi_{\V}(z)^{-1}\mbf{g}$ for all $z\in \Zp^{\times}$ and $\gamma^{\ast}\mbf{g}=\mbf{g}$ for all $\gamma \in G_{r}$. 

\item The space of $w$-overconvergent modular forms of weight $\V$ over $U$ is isomorphic to the space of functions $g_{0}$ on $\mc{T}_{U,r}(n,v) \times \V$ such that $z.g_{0}=g_{0}$ for all $z\in \Zp^{\times}$ and $\gamma^{\ast}g_{0}=\chi_{\V}(j_{U,r}(\gamma))^{-1}g_{0}$ for all $\gamma \in G_{r}$. The isomorphism is the map $f\mapsto f_{0}$ above composed with pullback from $U_{r} \times \V$ to $T_{U,r}(n,v) \times \V$.

\end{enumerate}
\end{prop}

\begin{proof}
(1) follows from the fact that $\pr^{-1}(U)\times \V$ is the quotient of $T_{U,r}(n,v) \times \V$ by $G_{r}$. The proof of (2) is a similar descent.
\end{proof}

To identify the two spaces, it remains to go from one kind of function on $\mc{T}_{U,r}(n,w) \times \V$ to the other. The key lies in trivializing the cocycle $\chi_{\V}(j_{U,r}(\gamma))$. Note that $s_{U,r}\eta_{U}$, as a nowhere vanishing section of $\mc{T}_{U,r}$ also canonically defines a function on $\mc{T}_{U,r}$ \emph{minus the zero section} which we will denote by $(s_{U,r}\eta_{U})^{\vee}$. The same applies to $\eta_{U}$ itself and we have $(s_{U,r}\eta_{U})^{\vee}=s_{U,r}^{-1}\eta_{U}^{\vee}$. By restriction we obtain functions on $\mc{T}_{U,r}(n,v)$.

\begin{lemm}\label{lemm: small}
$(s_{U,r}\eta_{U})^{\vee}\in \Zp^{\times}.(1+p^{w}\oo_{\mc{T}_{U,r}}^{+}(\mc{T}_{U,r}(n,v)))$.
\end{lemm}

\begin{proof}
This follows directly from Lemma \ref{lemm: triv}.
\end{proof}

This implies that we may apply $\chi_{\V}$ to $(s_{U,r}\eta_{U})^{\vee}$, and we may therefore define a $\oo_{\mc{X}(v)}(U) \ctens_{\Qp} S_{\V}$-module isomorphism
$$ \Phi\, :\, \oo_{\mc{T}_{U,r}}(\mc{T}_{U,r}(n,v)) \ctens_{\Qp} S_{\V} \ra \oo_{\mc{T}_{U,r}}(\mc{T}_{U,r}(n,v)) \ctens_{\Qp} S_{\V} $$
which is simply multiplication by $\chi_{\V}((s_{U,r}\eta_{U})^{\vee})$:
$$ \Phi(h)=\chi_{\V}((s_{U,r}\eta_{U})^{\vee})h. $$

\begin{theo}\label{pillonicomp}
The image of the space of functions in part (1) of Proposition \ref{prop: ident} under $\Phi$ is the space of functions in part (2) of Proposition \ref{prop: ident}. Moreover, the induced isomorphism between the space of $w$-overconvergent modular forms of weight $\V$ over $U$ and the space of Pilloni forms of weight $\V$ over $U$ is independent of all choices and hence the isomorphisms for varying $U$ glue together to an isomorphism of sheaves over $\Xw$.
\end{theo}

\begin{proof}
For the first part, the key thing to notice is that
$$ j_{U,r}(\gamma)=\frac{(s_{U,r}\eta_{U})^{\vee}}{\gamma^{\ast}(s_{U,r}\eta_{U})^{\vee}} $$
and hence that
$$ \chi_{\V}(j_{U,r}(\gamma))^{-1}=\frac{\chi_{\V}(\gamma^{\ast}(s_{U,r}\eta_{U})^{\vee})}{\chi_{\V}((s_{U,r}\eta_{U})^{\vee})}. $$
The rest is then a straightforward computation. For the second part we remark that the composite isomorphism from $w$-overconvergent modular forms to Pilloni forms is given by
$$ f \mapsto \chi_{\V}(t_{U}^{(r)})\chi_{\V}((s_{U,r}\eta_{U})^{\vee})^{-1}f $$
and then descending the right hand side to $\pr^{-1}(U)\times \V$. Morally, the right hand side is equal to $\chi_{\V}(\mf{s}^{\vee})^{-1}f$ (recall that $t_{U}^{(r)}s_{U,r}\eta_{U}=\mf{s}$) and hence independent of the choices made. To turn this into a rigorous argument is straightforward but tedious and notationally cumbersome. We leave the details to the interested reader.
\end{proof}

\medskip

\section{Overconvergent modular symbols}\label{sec:3}

We will recall some material on overconvergent modular symbols in the form we need. Most of these constructions are probably well known with the exception of certain filtrations defined in \cite{han}.

\subsection{Basic definitions and the filtrations}

Let $\mbf{A}^{s}$ be the affinoid ring over $\Qp$ defined by 
$$ \mbf{A}^{s}=\{ f: \Zp \ra \Qp \mid f\,\mathrm{analytic\, on\, each\,}p^{s}\Zp-\mathrm{coset} \} . $$
We let $\mbf{A}^{s,\circ}$ denote the subring of powerbounded elements of $\mbf{A}^{s}$. Given a weight $\U$, consider
the module $$ \mbf{A}_{\U}^{s,\circ} = \mbf{A}^{s,\circ} \ctens_{\Zp} A_{\U}^{\circ} $$ 
where the completion is the $p$-\emph{adic completion}. Recall the Amice basis $(e_{j}^{s})_{j\geq 0}$ of $\mbf{A}^{s,\circ}$ from Theorem \ref{theo: amice}. Using it we may write $\mbf{A}_{\U}^{s,\circ}$ as 
$$ \mbf{A}_{\U}^{s,\circ}=\wh{\bigoplus}_{j\geq 0} A_{\U}^{\circ}e_{j}^{s} $$
and hence view elements of $\mbf{A}_{\U}^{s,\circ}$ as functions $\Zp \ra A_{\U}^{\circ}$. 
For any $s \geq s_{\mc{U}}$ and any $d\in\Zp^{\times},c\in p\Zp$, $x \mapsto \chi_{\U}(cx+d)$ defines an element of $\mbf{A}_{\U}^{s,\circ}$ (by calculations very similar to those in the proof of Proposition \ref{prop-extend}), and we then consider $\mbf{A}_{\U}^{s,\circ}$ endowed with the right $\Delta_{0}(p)$-action 
$$(f\cdot_{\U}\gamma)(x)=\chi_{\U}(cx+d)f\left(\frac{ax+b}{cx+d}\right);$$
one checks without too much trouble that $f\cdot_{\U}\gamma \in \mbf{A}_{\U}^{s,\circ}$. We set $\mbf{D}_{\U}^{s,\circ}=\Hom_{A_{\U}^{\circ}}(\mbf{A}_{\U}^{s,\circ},A^{\circ}_{\U})$
and $\mbf{D}_{\U}^{s}=\mbf{D}_{\U}^{s,\circ}[\frac{1}{p}]$, with
the dual left action. It is the continuous $A_{\U}[\frac{1}{p}]$-dual of the $A_{\U}[\frac{1}{p}]$-Banach module $\mbf{A}_{\U}^{s}:=\mbf{A}_{\U}^{s,\circ}[\frac{1}{p}]$. Note that 
$$ \mbf{D}^{s,\circ}_{\U}=\Hom_{\Zp}(\mbf{A}^{s,\circ},A^{\circ}_{\U}). $$
The Amice basis gives an orthonormal $A_{\U}[\frac{1}{p}]$-basis of $\mbf{A}_{\U}^{s}$ and induces an isomorphism 
$$ \mbf{D}_{\U}^{s}\, \tilde{\longrightarrow}\, \prod_{j\geq 0}A^{\circ}_{\U} $$
given by
$$ \mu \mapsto (\mu(e^{s}_{j}))_{j\geq 0}. $$

\begin{prop} Let $\U=(R_{\U},\chi_{\U})$ be a small weight and let $s\geq 1+s_{\U}$.
\begin{enumerate}
\item $\mbf{D}_{\U}^{s,\circ}$ admits a decreasing $\Delta_0(p)$-stable filtration by sub-$R_{\U}$-modules
\[
\mbf{D}_{\U}^{s,\circ}=\mathrm{Fil}^{0}\mbf{D}_{\U}^{s,\circ}\supset\mathrm{Fil}^{1}\mbf{D}_{\U}^{s,\circ}\supset\cdots\supset\mathrm{Fil}^{i}\mbf{D}_{\U}^{s,\circ}\supset\cdots\]
such that each quotient $\mbf{D}_{\U}^{s,\circ}/\mathrm{Fil}^{k}\mbf{D}_{\U}^{s,\circ}$
is a finite abelian group of exponent $p^{k}$, the group $K(p^{s+k})$
acts trivially on $\mbf{D}_{\U}^{s,\circ}/\mathrm{Fil}^{k}\mbf{D}_{\U}^{s,\circ}$, and
$\mbf{D}_{\U}^{s,\circ}\cong \varprojlim _k \mbf{D}_{\U}^{s,\circ}/\mathrm{Fil}^{k}\mbf{D}_{\U}^{s,\circ}$.

\item $\mbf{D}_{\U}^{s,\circ}$, with the topology induced by the submodules $({\rm Fil}^{k}\mbf{D}_{\U}^{s,\circ})_{k\geq 0}$, is a profinite flat $\Zp$-module.

\end{enumerate}

\end{prop}
\begin{proof} We will only recall how the filtrations are constructed and refer to \cite[\S 2.2]{han} for the details (note that (2) is a consequence of (1)).  The module ${\rm Fil}^{k}\mbf{D}_{\U}^{s,\circ}$ is defined to be the kernel of the natural map $$\mbf{D}_{\U}^{s,\circ} \ra \mbf{D}_{\U}^{s-1,\circ}/\mf{a}_{\U}^{k}\mbf{D}_{\U}^{s-1,\circ}, $$ 
where we recall that $\mf{a}_{\U}$ is our fixed choice of ideal of definition for the profinite topology on $R_{\U}$.
\end{proof}

Let $k\geq 2$ be an integer and let $A$ be a ring. We let $\scl_{k}(A)$ denote the module of polynomials $A[X]^{\mathrm{deg}\leq k-2}$ with left $M_{2}(A)$-action
$$ (\delta\cdot_{k}p)(X)=(d+bX)^{k-2}p\left(\frac{c+aX}{d+bX}\right),$$
and set in particular $\scl_{k}=\scl_{k}(\Qp)$ and $\scl_{k}^{\circ}=\scl_{k}(\Zp)$.
By direct calculation, the map\begin{eqnarray*}
\rho_{k}:\mbf{D}_{k}^{s,\circ} & \to & \scl_{k}^{\circ}\\
\mu & \mapsto & \int(1+Xx)^{k-2}\mu(x)=\\
 &  & =\sum_{j=0}^{k-2}\left(\begin{array}{c}
k-2\\
j\end{array}\right)\mu(x^{j})X^{j}\end{eqnarray*}
is $\Delta_{0}(p)$-equivariant.

\begin{defi} \label{def-int}
 The integration map in weight $k$, denoted $i_{k}$, is the $\Delta_0(p)$-equivariant map $i_{k}:\mbf{D}_{\U}^{s,\circ}\to\scl_{k}^{\circ}$
defined by $i_{k}=\rho_{k}\circ\sigma_{k}$.
\end{defi}

\subsection{Slope decompositions}\label{slopedecomp}

We will recall material from \cite{as} and \cite{han1} in order to define slope decompositions on the spaces of overconvergent modular symbols that we are interested in. We will mildly abuse notation by writing $\XC$ for the complex Shimura curve of level $K:=K^{p}K_{0}(p)$ (viewed as a Riemann surface). Any $K_{0}(p)$-module $M$ defines a local system on $\XC$ which we will also denote by $M$. If $M$ in addition is a $\Delta_{0}(p)$-module then we get induced Hecke actions as well. The spaces of overconvergent modular symbols that we are interested in are 
$$ H^{1}(\XC, \mbf{D}_{\U}^{s}) $$
for open weights $\U$. We will give these spaces slope decompositions using the methods of \cite{han1}. Almost everything goes through \emph{verbatim} and we will content ourselves with a brief discussion. Throughout this section, $h$ will denote a non-negative rational number. For definitions and generalities on slope decompositions that we will use we refer to \cite[\S 2.2]{jn} (we remark that the notion introduced in \cite[\S 4]{as} needs slight tweaking for the purpose of constructing the whole eigencurve, as opposed to just local pieces).

\medskip  

Let us start by recalling the general setup from \cite[\S 2.1]{han1}. We have the functorial \emph{adelic} (co)chain complexes $C_{\bullet}^{\ad}(K,-)$ and $C_{\ad}^{\bullet}(K,-)$ whose (co)homology functorially computes homology and cohomology of local systems attached to $\Delta_{0}(p)$-modules, respectively. Fix once and for all a choice of triangulation of $\XC$. This choice gives us functorial complexes $C_{\bullet}(K,-)$ and $C^{\bullet}(K,-)$, called \emph{Borel-Serre complexes}, which are chain homotopic to $C_{\bullet}^{\ad}(K,-)$ and $C_{\ad}^{\bullet}(K,-)$ respectively. We fix such a chain homotopy once and for all. The key features of the Borel-Serre complexes are that there are non-negative integers $r(i)$ such that $r(i)=0$ for $i<0$ and for $i$ sufficiently large, and such that for any $R[\Dp]$-module $M$ (where $R$ is any commutative ring), 
$$ C_{i}(K,M) \cong M^{r(i)}$$
functorially as $R$-modules (and similarly for $C^{\bullet}(K,-)$, with the same integers $r(i)$). Therefore the total complex $\bigoplus_{i}C_{i}(K,M)$ inherits properties of $M$, such as being Banach if $R$ is a $\Qp$-Banach algebra (and also orthonormalizability). In this case, we have a canonical and functorial topological duality isomorphism
$$ C^{\bullet}(K,\Hom_{R,cts}(M,P)) \cong \Hom_{R,cts}(C_{\bullet}(K,M),P) $$
where $M$ and $P$ are Banach $R$-modules. 

\medskip

We now follow \cite[\S 3.1]{han1} using the orthonormalizable $A_{\U}[\frac{1}{p}]$-module $\mbf{A}_{\U}^{s}$, where $\U=(A_{\U},\chi_{\U})$ is an open weight. Our $U_{p}$-operator is given by the double coset $K_{0}(p) \left( \begin{smallmatrix} p & 0 \\ 0 & 1 \end{smallmatrix} \right) K_{0}(p)$ and the formalism gives us lifts to our Borel-Serre complexes that we denote by $\tilde{U}$. We note that $\tilde{U} \in \End_{A_{\U}[\frac{1}{p}]}(C_{\bullet}(K,\mbf{A}_{\U}^{s}))$ is compact and we denote its Fredholm determinant by $F_{\U}(X)\in A_{\U}[\frac{1}{p}][[X]]$. The proof of \cite[Proposition 3.1.1]{han1} goes through for small weights to show that this definition is independent of $s$. Thus the existence of a slope $\leq h$-decomposition of $C_{\bullet}(K,\mbf{A}_{\U}^{s})$ is equivalent to the existence of a slope $\leq h$-factorization of $F_{\U}(X)$. If $\V$ is another open weight with $\V^{\rig} \sub \U^{\rig}$, then the relation $\mbf{A}_{\V}^{s} \cong \mbf{A}^{s}_{\U} \ctens_{A_{\U}[\frac{1}{p}]} A_{\V}[\frac{1}{p}] $ implies that the $F_{\U}$ glue to a power series $F(X) \in \oo_{\ms{W}}(\ms{W})[[X]]$. Also \cite[Proposition 3.1.2]{han1} goes through in this more general setting: the slope $\leq h$-subcomplex $C_{\bullet}(K,\mbf{A}_{\U}^{s})_{\leq h}$ of $C_{\bullet}(K,\mbf{A}_{\U}^{s})$ is independent of $s$ (if it exists) and if $\V^{\rig} \sub \U^{\rig}$ then we have a canonical isomorphism
$$ C_{\bullet}(K,\mbf{A}_{\U}^{s})_{\leq h} \otimes_{A_{\U}[\frac{1}{p}]} A_{\V}[\frac{1}{p}] \cong C_{\bullet}(K,\mbf{A}_{\V}^{s})_{\leq h}. $$

We say that $(\U, h)$ is a \emph{slope datum} if $C_{\bullet}(K,\mbf{A}_{\U}^{s})$ has a slope $\leq h$-decomposition or equivalently if $F_{\U}$ has a slope $\leq h$-factorization. We have the following version of \cite[Proposition 3.1.3]{han1}:

\begin{prop}
Assume that $(\U,h)$ is a slope datum and that $\V = (S_{\V}, \chi _{\V})$ is an open affinoid weight with $\V \sub \U^{\rig}$. Then there is a canonical isomorphism
$$ H_{\ast}(\XC,\mbf{A}_{\U}^{s})_{\leq h} \otimes_{A_{\U}[\frac{1}{p}]}S_{\V} \cong H_{\ast}(\XC,\mbf{A}_{\V}^{s})_{\leq h} $$
for any $s\geq s_{\U}$.
\end{prop}

\begin{proof}
When $\U$ is affinoid this is exactly \cite[Proposition 3.1.3]{han1}, so assume without loss of generality that $\U = (A _{\U}, \chi _{\U})$ is small. The same proof will go through if we can verify that $A_{\U}[\frac{1}{p}] \ra S_{\V}$ is flat. This is standard. One proof goes as follows: Consider the adic space $\Spa(A_{\U},A_{\U})$ and pick a rational subset $V$ in its generic fibre which contains $\V$. Then $\oo(V) \ra S_{\V}$ is flat since it corresponds to an open immersion of affinoids in rigid geometry, and $A_{\U} \ra \oo(V)$ is flat since it is a rational localization in the theory of adic spaces when the rings of definition are Noetherian.
\end{proof}

We may then give the analogue of the above proposition for $H^{\ast}(\XC,\mbf{D}_{\U}^{s})$ (cf. \cite[Proposition 3.1.5]{han1}). 

\begin{prop}\label{slopeV}
Assume that $(\U,h)$ is a slope datum. Then $C^{\bullet}(K,\mbf{D}_{\U}^{s})$, and hence $H^{\ast}(\XC,\mbf{D}_{\U}^{s})$, admit slope $\leq h$-decompositions. If furthermore $\V \sub \U^{\rig}$ is an open affinoid weight, then there are canonical isomorphisms
$$ C^{\bullet}(K,\mbf{D}_{\U}^{s})_{\leq h} \otimes_{A_{\U}[\frac{1}{p}]}S_{\V} \cong C^{\bullet}(K,\mbf{D}_{\V}^{s})_{\leq h} $$
and 
$$ H^{\ast}(\XC,\mbf{D}_{\U}^{s})_{\leq h}  \otimes_{A_{\U}[\frac{1}{p}]}S_{\V} \cong H^{\ast}(\XC,\mbf{D}_{\V}^{s})_{\leq h}. $$
\end{prop}

\begin{proof}
Using the duality $C^{\bullet}(K,\mbf{D}_{\U}^{s}) \cong \Hom_{A_{\U}[\frac{1}{p}],cts}(C_{\bullet}(K,\mbf{A}_{\U}^{s}),A_{\U}[\frac{1}{p}])$ and the flatness of $A_{\U}[\frac{1}{p}] \ra S_{\V}$ the proof of \cite[Proposition 3.1.5]{han1} goes through \emph{verbatim}.
\end{proof}

\section{Sheaves on the pro-\'etale site}\label{sec:4}

In the next two sections we will often consider Shimura curves over $\Cp$ as well as $\Qp$. Any rigid analytic variety we have defined may be base changed from $\Qp$ to $\Cp$, and we will denote this base change by a subscript $-_{\Cp}$, e.g. $\X_{\Cp}$. The space $\X_{\Cp}$ may be considered as an object of $\X_{\proet}$ using the pro-\'etale presentation $\X_{\Cp}=\varprojlim_{K}\X_{K}$, where $K/\Qp$ is finite and $\X_{K}$ denotes the base change of $\X$ to $K$. Note that the slice $\X_{\proet}/\X_{\Cp}$ is equivalent to $\X_{\Cp,\proet}$ as sites (use \cite[Proposition 3.15]{sch2} and \cite[Proposition 7.4]{perf}). We define $\Xinf := \varprojlim_{K} \mc{X}_{\infty,K}$ (with $K$ as above) where the inverse limit is taken in the category of perfectoid spaces. Equivalently, we may define it as (the perfectoid space corresponding to) the perfectoid object $\varprojlim_{K}\X_{\infty,K}$ in $\X_{\proet}$; it lives in $\X_{\proet}/\X_{\Cp} \cong \X_{\Cp,\proet}$. Similar remarks apply to $\Xw$, $\Xinfw$ et cetera.

\subsection{A handy lemma}

Let $X$ be a rigid analytic variety over $\Spa(\Qp,\Zp)$, $G$ a profinite group, and $X_{\infty}$ a perfectoid space with $X_{\infty}\ra X$ a pro-\'etale $G$-covering of $X$, i.e. $X_{\infty}\sim\varprojlim_{j} X_{j}$
where $(X_{j})_{j\in J}$ is an inverse system of rigid analytic varieties finite
\'etale over $X$ equipped with a right action of $G$, such that each
$X_{j}\to X$ is a finite \'etale $G/G_{j}$-covering for $G_{j} \sub G$
a cofinal sequence of normal open subgroups. We can and often will view $X_{\infty}$ as an object in the pro-\'etale site of $X$. Let $M$ be a profinite flat $\Zp$-module equipped with a continuous left $G$-action.
Define a sheaf on the pro-\'etale site of $X$ by
$$ (M \ctens \oo_{X})(V)=\left(M \ctens \wh{\oo}_{X}(V\times_{X}X_{\infty})\right)^{G}, $$
where $V$ is a qcqs
object in $X_{\proet}$. Here
the right action of $G$ on the tower $(X_{i})_{i}$ induces a left
action on $\wh{\oo}_{X}(V_{\infty})$, and the action
indicated is the diagonal left action $g\cdot(m\otimes f)=gm\otimes g^{\ast}f$.

\begin{lemm}\label{handy}
Given any presentation $M = \varprojlim_{i} M_i$, $M_{i}=M/I_{i}$ as in Definition \ref{defi: ctens} with each $I_i$ preserved by $G$, the sheaf $M \ctens \oo_{X}$
coincides with the sheaf
$$ \wh{\mb{Q}}_{p}\otimes_{\wh{\mb{Z}}_{p}}\varprojlim _i \left(\nu^{\ast}\wt{M}_{i}\otimes_{\Zp}\oo_{X}^{+}\right), $$
where $\wt{M}_{i}$ is the locally constant sheaf on $X_{\et}$
associated with $M_{i}$, and we recall that $\nu\, :\, X_{\proet} \ra X_{\et}$ is the canonical morphism of sites.
\end{lemm}

\begin{proof} We may check this on affinoid perfectoid $V\in X_{\proet}$ since these form a basis of $X_{\proet}$ (and are qcqs). Set $V_{\infty}=V\times_{X}X_{\infty}$, this is also affinoid perfectoid (see e.g. the proof of \cite[Lemma 4.6]{sch2}). Each $M_{i}$ is finite abelian group of $p$-power exponent. The key observation is then that there is an almost equality
$$ M_{i} \otimes_{\Zp} \wh{\oo}_{X}^{+}(V_{\infty}) =^{a} (\nu^{\ast}\wt{M}_{i} \otimes_{\Zp} \oo_{X}^{+})(V_{\infty}) $$
using \cite[Lemma 4.10]{sch2}. Taking $G$-invariants we see that, since $V_{\infty}\ra V$ is a $G$-cover, 
$$ \left( M_{i} \otimes_{\Zp} \wh{\oo}_{X}^{+}(V_{\infty}) \right)^{G} =^{a} (\nu^{\ast}\wt{M}_{i} \otimes_{\Zp} \oo_{X}^{+})(V) $$
by Lemma \ref{newinv}. Now take inverse limits and invert $p$, and use that these operations commute with taking $G$-invariants.
\end{proof}

\subsection{Sheaves of overconvergent distributions}

We introduce certain pro-\'etale sheaves that compute overconvergent modular symbols and carry a Galois action. When $\mb{L}$ is a locally constant constructible sheaf on the \'etale site of a rigid analytic variety, we may pull it back to a sheaf $\nu^{\ast}\mb{L}$ on the pro-\'etale site. We will often abuse notation and denote $\nu^{\ast}\mb{L}$ by $\mb{L}$ as well; by \cite[Corollary 3.17]{sch2} there is little harm in this. For the rest of this section we will let $\U=(R_{\U},\chi_{\U})$ be a \emph{small} weight. Note that the $K_{0}(p)$-modules $\mbf{D}_{\U}^{s,\circ}/{\rm Fil}^{k}\mbf{D}_{\U}^{s,\circ}$ define locally constant constructible sheaves on the \'etale site $\X_{\et}$.

\begin{defi} Let $\U$ be a small weight and let $s\geq 1+s_{\U}$. Set \textbf{\begin{eqnarray*}
\mbf{V}_{\U}^{s,\circ} & = & \varprojlim _k H^{1}_{\et}(\X_{\Cp},\mbf{D}_{\U}^{s,\circ}/{\rm Fil}^{k}\mbf{D}_{\U}^{s,\circ})\\
 & = & \varprojlim _k H^{1}_{\proet}(\X_{\Cp},\mbf{D}_{\U}^{s,\circ}/{\rm Fil}^{k}\mbf{D}_{\U}^{s,\circ})\end{eqnarray*}}
and
$$ \mbf{V}_{\U,\oo_{\Cp}}^{s,\circ}=\varprojlim _k \left(H^{1}_{\proet}(\X_{\Cp},\mbf{D}_{\U}^{s,\circ}/{\rm Fil}^{k}\mbf{D}_{\U}^{s,\circ})\otimes_{\Zp}\oo_{\Cp}\right). $$
We also define $\mbf{V}_{\U}^{s}=\mbf{V}_{\U}^{s,\circ}[\frac{1}{p}]$ and
$\mbf{V}_{\U,\Cp}^{s}=\mbf{V}_{\U,\oo_{\Cp}}^{s,\circ}[\frac{1}{p}]$.
\end{defi}

These objects carry natural global Galois actions. We have the following comparison:

\begin{prop}\label{prop:artin}
There is a canonical Hecke-equivariant isomorphism $\mbf{V}_{\U}^{s}\cong H^{1}(\XC,\mbf{D}_{\U}^{s})$.
\end{prop}

\begin{proof}
Artin's comparison theorem between \'etale and singular cohomology gives us canonical Hecke-equivariant isomorphisms 
$$H^{1}_{\et}(\X_{\Cp},\mbf{D}_{\U}^{s,\circ}/{\rm Fil}^{k}\mbf{D}_{\U}^{s,\circ})\cong H^{1}(\XC,\mbf{D}_{\U}^{s,\circ}/{\rm Fil}^{k}\mbf{D}_{\U}^{s,\circ}) $$
for all $k$. Now take inverse limits and invert $p$. For this to do the job we need to be able to commute the inverse limit and the $H^{1}$ on the right hand side. But we may do this since the relevant higher inverse limits vanish, by finiteness of $\mbf{D}_{\U}^{s,\circ}/{\rm Fil}^{k}\mbf{D}_{\U}^{s,\circ}$ and by applying \cite[Lemma 3.18]{sch2}.
\end{proof}

\medskip

We now introduce a sheaf on $\X_{\proet}$ that computes overconvergent modular symbols.

\begin{prop}\label{prop:diam}
Let $\U$ be a small weight. There is a canonical Hecke- and Galois-equivariant isomorphism
$$ \mbf{V}_{\U,\Cp}^{s}\cong H^{1}_{\proet}(\X_{\Cp},\oo \mbf{D}_{\U}^{s}) $$
where $\oo\mbf{D}_{\U}^{s}$ is the sheaf on $\X_{\proet}$
defined by 
$$ \oo\mbf{D}_{\U}^{s} := \left( \varprojlim _k \left( (\mbf{D}_{\U}^{s,\circ}/{\rm Fil}^{k}\mbf{D}_{\U}^{s,\circ})\otimes_{\Zp}\oo_{\mc{X}_{\Cp}}^{+}\right) \right) \left[\frac{1}{p}\right]. $$
\end{prop}

\begin{proof} By \cite[Theorem 5.1]{sch2}), we have an almost isomorphism
$$ H^{1}_{\proet}(\X_{\Cp}, \mbf{D}_{\U}^{s,\circ}/{\rm Fil}^{k}\mbf{D}_{\U}^{s,\circ})\otimes_{\Zp}\oo_{\Cp} \overset{a}{\cong} H^{1}_{\proet}(\X_{\Cp},(\mbf{D}_{\U}^{s,\circ}/{\rm Fil}^{k}\mbf{D}_{\U}^{s,\circ}) \otimes_{\Zp}\oo_{\mc{X}_{\Cp}}^{+}) $$
of $\oo_{\Cp}$-modules. Passing to the inverse
limit over $k$, note that on one hand, the higher inverse limits
of any $H^{i}_{\proet}(\X_{\Cp},\mbf{D}_{\U}^{s,\circ}/{\rm Fil}^{k}\mbf{D}_{\U}^{s,\circ})\otimes_{\Zp}\oo_{\Cp}$
vanish since $H^{i}_{\proet}(\X_{\Cp},\mbf{D}_{\U}^{s,\circ}/{\rm Fil}^{k}\mbf{D}_{\U}^{s,\circ})$ is finite (by the finiteness properties of ${\rm Fil}^{\bullet}\mbf{D}_{\U}^{s,\circ}$), thus $H^{i}_{\proet}(\X_{\Cp},\mbf{D}_{\U}^{s,\circ}/{\rm Fil}^{k}\mbf{D}_{\U}^{s,\circ})\otimes_{\Zp}\oo_{\Cp}$ satisfy the Mittag-Leffler condition. On the other hand the higher inverse limit 
$$ R^{i}\varprojlim _k \, ((\mbf{D}_{\U}^{s,\circ}/{\rm Fil}^{k}\mbf{D}_{\U}^{s,\circ})\otimes_{\Zp}\oo_{X}^{+}) $$
vanishes for $i>0$. This follows from \cite[Lemma 3.18]{sch2} upon noting that (in the notation of Lemma \ref{handy}, with $X=\X_{\Cp}$, $X_{\infty}=\X_{\infty,\Cp}$ and $G=K_{0}(p)$) objects of
the form $V_{\infty}$ give a basis of $\X_{\Cp,\proet}$
satisfying the hypotheses of that lemma. Thus we may commute the inverse limit with taking cohomology on the right hand side. Inverting $p$ we get the result.
\end{proof}

Let $V=\varprojlim _i V_{i}\ra \X$ be a pro-\'etale
presentation of an affinoid perfectoid object of $\X_{\proet}$,
and let $V_{\infty}=V \times_{\X} \Xinf$.
This is still affinoid perfectoid and $V_{\infty}\to V$ is pro-\'etale. The following is immediate from Lemma \ref{handy}:

\begin{lemm}\label{lem-explicitD} 
Let $\U$ be a small weight. The sheaf $\oo\mbf{D}_{\U}^{s}$ admits
the following explicit description on qcqs $V\in \X_{\proet}$:
$$ \oo\mbf{D}_{\U}^{s}(V)= \left(\mbf{D}_{\U}^{s,\circ}\ctens \wh{\oo}_{\X}(V_{\infty})\right)^{K_{0}(p)}. $$
\end{lemm}
We remark that the sheaf $V_{\infty} \mapsto \mbf{D}_{\U}^{s,\circ}\ctens \wh{\oo}_{\X}(V_{\infty})$ on $\X_{\proet}/\Xinf$ is $\Delta_{0}(p)$-equivariant and the induced Hecke action on $\oo\mbf{D}_{\U}^{s}$ agrees with the natural action by correspondences.

\subsection{Completed sheaves of overconvergent modular forms}

Let $q:\Xinf \to \X$ be the natural projection as before. Recall that for any weight $\U$ and any $w\geq 1+s_{\U}$, we have the sheaf $\omega_{\U,w}^\dagger$ of
overconvergent modular forms of weight $\U$ on the analytic site of
$\Xw$, whose sections over a rational subset $U \sub \Xw$
are given by 
$$ \omega_{\U,w}^\dagger (U)=\left\{ f \in \oo_{\Xinf}(U_{\infty}) \ctens A_{\U} \mid  \gamma^{\ast}f=\chi_{\U}(b\mf{z}+d)^{-1}f \; \forall \gamma \in K_0(p)\right\} . $$
We may similarly define a sheaf $\omega_{\U,w,\Cp}^{\dg}$ on the analytic site of $\X_{w,\Cp}$ by
$$ \omega_{\U,w,\Cp}^\dagger (U)=\left\{ f \in \oo_{\X_{\infty,\Cp}}(U_{\infty}) \ctens A_{\U} \mid  \gamma^{\ast}f=\chi_{\U}(b\mf{z}+d)^{-1}f \; \forall \gamma \in K_0(p)\right\} $$
with $U\sub \X_{w,\Cp}$ quasi-compact.
We now define completed versions of these sheaves.

\begin{defi} We define a sheaf $\wh{\omega}_{\U,w}^{\dg}$
on $\X_{w,\proet}$ by 
$$\wh{\omega}_{\U,w}^{\dg}(V)=\left\{ f \in \wh{\oo}_{\Xw}(V_{\infty}) \ctens A_{\U} \mid  \gamma^{\ast}f=\chi_{\U}(b\mf{z}+d)^{-1}f \; \forall \gamma \in K_0(p)\right\} ,$$
where $V\in \X_{w,\proet}$ is qcqs and $V_{\infty}=V\times_{\Xw} \Xinfw$. We let $\wh{\omega}_{\U,w,\Cp}^{\dg}$ denote the restriction of $\wh{\omega}_{\U,w}^{\dg}$ to the slice $\X_{w,\proet}/\X_{w,\Cp}\cong \X_{w,\Cp,\proet}$.
\end{defi}

Let us write $\eta$ for the canonical morphism of sites from the \'etale site to the analytic site, and put $\lambda = \eta \circ \nu$.

\begin{prop}\label{prop:omega} There are canonical isomorphisms
$$ \wh{\omega}_{\U,w}^{\dg} \cong \lambda^{\ast}\omega_{\U,w}^{\dg} \otimes_{\oo_{\Xw} \ctens A_{\U}}(\wh{\oo}_{\Xw} \ctens A_{\U}) $$
and
$$ \wh{\omega}_{\U,w,\Cp}^{\dg} \cong \lambda^{\ast}\omega_{\U,w,\Cp}^{\dg} \otimes_{\oo_{\X_{w,\Cp}} \ctens A_{\U}}(\wh{\oo}_{\X_{w,\Cp}} \ctens A_{\U}). $$
This has the following consequences:
\begin{enumerate}
\item We have a canonical isomorphism
$$ \wh{\omega}_{\U,w}^{\dg} \cong \nu^{\ast}\omega_{\U,w}^{\dg} \otimes_{\oo_{\Xw} \ctens A_{\U}}(\wh{\oo}_{\Xw} \ctens A_{\U}) $$
and similarly for the $\Cp$-sheaves, where we abuse the notation and write $\omega_{\U,w}^{\dg}$ also for the \'etale sheaf $\eta^{\ast}\omega_{\U,w}^{\dg} \otimes_{\oo_{\X_{w,an}} \ctens A_{\U}} (\oo_{\X_{w,\et}} \ctens A_{\U})$ (where we have written out "an" and "et" for clarity).

\item $\omega_{\U,w,\Cp}^{\dg}$ is the base change of $\omega_{\U,w}^{\dg}$ from $\Xw$ to $\X_{w,\Cp}$ (see Lemma \ref{lemm: bc}) and hence $\omega_{\U,w,\Cp}^{\dg}={\rm Loc}(\omega_{\U,w,\Cp}^{\dg}(\X_{w,\Cp}))$.

\item We have a canonical isomorphism $\nu_{\ast}\wh{\omega}_{\U,w,\Cp}^{\dg}\cong \omega_{\U,w,\Cp}^{\dg}$ and
$$R^1\nu_{\ast}\wh{\omega}_{\U,w,\Cp}^\dagger=\omega_{\U,w,\Cp}^\dagger \otimes_{\oo_{\X_{w,\Cp}}}\Omega^1_{\X_{w,\Cp}}(-1).$$ 
Furthermore, $R^i\nu_{\ast}\wh{\omega}_{\U,w,\Cp}^\dagger$ vanishes for $i \geq 2$.

\end{enumerate}
\end{prop}

\begin{proof} The isomorphisms are proven by repeating much of the proof of Theorem \ref{theo:rankone} (including the preliminary lemmas). We focus on the $\Qp$-sheaves; the case of the $\Cp$-sheaves is identical. Let $V\in \X_{w,\proet}$ be an affinoid perfectoid with image $U$ in $\Xw$ and assume without loss of generality that $U$ is a rational subset of $\Xw$ and that $\omega|_{U}$ is trivial. We use the notation of the proof of Theorem \ref{theo:rankone} freely. Arguing in the same way we obtain isomorphisms
$$ \wh{\omega}_{\U,w}^{\dg}(V) \cong (\wh{\oo}_{\Xw}(V_{n}) \ctens A_{\U})^{G_{n}} $$
and
$$ (\nu^{\ast}\omega_{\U,w}^{\dg}\otimes_{\oo_{\Xw}\ctens A_{\U}} (\wh{\oo}_{\Xw}\ctens A_{\U}))(V) \cong (\wh{\oo}_{\Xw}(U_{n}) \ctens A_{\U})^{G_{n}} \otimes_{\wh{\oo}_{\Xw}(U) \ctens A_{\U}}(\wh{\oo}_{\Xw}(V) \ctens A_{\U}) $$
where $V_{n}= U_{n} \times_{U} V$ and $G_{n}$ acts by the twisted action $f_{0} \mapsto \chi_{\U}(j_{U,n}(\gamma))\gamma^{\ast}f_{0}$. Applying Lemma \ref{lemm: commute} twice one obtains a $G_{n}$-equivariant isomorphism
$$ \wh{\oo}_{\Xw}(V_{n}) \ctens A_{\U} \cong (\wh{\oo}_{\Xw}(U_{n}) \ctens A_{\U}) \otimes_{\wh{\oo}_{\Xw}(U) \ctens A_{\U}}(\wh{\oo}_{\Xw}(V) \ctens A_{\U}). $$
Taking invariants we obtain we desired isomorphism.

\medskip

\noindent Assertion (1) then follows by transitivity of pullbacks. To prove (2), first evaluate the isomorphism for $\wh{\omega}_{\U,w}^{\dg}$ on $U\sub \X_{w,\Cp}$ to see that $\wh{\omega}_{\U,w,\Cp}^{\dg}$ restricted to the analytic site of $\X_{w,\Cp}$ is the base change of $\omega_{\U,w}^{\dg}$, and then restrict the second isomorphism to the analytic site of $\X_{w,\Cp}$ to get the first part (2). The second then follows from Lemma \ref{lemm: bc}. 

\medskip

\noindent To prove (3) we apply the projection formula to the isomorphism in (1) to get
$$R^i\nu_{\ast}\wh{\omega}_{\U,w,\Cp}^\dagger \cong \omega_{\U,w,\Cp}^\dagger \otimes_{(\oo_{\mc{X}_{w,\Cp}} \ctens A_{\U})} R^i \nu_{\ast} (\wh{\oo}_{\mc{X}_{w,\Cp}} \ctens A_{\U}),$$
and the result then follows from Corollary \ref{coro:comp-r}.
\end{proof}

\begin{coro} \label{cor:leray} The Leray spectral sequence 
$$H^i_{\et}(\X_{w,\Cp},R^j\nu_{\ast}\wh{\omega}_{\U,w,\Cp}^{\dg}) \Rightarrow H^{i+j}_{\proet}(\X_{w,\Cp},\wh{\omega}_{\U,w,\Cp}^{\dg})$$ 
induces a canonical Hecke- and Galois-equivariant isomorphism  
$$ H^{1}_{\proet}(\X_{w,\Cp},\wh{\omega}_{\U,w,\Cp}^{\dg}) \cong H^0(\mc{X}_{w,\Cp}, \omega_{\U,w,\Cp}^{\dg} \otimes_{\oo_{\mc{X}_{w,\Cp}}}\Omega^1_{\mc{X}_{w,\Cp}})(-1)$$
\end{coro}

\begin{proof}
By Proposition \ref{prop:omega} $R^j\nu_{\ast}\wh{\omega}_{\U,w,\Cp}^{\dg}$ is a locally projective $\oo_{\X_{w,\Cp}} \ctens A_{\U}$-module for all $j$ and hence has no higher cohomology by (the \'etale version of) Proposition \ref{prop:thmB}. Therefore the Leray spectral sequence degenerates at the $E_{2}$-page, giving us isomorphisms
$$ H_{\et}^{0}(\X_{w,\Cp}, R^{j}\nu_{\ast}\wh{\omega}_{\U,w,\Cp}^{\dg}) \cong H^{j}_{\proet}(\X_{w,\Cp}, \wh{\omega}_{\U,w,\Cp}^\dagger) $$
for all $j$. Now apply Proposition \ref{prop:omega} again.  

\medskip

For Hecke-equivariance we focus on the Hecke operators away from $p$; the Hecke equivariance at $p$ follows by the same arguments but is notationally slightly different. We will only check compatibility for the trace maps (compatibility for pullbacks follow from functoriality), and we put ourselves in the situation of Proposition \ref{prop:hecke}. By the first isomorphisms in Proposition \ref{prop:omega}, the trace map
$t$ of Proposition \ref{prop:hecke}(2) induces a trace map
\[
\widehat{t}:f_{\proet,\ast}\widehat{\omega}_{\mathcal{U},w,\mathbb{C}_{p},K_{2}^{p}}^{\dagger}\to\widehat{\omega}_{\mathcal{U},w,\mathbb{C}_{p},K_{1}^{p}}^{\dagger}.
\]
We need to check that the diagram
\[
\xymatrix{H_{\proet}^{1}(\mathcal{X}_{w,\mathbb{C}_{p},K_{2}^{p}},\widehat{\omega}_{\mathcal{U},w,\mathbb{C}_{p},K_{2}^{p}}^{\dagger})\ar[r]\ar[d]_{\widehat{t}} & H^{0}(\mathcal{X}_{w,\mathbb{C}_{p},K_{2}^{p}},\omega_{\mathcal{U},w,\mathbb{C}_{p},K_{2}^{p}}^{\dagger}\otimes\Omega^{1})(-1)\ar[d]\\
H_{\proet}^{1}(\mathcal{X}_{w,\mathbb{C}_{p},K_{1}^{p}},\widehat{\omega}_{\mathcal{U},w,\mathbb{C}_{p},K_{1}^{p}}^{\dagger})\ar[r] & H^{0}(\mathcal{X}_{w,\mathbb{C}_{p},K_{1}^{p}},\omega_{\mathcal{U},w,\mathbb{C}_{p},K_{1}^{p}}^{\dagger}\otimes\Omega^{1})(-1)
}
\]
commutes, where the left hand vertical arrow is induced by $\hat{t}$
as in Lemma \ref{proetaletrace}(2) and the right hand vertical arrow is induced by $t$
in the obvious way. We do this in two steps. For the first step, we
note that Lemma \ref{proetaletrace}(3) gives a commutative diagram
\[
\xymatrix{H_{\proet}^{1}(\mathcal{X}_{w,\mathbb{C}_{p},K_{2}^{p}},\widehat{\omega}_{\mathcal{U},w,\mathbb{C}_{p},K_{2}^{p}}^{\dagger})\ar[r]\ar[d]_{\hat{t}} & H^{0}(\mathcal{X}_{w,\mathbb{C}_{p},K_{2}^{p}},R^{1}\nu_{\ast}\widehat{\omega}_{\mathcal{U},w,\mathbb{C}_{p},K_{2}^{p}}^{\dagger})\ar[d]^{R^{1}\nu_{\ast}\hat{t}}\\
H_{\proet}^{1}(\mathcal{X}_{w,\mathbb{C}_{p},K_{1}^{p}},\widehat{\omega}_{\mathcal{U},w,\mathbb{C}_{p},K_{1}^{p}}^{\dagger})\ar[r] & H^{0}(\mathcal{X}_{w,\mathbb{C}_{p},K_{1}^{p}},R^{1}\nu_{\ast}\widehat{\omega}_{\mathcal{U},w,\mathbb{C}_{p},K_{1}^{p}}^{\dagger}).
}
\]
For the second step, we claim that the diagram
\[
\xymatrix{R^{1}\nu_{\ast}\widehat{\omega}_{\mathcal{U},w,\mathbb{C}_{p},K_{2}^{p}}^{\dagger}\ar[r]\ar[d]_{R^{1}\nu_{\ast}\hat{t}} & \omega_{\mathcal{U},w,\mathbb{C}_{p},K_{2}^{p}}^{\dagger}\otimes\Omega^{1}(-1)\ar[d]^{t}\\
R^{1}\nu_{\ast}\widehat{\omega}_{\mathcal{U},w,\mathbb{C}_{p},K_{1}^{p}}^{\dagger}\ar[r] & \omega_{\mathcal{U},w,\mathbb{C}_{p},K_{1}^{p}}^{\dagger}\otimes\Omega^{1}(-1)
}
\]
commutes, where the horizontal arrows are the isomorphisms of Proposition
\ref{prop:omega}(3); this follows upon combining the first isomorphism of Proposition
\ref{prop:omega}, the pullback-pushforward adjunction of $\nu$, and Lemma \ref{differentialtrace}.
\end{proof}

\begin{lemm}\label{proetaletrace}Let $f: X \to Y$ be a finite \'etale morphism of rigid analytic varieties.
\begin{enumerate}

\item The functor $f_{\proet,\ast}$ is exact.

\item Suppose we are given abelian sheaves $\mc{F}$ and $\mc{G}$ on $X_{\proet}$ and $Y_{\proet}$, respectively, together with a ``trace" map $t: f_{\proet, \ast} \mc{F} \to \mc{G}$ of abelian sheaves on $Y_{\proet}$.  Then $t$ induces a canonical map $t: H^n_{\proet}(X,\mc{F}) \to H^n_{\proet}(Y,\mc{G})$ together with derived trace maps $R^i\nu_{\ast} t: f_{\et, \ast}R^i\nu_{X,\ast}\mc{F} \to R^i\nu_{Y,\ast}\mc{G}$ of abelian sheaves on $Y_{\et}$. 
 
\item The maps in (2) are compatible with the Grothendieck spectral sequences for $R\Gamma_{\proet} = R\Gamma_{\et} \circ R\nu_{\ast}$, i.e. the maps $t$ and $R^i \nu_{\ast} t$ fit into compatible morphisms
\[
\xymatrix{H_{\et}^{i}(X,R^{j}\nu_{X,\ast}\mathcal{F})\ar@{=>}[r]\ar[d]_{R^{j}\nu_{\ast}t} & H_{\proet}^{i+j}(X,\mathcal{F})\ar[d]^{t}\\
H_{\et}^{i}(Y,R^{j}\nu_{Y,\ast}\mathcal{G})\ar@{=>}[r] & H_{\proet}^{i+j}(Y,\mathcal{G})}
\]
of pro-\'etale cohomology groups and of the spectral sequences computing them.
\end{enumerate}
\end{lemm}

\begin{proof} To prove right-exactness of $f_{\proet,\ast}$ we need to check that it preserves surjections. For this we may work \'etale locally on $Y$, so we may pick an \'etale cover which splits $f$ and the assertion is then trivial. 

\medskip

\noindent For (2) and (3), we argue as follows.  By (1), the Leray spectral sequence gives an isomorphism $H^n_{\proet}(X,\mc{F}) \cong H^n_{\proet}(Y,f_{\proet,\ast}\mc{F})$, so composing this with the evident map $H^n_{\proet}(Y,f_{\proet,\ast}\mc{F})\to H^n_{\proet}(Y,\mc{G})$ gives the claimed map on pro-\'etale cohomology.  On the other hand, a direct calculation gives a natural isomorphism $f_{\et,\ast} \nu_{X,\ast} \cong \nu_{Y,\ast}f_{\proet,\ast}$ as functors $\mathrm{Sh}(X_{\proet}) \to \mathrm{Sh}(Y_{\et})$. Indeed, both functors send a sheaf $\mc{F}$ to the sheaf associated with the presheaf $U \mapsto \mc{F}(U \times_{Y} X)$ (where $U \times_{Y} X$ is regarded as an element of $X_{\proet}$). Since all these functors preserve injectives, we may pass to total derived functors, getting $Rf_{\et,\ast} R\nu_{X,\ast} \cong R\nu_{Y,\ast} Rf_{\proet,\ast}$ as functors $D^{+}(\mathrm{Sh}(X_{\proet})) \to D^+(\mathrm{Sh}(Y_{\et}))$. Since $f_{\proet,\ast}$ and $f_{\et,\ast}$ are both exact functors, this becomes $f_{\et,\ast} R \nu_{X,\ast} \cong R \nu_{Y,\ast}f_{\proet,\ast}$.  Applying this to $\mc{F}$ and composing with the evident map $R \nu_{Y,\ast}f_{\proet,\ast} \mc{F} \to R \nu_{Y,\ast}\mc{G}$ induced by $t$ gives a map $R\nu_{\ast} t: f_{\et, \ast}R\nu_{X,\ast}\mc{F} \to R\nu_{Y,\ast}\mc{G}$, and we obtain the maps $R^i\nu_{\ast} t$ upon passing to cohomology sheaves.  Finally, (3) follows by applying $R\Gamma_{Y,\et}$ to the map $R\nu_{\ast}t$, in combination with the isomorphism $R\Gamma_{X, \et} R\nu_{X,\ast} \cong R\Gamma_{Y,\et} f_{\et,\ast} R\nu_{X,\ast}$.
\end{proof}

\begin{lemm}\label{differentialtrace}Maintain the setup of the previous lemma, and suppose $X$ and $Y$ are smooth and are defined over a complete algebraically closed extension $C / \Qp$. Take $\mc{F}=\wh{\mc{O}}_X$ and $\mc{G}=\wh{\mc{O}}_Y$, and let $t$ be the natural $\wh{\mc{O}}_Y$-linear trace map $\mathrm{\wh{tr}}: f_{\proet,\ast} \wh{\mc{O}}_X \to \wh{\mc{O}}_Y$. Then $R^i\nu_{\ast} \mathrm{\wh{tr}}: f_{\et, \ast}R^i\nu_{X,\ast}\wh{\mc{O}}_X \to R^i\nu_{Y,\ast}\wh{\mc{O}}_Y$ coincides with the composite map 
\[ f_{\et, \ast}R^i\nu_{X,\ast}\wh{\mc{O}}_X \cong f_{\et,\ast} \Omega^i_{X}(-i) \cong \Omega^i_{Y}(-i) \otimes_{\mc{O}_Y} f_{\et,\ast}\mc{O}_X \overset{1 \otimes \mathrm{tr}}{\longrightarrow} \Omega^i_{Y}(-i) \cong R^i\nu_{Y,\ast}\wh{\mc{O}}_Y.\]
\end{lemm}
\begin{proof}This follows from a careful reading of the proof of Lemma 3.24 of \cite{sch1}. Let us give a very brief outline of the argument. One considers a short exact sequence $S_X$
$$0 \ra \hat{\mb{Z}}_p(1) \ra \varprojlim _{\times p } \mc{O}_X ^{\times} \ra \mc{O} _X ^{\times} \ra 0$$
as in Lemma 3.24 of \cite{sch1}. We have also an analogous short exact sequence $S_Y$ on $Y$ and a norm map $f_{\proet,\ast} S_X \ra S_Y$ on $Y_{\proet}$. To both $S_X$ and $S_Y$ we can associate commutative diagrams as in \cite[Lemma 3.24]{sch1} and one shows that they fit together into a commutative cube. Chasing through this cube gives us a lemma for $i=1$. The general case follows from this as in \cite[Proposition 3.23]{sch1} by taking exterior products.
\end{proof}

\subsection{The overconvergent Eichler-Shimura map}

Working with the pro-\'etale site of $\Xw$ has the advantage of allowing us to define the overconvergent Eichler-Shimura map at the level of sheaves explicitly, using the cover $\Xinfw$ and the fundamental period $\mf{z}$. 

\medskip

Fix a \emph{small} weight $\U$. The following proposition is the key step in the construction of the overconvergent Eichler-Shimura map.

\begin{prop}\label{prop:delta} Let $w\geq 1+s_{\U}$ and $s\geq 1+s_{\U}$. Let $V_{\infty} \in \X_{w,\proet}/\Xinfw$
be qcqs. We have a map 
$$ \beta_{\U}:\mbf{D}_{\U}^{s,\circ}\ctens \wh{\oo}_{\Xw}(V_{\infty}) \ra R_{\U} \ctens \wh{\oo}_{\Xw}(V_{\infty}) $$
defined on pure tensors by $\beta_{\U}:\mu \otimes f \mapsto \mu(\chi_{\U}(1+\mf{z}x))f$. Here we use Proposition \ref{prop-extend}, with $B=\wh{\oo}_{\Xw}(V_{\infty})$, to be able to apply $\chi_{\U}$ to $1+\mf{z}x$ for any $x\in \Zp$. This defines a morphism of sheaves on $\X_{w,\proet}/\Xinfw$. $\beta_{\U}$ satisfies the equivariance relation
$$ \beta _{\U}(\gamma^{\ast} h)=\chi_{\U}(b\mf{z}+d)\gamma^{\ast}\beta _{\U}(h) $$
for any $\gamma \in \Delta_{0}(p)$, $V_{\infty}\in \X_{w,\proet}$ qcqs and $h\in \mbf{D}_{\U}^{s,\circ} \ctens \wh{\oo}_{\Xw}(\gamma^{-1}V_{\infty})$.
In particular, pushing forward to $\X_{w,\proet}$ and passing to $K_{0}(p)$-invariants, $\beta_{\U}$ induces an $R_{\U}$-linear and Hecke equivariant map 
$$ \delta_{\U}:\oo\mbf{D}_{\U}^{s} \ra \wh{\omega}_{\U,w}^{\dg} $$
of abelian sheaves on $\X_{w,\proet}$.
\end{prop}

\noindent Here we are interpreting $\chi_{\U}(1+\mf{z}x)$
as an element of $\mbf{A}_{\U}^{s,\circ}\ctens \wh{\oo}_{\Xw}(V_{\infty})$,
and extending $\mu$ by linearity and continuity to an $R_{\U}\ctens \wh{\oo}_{\Xw}(V_{\infty})$-linear
map $\mbf{A}_{\U}^{s,\circ}\ctens \wh{\oo}_{\Xw}(V_{\infty}) \ra R_{\U} \ctens \wh{\oo}_{\Xw}(V_{\infty})$.

\begin{proof} First we check that $\beta_{\U}$ is well defined, i.e. that the description on pure tensors extends by continuity. To do this, note that the description on tensors define compatible maps 
$$ (\mbf{D}_{\U}^{s-1,\circ}/\mf{a}_{\U}^{k}\mbf{D}_{\U}^{s-1,\circ}) \otimes_{\Zp} \wh{\oo}_{\Xw}(V_{\infty}) \ra (R_{\U}/\mf{a}_{\U}^{k}) \otimes_{\Zp} \wh{\oo}_{\Xw}(V_{\infty}) $$
for all $k$. By the definition of the filtrations we then get compatible maps
$$ (\mbf{D}_{\U}^{s,\circ}/{\rm Fil}^{k}\mbf{D}_{\U}^{s,\circ}) \otimes_{\Zp} \wh{\oo}_{\Xw}(V_{\infty}) \ra (R_{\U}/\mf{a}_{\U}^{k}) \otimes_{\Zp} \wh{\oo}_{\Xw}(V_{\infty}). $$
Taking limits and inverting $p$ we get the desired map. It is then enough to check equivariance on pure tensors $h=\mu\otimes f$. We compute
\begin{eqnarray*}
\beta_{\U}(\gamma h) & = & (\gamma\cdot_{\U}\mu)(\chi_{\U}(1+\mf{z}x))\gamma^{\ast}f\\
 & = & \mu\left(\chi_{\U}(cx+d)\chi_{\U}(1+\mf{z}\frac{ax+b}{cx+d})\right)\gamma^{\ast}f\\
 & = & \mu\left(\chi_{\U}(cx+d+\mf{z}(ax+b))\right)\gamma^{\ast}f\\
 & = & \mu\left(\chi_{\U}(b\mf{z}+d)\chi_{\U}(1+x\frac{a\mf{z}+c}{b\mf{z}+d})\right)\gamma^{\ast}f\\
 & = & \chi_{\U}(b\mf{z}+d)\mu\left(\chi_{\U}(1+x\frac{a\mf{z}+c}{b\mf{z}+d})\right)\gamma^{\ast}f\\
 & = & \chi_{\U}(b\mf{z}+d)\gamma^{\ast}\left(\mu\left(\chi_{\U}(1+\mf{z}x)\right)f\right)\\
 & = & \chi_{\U}(b\mf{z}+d)\gamma^{\ast}\beta _{\U}(h)\end{eqnarray*}
and arrive at the desired conclusion. We leave the Hecke equivariance away from $p$ to the reader.
\end{proof}

The map $\delta_{\U}$ is our version of the map $\delta_{k}^{\vee}(w)$ from \cite[\S 4.2]{oes}. We may then define the overconvergent Eichler-Shimura map at the level of spaces in the same way as in \cite{oes}:

\begin{defi}
Let $\U$ be a small weight. The weight $\U$ overconvergent Eichler-Shimura map $ES _{\U}$ is given as the composite
\begin{multline*}
\mbf{V}_{\U,\Cp}^{s} \cong H^1_{\proet}(\mc{X}_{\Cp},\oo\mbf{D}_{\U}^{s}) \overset{\mathrm{res}}{\ra} H^1_{\proet}(\mc{X}_{w,\Cp},\oo\mbf{D}_{\U}^{s})\overset{\delta_{\U}}{\ra}H^{1}_{\proet}(\mc{X}_{w,\Cp},\hat{\omega}_{\U,w,\Cp}^\dagger) \cong \\
\cong H^{0}(\mc{X}_{w,\Cp},\omega _{\U,w,\Cp}^\dagger \otimes \Omega_{\mc{X}_{w,\Cp}}^{1})(-1)=:\mc{M}^{\dg,w}_{\U,\Cp}(-1)
\end{multline*}
where the first map is induced by restricting cohomology classes along the inclusion $\Xw \sub \mc{X}$. $ES_{\U}$ is Hecke- and $G_{\Qp}$-equivariant by Propositions \ref{prop:artin}, \ref{prop:diam}, \ref{prop:delta} and Corollary \ref{cor:leray}.
\end{defi}

The above definition gives also a map $ES_{\kappa}$ for any individual point $\kappa \in \mc{W}(\overline{\mb{Q}}_{p})$, since any such weight is both small and affinoid. Note also that the construction of $ES_{\U}$ is functorial in $\U$: if we have a morphism $\U \ra \U^{\prime}$ compatible with the characters, then we have a commutative diagram 
\[
\xymatrix{ \mbf{V}_{\U^{\prime},\Cp}^{s} \ar[r]^{ES _{\U^{\prime}}} \ar[d] & \mc{M}^{\dg,w}_{\U^{\prime},\Cp}(-1) \ar[d] \\
\mbf{V}_{\U,\Cp}^{s} \ar[r]^{ES _{\U}} & \mc{M}^{\dg,w}_{\U,\Cp}(-1)
}
\]
where the vertical maps are the natural ones. In particular, when $\U$ is a small open weight and $\kappa\in \U^{\rig}$, we have a commutative diagram
\[
\xymatrix{ \mbf{V}_{\U,\Cp}^{s} \ar[r]^{ES _{\U}} \ar[d] & \mc{M}^{\dg,w}_{\U,\Cp}(-1) \ar[d] \\
\mbf{V}_{\kappa,\Cp}^{s} \ar[r]^{ES _{\kappa}} & \mc{M}^{\dg,w}_{\kappa,\Cp}(-1).
}
\]

\bigskip

\subsection{Factorization for weights $k\geq 2$}\label{sec: factor}
To gain some control of the overconvergent Eichler-Shimura map we will prove that it factors through the $p$-adic Eichler-Shimura map defined by Faltings (\cite{fal}) for integral weights $k \in \mb{Z} _{\geq 2}$. The key step is to prove this factorization on the level of sheaves, which is the goal of this section. In our setup, this factorization turns out to be rather transparent. 

\medskip

Let $V\in \X_{\proet}$ be qcqs and put $V _{\infty} = V \times _{\X} \Xinf$. Let $k\geq 2$. By Lemma \ref{lem-explicitD} we know that 
$$\oo\mbf{D}^{s}_{k}(V) = (\mbf{D}_{k}^{s} \ctens_{\Qp} \wh{\oo}_{\X}(V_{\infty}))^{K_{0}(p)}.$$
Let $\mc{G}^{\univ}$ be the universal $p$-divisible group over $\X$. We define $\mc{T}$ to be the relative Tate module of $\mc{G}^{\univ}$ viewed as a sheaf on $\X_{\proet}$. Let $\omega$ be as before. We then define variants of those sheaves on $\X_{\proet}$ by
$$\wh{\mc{T}} = \mc{T} \otimes _{\wh{\mb{Z}}_p} \wh{\oo} _{\X};$$
$$\wh{\omega} = \lambda^{\ast}\omega \otimes _{\oo _{\X}} \wh{\oo} _{\X}.$$
Let us write $\wh{\mc{V}}_{k} = \Sym ^{k-2} (\mc{T}) \otimes _{\wh{\mb{Z}}_p} \wh{\oo}_{\X} =  \Sym^{k} (\wh{\mc{T}})$ and identify $\ms{L}_{k}^{\circ}$ with $\Sym ^{k-2} (\Zp ^2)$ via the isomorphism sending $X^{i}$ to $e_{1}^{i}e_{2}^{k-2-i}$ ($0\leq i \leq k-2$). This identifies the $M_{2}(\Zp)$-action on $\ms{L}_{k}^{\circ}$ with the standard left action on $\Sym^{k-2}(\Zp^{2})$.
\begin{lemm}\label{lem-explicitV}
For any $V\in \X_{\proet}$ qcqs, we have $\wh{\mc{V}}_k (V) \cong (\ms{L}_k \otimes _{\mb{Q}_p} \wh{\oo}_{\X}(V_{\infty}))^{K_{0}(p)}$.
\end{lemm}
\begin{proof}
Follows from Lemma \ref{handy} upon noting that $\mc{T}(V_{\infty}) \cong \ms{L}_{1}^{\circ}$ functorially in $V$ and $\Delta_{0}(p)$-equivariantly via the universal trivialization over $\Xinf$.
\end{proof}
The above lemma and Lemma \ref{lem-explicitD} gives us a map 
$$\mc{O}\mbf{D}_k ^s(V) \ra \wh{\mc{V}}_k(V)$$
induced by the integration map
$$i_k : \mbf{D} _k ^s \ra \ms{L}_k $$
defined in Definition \ref{def-int}.

\medskip

Let $\mf{s} \in H^0(\Xinfw, \omega )$ be the non-vanishing section defined in \S \ref{sec-period}. Recall that
$$\wh{\omega} ^{\dg} _{k, w}(V) = \{ f \in \wh{\oo}_{\Xw}(V_{\infty}) | \gamma^{\ast} f = (b\mf{z}+d)^{-k}f\}$$
for $V\in \mc{X}_{w,\proet}$ qcqs.
\begin{lemm}\label{eta-isom}
We have
$$\wh{\omega} ^{\dg} _{k, w} \cong (\wh{\omega}^{\otimes k-2})|_{ \mc{X}_{w}}$$ 
via the map
$$f \mapsto f \cdot \mf s ^{\otimes k-2}$$
\end{lemm}
\begin{proof}
For $V$ qcqs the map induces an isomorphism between $\wh{\omega} ^{\dg} _{k, w}(V)$ and the set $\{ \eta \in \wh{\omega}^{\otimes k-2}(V_{\infty}) \mid \gamma^{\ast}\eta = \eta \}$ using the transformation rule for $\mf{s}$. By an argument similar to that in the (first half of the) proof of Lemma \ref{lemm: inv2} the latter is functorially isomorphic to $\wh{\omega}^{\otimes k-2}(V)$.
\end{proof}

Recall the linearized Hodge-Tate map
$$\wh{\mc{T}} \ra \wh{\omega}$$ 
as a map of sheaves of $\X_{\proet}$. Taking $(k-2)$-th symmetric powers we get a map
$$\wh{\mc{V}}_k \ra  \wh{\omega} ^{\otimes k-2}.$$
Restricting to $\Xw$ and using Lemma \ref{eta-isom} we get a map  
$$v_k: \wh{\mc{V}}_k \ra  (\wh{\omega} ^{\otimes k-2})|_{\Xinfw} \cong \wh{\omega} ^{\dg} _{k,w}$$
which we want to describe explicitly:

\begin{lemm}\label{exp-v}
Let $k\geq 2$ and let $V_{\infty} \in \X_{w,\proet}/\Xinf$ be qcqs. Define a map
$$ \ms{L}_k \otimes _{\mb{Q}_p} \wh{\oo}_{\Xw}(V_{\infty}) \ra \wh{\oo}_{\Xw}(V_{\infty}) $$
by $X^i \mapsto \mf{z}^i$ for $0\leq i \leq k-2$. This is an $\ohat_{\X}$-linear and $\Delta_{0}(p)$-equivariant morphism of sheaves on $\X_{w,\proet}/\Xinf$. Pushing forward to $\X_{w,\proet}$ and taking $K_{0}(p)$-invariants we get a map $\wh{\V}_{k} \ra \wh{\omega}_{k,w}^{\dg}$. Evaluating this on $V\in \X_{w,\proet}$ qcqs and putting $V_{\infty}=V \times_{\X}\Xinfw$, we obtain a map
$$ (\ms{L}_k \otimes _{\mb{Q}_p} \wh{\oo}_{\Xw}(V_{\infty}))^{K_{0}(p)} \ra \wh{\omega} ^{\dg} _{k,w}(V).$$
which is equal to $v_{k}$ (using  Lemma \ref{lem-explicitV}). 
\end{lemm}
\begin{proof}
We describe the map $\wh{\mc{V}}_k \ra  \wh{\omega} ^{\otimes k-2}$ and then use the isomorphism of the previous lemma. Over $V_{\infty}$, the sheaf $\wh{\omega} ^{\otimes k-2}$ is trivialized by $\mf{s}^{\otimes k-2}$, so $v_{k}$ is morphism
$$ \ms{L}_k \otimes _{\mb{Q}_p} \wh{\oo}_{\Xw}(V_{\infty}) \ra \wh{\oo}_{\Xw}(V_{\infty})\mf{s}^{\otimes k-2}.$$
From the definitions we see that it sends $X^{i}$ to $\HT(\alpha(e_{1}))^{i}\HT(\alpha(e_{2}))^{k-2-i}$ where $\alpha$ denotes the trivialization of $\mc{T}$ over $V_{\infty}$ and $\HT$ denotes the Hodge-Tate map. By definition, $\mf{s}=\HT(\alpha(e_{2}))$ and over $\Xw$ we have the relation $\HT(\alpha(e_{1}))=\mf{z}\HT(\alpha(e_{2}))$. Thus $X^{i}$ is sent to $\mf{z}^{i}\mf{s}^{\otimes k-2}$. Finally observe that the isomorphism of Lemma \ref{eta-isom} sends $\mf{s}^{\otimes k-2}$ to $1$.
\end{proof}

We can now prove the factorization theorem at the level of sheaves for our overconvergent Eichler-Shimura map.

\begin{prop}\label{factor}
For $k\geq 2$ we have a commutative diagram:
\begin{equation*}
\renewcommand{\labelstyle}{\textstyle}
\xymatrix@R=3pc@C=9pc{
\mc{O}\mbf{D}^{s}_{k} \ar[r]^-{\mu \mapsto \mu((1+\mf{z}x)^{k-2})} \ar[d]_-{\mu \mapsto \mu((1+Xx)^{k-2})} & \wh{\omega} ^{\dagger} _{k, w} \\
\wh{\mc{V}}_k \ar[ur]_-{X ^i \mapsto \mf{z} ^i}
}
\end{equation*}

\end{prop}
\begin{proof}
It follows from our prior discussion and Lemma \ref{exp-v}.
\end{proof}

\begin{rema}\label{factorglobal}
This diagram, which is a diagram of sheaves over $\Xw$, may be viewed as the restriction of the following diagram of sheaves over $\X$:
\begin{equation*}
\renewcommand{\labelstyle}{\textstyle}
\xymatrix@R=3pc@C=9pc{
\mc{O}\mbf{D}^{s}_{k} \ar[r]^-{\mu \mapsto \mu((1+\mf{z}x)^{k-2})\mf{s}^{\otimes k-2}} \ar[d]_-{\mu \mapsto \mu((1+Xx)^{k-2})} & \wh{\omega}^{\otimes k-2} \\
\wh{\mc{V}}_k \ar[ur]_-{X ^i \mapsto \mf{z}^i\mf{s}^{\otimes k-2}}
}
\end{equation*}
using the isomorphism of Lemma \ref{eta-isom}. Note that although $\mf{z}$ and $\mf{s}$ do not extend to $\Xinf$, the products $\mf{z}^{i}\mf{s}^{\otimes k-2}$ do extend to $\Xinf$ for all $0\leq i \leq k-2$, so the formulas in the diagram make sense. This will cause our overconvergent Eichler-Shimura maps to have very large kernels at classical weights $k\geq 2$ (see \S \ref{sec: Faltings}).  
\end{rema}

\section{The overconvergent Eichler-Shimura map over the eigencurve}\label{sec:5}

In this section we prove our main results concerning the overconvergent Eichler-Shimura map. These results are analogous to the main results of \cite{oes}. However, the payoff for working with arbitrary small open weights is that we may glue our overconvergent Eichler-Shimura maps for different small open weights into a morphism of coherent sheaves over the eigencurve. This allows us to work over rather arbitrary regions of weight space and with general slope cutoffs, as opposed to the rather special open discs inside the analytic part of weight space used in \cite{oes}.

\subsection{Sheaves on the eigencurve}\label{sec:sheaves}

Let $\U$ be a small weight and let $s \geq 1+s_{\U}$. Recall we have defined the modules
$$\mbf{V}^{s} _{\U} = H^{1} _{\proet}(\X _{\Cp}, \mbf{D} _{\U}^s);$$ 
$$\mbf{V}^{s} _{\U, \Cp} \cong H^{1} _{\proet}(\X_ {\Cp}, \oo \mbf{D} _{\U}^s);$$ 
$$\mc{M}^{\dg,w} _{\U} = H^{0}(\Xw, \omega^{\dg} _{\U,w} \otimes _{\oo _{\Xw}} \Omega^{1} _{\Xw});$$  
$$\mc{M}^{\dg,w} _{\U, \Cp} = H^{0} (\X_{w, \Cp}, \omega ^{\dg} _{\U,w, \Cp} \otimes _{\oo_{\X_{w, \Cp}}} \Omega ^1 _{\X_{w,\Cp}}).$$
In this subsection we spread the finite-slope pieces of these modules into sheaves on the whole eigencurve $\mc{C}=\mc{C}_{K^p}$, and we glue our Eichler-Shimura maps into a morphism of sheaves.

\medskip

Recall the rigid analytic weight space $\mc{W}=\Spf(\mathbb{Z}_p[[\mathbb{Z}_p^{\times}]])^{\mathrm{rig}}$, with its universal character $\chi_{\mc{W}}:\mathbb{Z}_p^{\times} \to \mc{O}(\mc{W})^{\times}$. If $\mc{U}$ is an open weight, we write $\mc{O}^{bd}(\mc{U}^{\rig})$ for the ring of bounded functions on $\U^{\rig}$. This is a Banach $\Qp$-algebra equal to $A_{\U}[\frac{1}{p}]$ as a dense subring, and with a natural map $\mc{O}(\mc{W})\to \mc{O}(\mc{U}^{\rig})$. Recall that we have defined the Fredholm determinant
$$F_{\mc{U}} (T) = \det (1- \wt{U}T | C_{\bullet}(K,\mbf{A}_{\U}^{s}))$$
in \S \ref{slopedecomp}. It is an element of $\oo(\U^{\rig})^{bd}\{\{T\}\}$, the ring of entire power series with coefficients in $\oo^{bd}(\U^{rig})\{\{T\}\}$, and we showed that it glues to a Fredholm series $F(T)\in \oo(\mc{W})\{\{T\}\}$ (see e.g. \cite[Definition 4.1.1]{han1} for the definition of a Fredholm series). Now let $\V=(S_{\V},\chi_{\V})$ be an open \emph{affinoid} weight. Arguing as in \cite[\S 5.2]{pil} we see that $\mc{M}_{\V}^{\dg,w}$ satisfies Buzzard's property (Pr) over $S_{\V}$ (i.e. it is a direct summand of an orthonormalizable $S_{\V}$-module, see the paragraph before Lemma 2.11 of \cite{buz2}). Since $U_{p}$ is compact we get a Fredholm determinant
$$G_{\mc{V}} (T) = \det (1- U_p T | \mc{M}_{\mc{V}}^{\dg,w})$$
in $\mc{O}(\mc{V})\{\{ T \}\}$ and these glue to a Fredholm series $G(T)\in \oo(\mc{W})\{\{T\}\}$. Let us set $H(T)=F(T)G(T)$, this is also a Fredholm series. Recall (\cite[\S 4.1]{han1}) that a \emph{Fredholm hypersurface} is a closed subvariety of $\mc{W}\times \mb{A}^{1}$ cut out by a Fredholm series. Given a Fredholm series $f$ we write $\ms{Z}(f)$ for the corresponding Fredholm hypersurface. Recall that, if $h\in \mb{Q}_{\geq 0}$ and $W\sub \mc{W}$ is open, then $f$ has a slope $\leq h$-decomposition in $\oo^{bd}(W)\{\{T\}\}$ if and only if the natural map $\ms{Z}_{W, h}(f):=\ms{Z}(f)\cap (W\times \mb{B}[0,h]) \ra \mb{B}[0,h]$ is finite flat, where $\mb{B}[0,h]\sub \mb{A}^{1}$ is the closed disc around $0$ of radius $p^{h}$. We say that $(W,h)$ is \emph{slope-adapted} for $f$ if these equivalent conditions hold. We have the following key result:
\begin{lemm}
There exists a collection of pairs $(\V_{i},h_{i})$, with $\mc{V}\sub \mc{W}$ open and affinoid and $h_{i}\in \mb{Q}_{\geq 0}$, such that $(\V_{i},h_{i})$ is slope-adapted for $H$ for all $i$ and the $(\ms{Z}_{\V_{i},h_{i}}(H))_{i}$ form an open (admissible) cover of $\ms{Z}(H)$. Moreover, for each $i$ we may find a small open weight $\U_{i}$ such that $\V_{i}\sub \U^{\rig}_{i}$ and $(\U^{\rig}_{i},h_{i})$ is slope-adapted for $H$.
\end{lemm}
\begin{proof}
See \cite[Lemmas 2.3.1-2.3.4]{han} (the first part is a theorem of Coleman-Mazur and Buzzard and is the key step in the "eigenvariety machine", cf. \cite[Theorem 4.6]{buz2}).
\end{proof}

Fix collections $\V_{i}$, $\U_{i}$ and $h_{i}$ satisfying the conclusions of Lemma. For now, let $h\in \mb{Q}_{\geq 0}$ and let $\U$ be any small open weight such that $(\U^{rig},h)$ is slope-adapted for $H$, and let $\V \sub \U^{\rig}$ be an open affinoid weight. Recall our overconvergent Eichler-Shimura map 
$$ ES_{\U}\, :\, \mbf{V}_{\U,\Cp}^{s} \ra \mc{M}_{\U,\Cp}^{\dg,w}(-1). $$
Since $(\U^{\rig},h)$ is slope-adapted for $H$ (and hence for $F$ and $G$) and $ES_{\U}$ is Hecke-equivariant we may take slope $\leq h$-parts on both sides to obtain an $R_{\U}[\frac{1}{p}]$-linear map
$$ ES_{\U,\leq h}\,:\, \mbf{V}_{\U,\Cp,\leq h}^{s} \ra \mc{M}_{\U,\Cp,\leq h}^{\dg,w}(-1)$$
of finite projective $R_{\U}[\frac{1}{p}]$-modules. We may compose with this map the natural map $\mc{M}_{\U, \Cp, \leq h}^{\dg,w}(-1) \ra \mc{M}_{\V,\Cp, \leq h}^{\dg,w}(-1)$ and tensor the source with $S_{\V}$ to obtain an $S_{\V}$-linear map
$$ ES_{\V,\leq h}\, :\, \mbf{V}_{\V,\Cp,\leq h}^{s}\cong \mbf{V}_{\U,\Cp,\leq h}^{s}\otimes_{R_{\U}[\frac{1}{p}]}S_{\V} \ra \mc{M}_{\mc{V},\Cp, \leq h}^{\dg,w}(-1)$$
where the first isomorphism is that of Proposition \ref{slopeV}.

\medskip

Let us now recall the eigencurves for $B^{\times}$ of tame level $K^{p}$ constructed out of overconvergent modular forms resp. overconvergent modular symbols, using the perspective of \cite[\S 4]{han1}. Let $\Sigma_{0}$ denote the finite set of primes $\ell$ for which $K_{\ell}\not \cong \gl(\mb{Z}_{\ell})$ and let $\Sigma = \Sigma_{0}\cup \{p\}$. We let $\mb{T}$ denote the abstract Hecke algebra generated by commuting formal variables $T_{\ell}$, $S_{\ell}$ for $\ell \notin \Sigma$ and $U_{p}$. Then $\mb{T}$ acts on all spaces of modular forms and modular symbols used in this article. We may construct two eigenvariety data (\cite[Definition 4.2.1]{han1}) $\mf{D}_{\mc{M}^{\dg}}=(\mc{W}, \ms{Z}(H), \mc{M}^{\dg}, \mb{T}, \psi_{\mc{M}^{\dg}})$ and $\mf{D}_{\mbf{V}}=(\mc{W}, \ms{Z}(H), \mbf{V}, \mb{T}, \psi_{\mbf{V}})$. Here $\mc{M}^{\dg}$ and $\mbf{V}$ are the coherent sheaves on $\ms{Z}(H)$ obtained by gluing the $\varinjlim_{w} \mc{M} ^{\dg,w} _{\mc{V}_{i},\leq h_{i}}$ resp. the $\varprojlim_{s}\mbf{V}_{\mc{V}_{i},\leq h_{i}}^{s}$ over the $\ms{Z}_{\mc{V}_{i},h_{i}}(H)$ (see \cite[Proposition 4.3.1]{han1} and the discussion following it for the construction of $\mbf{V}$; the construction of $\mc{M}^{\dg}$ is similar but easier). Given these, \cite[Theorem 4.2.2]{han1} allows us to construct the eigenvarieties $\mc{C}_{\mc{M}^{\dg}}$ and $\mc{C}_{\mbf{V}}$ together with coherent sheaves $\mc{M}^{\dg}$ on $\mc{C}_{\mc{M}^{\dg}}$ resp. $\mbf{V}$ on $\mc{C}_{\mbf{V}}$ (this is a mild abuse of notation; we will have no further reason to consider the sheaves with the same name over $\ms{Z}(H)$). Since the following proposition is well known we only give a brief sketch of the proof:

\begin{prop}
There is a canonical isomorphism $\mc{C}_{\mc{M}^{\dg}}\cong \mc{C}_{\mbf{V}}$ compatible with the projection maps down to $\mc{W}$ and the actions of $\mb{T}$.
\end{prop}

\begin{proof}
This follows by applying \cite[Theorem 5.1.2]{han1} twice (once in each direction), upon noting that both eigencurves are reduced and equi-dimensional (of dimension $1$; they are "unmixed" in the terminology of \cite{han1}). For these applications one should modify the eigenvariety data, replacing $\ms{Z}(H)$ in $\mf{D}_{\mc{M}^{\dg}}$ with the support of $\mc{M}^{\dg}$ and similarly for $\mf{D}_{\mbf{V}}$. The relevant very Zariski dense subsets are then constructed using the control/classicality theorems of Stevens and Coleman together with the usual Eichler-Shimura isomorphism (see \cite[Theorem 3.2.5]{han1} for a general version of the control theorem for overconvergent modular symbols, and see e.g. \cite[Theorem 4.16]{joh} for a proof of Coleman's control theorem in the context of the compact Shimura curves used here).
\end{proof}

In light of this we will identify the two eigencurves and simply denote it by $\mc{C}$. It carries two coherent sheaves $\mc{M}^{\dg}$ and $\mbf{V}$ which are determined by canonical isomorphisms
$$ \mc{M}^{\dg}(\mc{C}_{\V_{i},h_{i}}) \cong \varinjlim_{w} \mc{M}_{\V_{i}, \leq h_{i}}^{\dg,w} \cong \mc{M}_{\V_{i}, \leq h_{i}}^{\dg,w}; $$
$$ \mbf{V}(\mc{C}_{\V_{i},h_{i}}) \cong \varprojlim_{s}\mbf{V}_{\V_{i}, \leq h_{i}}^{s} \cong \mbf{V}_{\V_{i}, \leq h_{i}}^{s} $$
for all $i$ and sufficiently large $w$ resp. $s$, where $\mc{C}_{\V_{i},h_{i}}$ is the preimage of $\ms{Z}_{\V_{i},h_{i}}(H)$ in $\mc{C}$. We form the sheaves $\mc{M}^{\dg}_{\Cp} := \mc{M}^{\dg} \ctens_{\Qp} \Cp $ and $\mbf{V}_{\Cp}:= \mbf{V}\ctens_{\Qp} \Cp$. They are determined by canonical isomorphisms
$$ \mc{M}^{\dg}_{\Cp}(\mc{C}_{\V_{i},h_{i}}) \cong \varinjlim_{w} \mc{M}_{\V_{i},\Cp, \leq h_{i}}^{\dg,w} \cong \mc{M}_{\V_{i},\Cp, \leq h_{i}}^{\dg,w}; $$
$$ \mbf{V}_{\Cp}(\mc{C}_{\V_{i},h_{i}}) \cong \varprojlim_{s}\mbf{V}_{\V_{i}, \Cp, \leq h_{i}}^{s} \cong \mbf{V}_{\V_{i}, \Cp, \leq h_{i}}^{s} $$
for any sufficiently large $w$ resp. $s$. We may also naturally view these sheaves as coherent sheaves on the base change $\mc{C}_{\Cp}$ of $\mc{C}$ from $\Qp$ to $\Cp$. They are then determined by the obvious modifications of the above canonical isomorphisms. 

\medskip

With these preparations we may now state the main theorem of this section which glues the maps $ES_{\V_{i},\leq h_{i}}$:
\begin{theo}
There exists a canonical Hecke and Galois-equivariant morphism
$$\mc{ES}: \mbf{V} _{\Cp} \ra \mc{M}^{\dg} _{\Cp}(-1)$$
of coherent sheaves of $\oo _{\mc{C}} \ctens_{\Qp} \Cp$-modules on $\mc{C}$ (or coherent sheaves of $\oo_{\mc{C}_{\Cp}}$-modules on $\mc{C}_{\Cp}$), which glues $(ES _{\V_i, \leq h_i})_{i}$.
\end{theo}

\begin{proof}
The $ES_{\V_{i},\leq h_{i}}$ induce maps over the cover $(\mc{C}_{\V_{i},h_{i}})_{i}$ using the canonical isomorphisms. Checking that they glue is tedious but straightforward, using the functoriality of the maps $ES_{\U}$ in the small open weight $\U$ as well as the naturality of the construction of $ES_{\V,\leq h}$ for $\V \sub \U^{rig}$, $\V$ open affinoid and $\U$ small open such that $(\U^{rig},h)$ is slope-adapted for $H$. We leave the remaining details to the reader.
\end{proof}

The morphism $\mc{ES}$ is our main object of study in this section. Our main technique to study it is via specialization to non-critical classical points of $\mc{C}$ (recall that a classical point $x$ is non-critical if $h<k-1$, where $h$ is the slope and $k$ is the weight of $x$). Let $\kappa \in \mc{W}(L)$ with $L/\Qp$ finite and let $(\V,h)$ be slope-adapted with $\V$ connected and affinoid and $\kappa \in \V(L)$. By \cite[Theorem 3.3.1]{han1} there is a second quadrant spectral sequence (with $s$ big enough)
$$ E_{2}^{pq}={\rm Tor}_{-p}^{\oo_{\mc{W}}(\V)}(H^{q}(K,\mbf{D}_{\V}^{s})_{\leq h}, L) \implies H^{p+q}(K, \mbf{D}_{\kappa}^{s})_{\leq h} $$
which degenerates at the $E_{2}$-page since $\oo_{\mc{W}}(\V)$ is a Dedekind domain. This gives us a short exact sequence 
$$ 0 \ra \mbf{V}(\mc{C}_{\V,h})\otimes_{\oo_{\mc{W}}(\V)}L \ra H^{1}(K, \mbf{D}_{\kappa}^{s})_{\leq h} \ra {\rm Tor}_{1}^{\oo_{\mc{W}}(\V)}(H^{2}(K,\mbf{D}_{\V}^{s})_{\leq h},L) \ra 0. $$
Now assume that $x\in \mc{C}(L^{\prime})$ is a classical non-critical point of weight $k=\kappa$ (with $L^{\prime}/L$ a finite field extension). After localizing the above short exact sequence at $x$ it becomes an isomorphism 
$$ (\mbf{V}(\mc{C}_{\V,h})\otimes_{\oo_{\mc{W}}(\V)}L)_{x} \cong (H^{1}(K, \mbf{D}_{k}^{s})_{\leq h})_{x}. $$
This follows from the fact that $x$ does not occur in $H^{2}(K,\mbf{D}_{\V}^{s})_{\leq h}$, which by \cite[Proposition 4.5.2]{han1} follows from the fact that it does not occur in $H^{2}(K,\mbf{D}_{k}^{s})_{\leq h}$, which in turn follows from the control theorem. Using the control theorem again, the right hand side of the isomorphism is canonically isomorphic to $H^{1}_{\et}(\X_{\Cp}, \V_{k})_{x}$. From this we deduce that the fiber of $\mbf{V}$ at $x$ is the largest semisimple quotient of the generalized eigenspace in $H_{\et}^{1}(\X_{\Cp}, \V_{k})$ associated with $x$.

\medskip

By a similar but simpler analysis (using Lemma \ref{lemm:fibre}) one sees that the fibre of $\mc{M}^{\dg}$ at an arbitrary point $x\in \mc{C}(\overline{\mb{Q}}_{p})$ is equal to the largest semisimple quotient of the generalized eigenspace in $H^{0}(\Xw,\omega^{\dg}_{\kappa,w}\otimes_{\oo_{\Xw}}\Omega_{\Xw}^{1})$ associated with $x$ where $w$ is sufficiently large. One may of course apply the control theorem if $x$ is non-critical and classical to gain a further refinement. Similarly, the analogous statements apply to the same objects based changed from $\Qp$ to $\Cp$. 

\medskip

We will need a few extras fact about $\mc{C}$ before we proceed to analyze $\mc{ES}$.

\begin{lemm}
Let $\kappa \in \mc{W}(\overline{\mb{Q}}_{p})$ be a weight and let $s\geq s_{\kappa}$. Then $H^{0}(K,\mbf{D}_{\kappa}^{s})=0$.
\end{lemm}

\begin{proof}
Write $G^1$ for the closed subgroup of $G$ of elements of reduced norm one. Under our fixed isomorphism $G(\Zp) \cong \gl(\Zp)$ we have $G^{1}(\Zp)\cong \SL_{2}(\Zp)$. To prove the lemma, it suffices to show that $H^0(\Gamma, \mbf{D}_{\kappa}^s)=0$ for any congruence subgroup $\Gamma \subset G^1(\mb{Z})$ contained in $K_{0}(p)$. By the $p$-adic continuity of the $\Gamma$-action on $\mbf{D}_{\kappa}^s$, we have 
$$ H^0(\Gamma, \mbf{D}_{\kappa}^s)=H^0(\Gamma_p, \mbf{D}_{\kappa}^s), $$ 
where $\Gamma_p$ denotes the $p$-adic closure of $\Gamma$ in $\gl(\Zp)$. By \cite[Lemma 2.7]{rap} and the Zariski density of $\Gamma$ in $G^1$ (see e.g. \cite[Theorem 4.10]{pr}), $\Gamma_p$ contains an open subgroup of $\SL_2(\Zp)$, and in particular contains a nontrivial element $u=\left( \begin{matrix} 1 & a \\ 0 & 1 \end{matrix} \right)$ for some $a \in \Zp, a\neq 0$. Note that $u$ acts on $\mu \in \mbf{D}_{\kappa}^s$ by $(u \cdot \mu)(f(x))= \mu(f(x+a))$; in particular this action does not depend on $\kappa$. Since any element of $H^0(\Gamma_p, \mbf{D}_{\kappa}^s) \sub \mbf{D}_{\kappa}^s = \mbf{D}^s \otimes_{\mb{Q}_p} L$ ($L$ the residue field of $\kappa$) is fixed by $u$, it now suffices to show that no nonzero element of $\mbf{D}^s$ is fixed by $u$. To see this, note that the Amice transform \[ \mu \mapsto A_{\mu}(T)=\int (1+T)^x \mu(x) \in \Qp[[T]] \] defines a $\Qp$-linear injection of $\mbf{D}^s$ (regarded as a ring under convolution) into a subring of $\Qp[[T]]$ (cf. \cite[\S 1.8]{col} for more detailed statements).  An easy calculation shows that $A_{u \cdot \mu}(T) = (1+T)^a A_{\mu}(T)$, so if $u \cdot \mu = \mu$ then $\left( (1+T)^a -1\right)\cdot A_\mu(T) =0$, and since $\Qp[[T]]$ is a domain this implies $A_\mu(T)=0$ and then $\mu=0$ as desired.  
\end{proof}

\begin{lemm}\label{lemm:torsionfree}
The sheaves $\mc{M}^{\dg}$ and $\mbf{V}$ are torsion free.
\end{lemm}

\begin{proof}
We start with $\mc{M}^{\dg}$. This is a local statement so we may work over some $\mc{C}_{\V,h}$ coming from a slope-adapted $(\V,h)$ with $\V$ affinoid. Then $\oo_{\mc{C}}(\mc{C}_{\V,h})$ is finite over $\oo_{\mc{W}}(\V)$, so it suffices to show torsion freeness over $\oo_{\mc{W}}(\V)$, which is clear by definition.

\medskip

\noindent We now prove torsion-freeness for $\mbf{V}$. Working locally as above it is enough to show that $H^{1}(K,\mbf{D}_{\V}^{s})_{\leq h}$ is a torsion-free $\oo_{\mc{W}}(\V)$-module, where $\V \sub \mc{W}$ is affinoid and $(\V,h)$ is slope-adapted. Let $\kappa \in \V(\overline{\mb{Q}}_{p})$ be a weight, cutting out a maximal ideal $\mf{m}_{\kappa}$ and let $L:=\oo_{\mc{W}}(\V)/\mf{m}_{\kappa}$. It is enough to show to that the $\mf{m}_{\kappa}$-torsion vanishes for all $\kappa$. This torsion is equal to ${\rm Tor}_{1}^{\oo_{\mc{W}}(\V)}(H^{1}(K,\mbf{D}_{\V}^{s})_{\leq h},L)$. Using the Tor-spectral sequence (\cite[Theorem 3.3.1]{han1}) one sees that this ${\rm Tor}$-group is a subquotient of $H^{0}(K,\mbf{D}_{\kappa}^{s})_{\leq h}$, which vanishes by the previous Lemma.
\end{proof}

Next, assume that our tame level is of the form $K_1(N)$ for some $N$ with $(N,dp)=1$. Let $\mc{C}^{\mathrm{nc}}\subset \mc{C}$ denote the set of non-critical classical crystalline points for which the roots of the $p$th Hecke polynomial are distinct.\footnote{This last condition is conjecturally automatic.} For each positive divisor $M|N$, let $\mc{C}^{\mathrm{nc}}_{M-\mathrm{new}}$ denote the set of points $x$ in $\mc{C}^{\mathrm{nc}}$ for which the tame Artin conductor of the Galois representation $\rho_x$ associated to $x$ is exactly $Md$. Note that the tame Artin conductor is sometimes just called the tame conductor; see the discussion in \cite{bel} before Lemma IV.4.3. Let $\mc{C}_{M-\mathrm{new}}$ denote the Zariski-closure of $\mc{C}^{\mathrm{nc}}_{M-\mathrm{new}}$ in $\mc{C}$. The following lemma was stated without proof for modular curves in \cite{han}.

\begin{lemm}Assume that the tame level is of the form $K_{1}(N)$. For any $M|N$, $\mc{C}_{M-\mathrm{new}}$ is a union of irreducible components of $\mc{C}$, and $\mc{C} = \cup_{M|N} \mc{C}_{M-\mathrm{new}}$.  
\end{lemm}

\begin{proof}Adapting the proof of \cite[Lemma IV.4.7]{bel}, one shows that $\mc{C}^{\mathrm{nc}}_{M-\mathrm{new}}$ is an \emph{accumulation subset} of $\mc{C}$: each point $x \in \mc{C}^{\mathrm{nc}}_{M-\mathrm{new}}$ has a neighborhood basis of affinoids $U\subset \mc{C}$ for which $U \cap \mc{C}^{\mathrm{nc}}_{M-\mathrm{new}}$ is Zariski-dense in $U$.  This property implies that each irreducible component of the Zariski-closure $\mc{C}_{M-\mathrm{new}}$ has positive dimension. Since $\mc{C}$ is equidimensional of dimension one, we deduce from \cite[Corollary 2.2.7]{con}  that $\mc{C}_{M-\mathrm{new}}$ is a union of irreducible components of $\mc{C}$. Since $\mc{C}^{\mathrm{nc}}=\cup_{M \mid N} \mc{C}^{\mathrm{nc}}_{M-\mathrm{new}}$ is a Zariski-dense accumulation subset of $\mc{C}$, the remainder of the lemma is clear.
\end{proof}

\subsection{Faltings's Eichler-Shimura map}\label{sec: Faltings} In this section we adapt Faltings's construction (\cite{fal}) of the $p$-adic Eichler-Shimura morphism to our setting, and use it to give a precise description of $ES_{k}$ at weights $k\geq 2$. Fix $k \in \mb{Z}_{\geq 2}$. Recall the morphism 
$$ \wh{\V}_{k} \ra \wh{\omega}^{\otimes k-2} $$
defined in \S \ref{sec: factor}. It induces a map
$$H^{1}_{\proet}(\X_{\Cp},\wh{\V}_{k}) \ra H^{1}_{\proet}(\X_{\Cp}, \wh{\omega}^{\otimes k-2}).$$ 
Note that we have an exact sequence
\begin{equation}\label{eq:omega} 
0 \ra H^{1}_{\et}(\X_{\Cp}, \omega^{\otimes k-2}) \ra H^{1}_{\proet}(\X_{\Cp}, \wh{\omega}^{\otimes k-2}) \ra H^{0}_{\et}(\X_{\Cp}, \omega^{\otimes k-2}\otimes \Omega_{\X_{\Cp}}^{1})(-1) \ra 0
\end{equation}
coming from the spectral sequence $H^i_{\et}(\X_{\Cp},R^j\nu_{\ast}\wh{\omega} ^{\otimes k-2}) \Rightarrow H^{i+j}_{\proet}(\X_{\Cp},\wh{\omega} ^{\otimes k-2})$, which degenerates at the $E_{2}$-page by an argument similar to (but simpler than) the proof of Proposition \ref{prop:omega}.

\begin{prop}\label{Faltings ES}
The composite map 
$$H^{1}_{\proet}(\X_{\Cp},\wh{\V}_{k}) \ra H^{1}_{\proet}(\X_{\Cp}, \wh{\omega}^{\otimes k-2}) \ra H^{0}_{\et}(\X_{\Cp}, \omega^{\otimes k-2}\otimes \Omega_{\X_{\Cp}}^{1})(-1)$$
is surjective, and the kernel is isomorphic to $H^{0}_{\et}(\X_{\Cp}, \omega^{\otimes k-2}\otimes \Omega_{\X_{\Cp}}^{1})(k)$ as a Hecke- and Galois module.
\end{prop}

\begin{proof}
Recall the Hodge-Tate sequence 
$$ 0 \ra \wh{\omega}^{\otimes -1}(1) \ra \wh{\mc{T}} \ra \wh{\omega} \ra 0 $$
for $\mc{G}^{\univ}$. Since the morphism $\wh{\mc{V}}_{k} \ra \wh{\omega}^{\otimes k-2}$ comes from the map $\wh{\mc{T}} \ra \wh{\omega}$ by taking $(k-2)$-th symmetric powers, we see that its kernel ${\rm Ker}_{k}$ sits in a short exact sequence
$$  0 \ra \wh{\omega}^{\otimes 2-k}(k-2) \ra {\rm Ker}_{k} \ra Q \ra 0 $$
where $Q$ is simply defined to be the quotient. We get exact sequences
$$ H^{1}_{\proet}(\X_{\Cp}, {\rm Ker}_{k}) \ra H^{1}_{\proet}(\X_{\Cp},\wh{\V}_{k}) \ra H^{1}_{\proet}(\X_{\Cp}, \wh{\omega}^{\otimes k-2}); $$
$$ H^{1}_{\proet}(\X_{\Cp}, \wh{\omega}^{\otimes 2-k}(k)) \ra H^{1}_{\proet}(\X_{\Cp},{\rm Ker}_{k}) \ra H^{1}_{\proet}(\X_{\Cp}, Q). $$
We have an isomorphism $H^{1}_{\proet}(\X_{\Cp},\wh{\V}_{j})\cong H^{1}_{\et}(\X_{\Cp},\Sym^{j}\mc{T})\otimes_{\Zp}\Cp $ for all $j\geq 0$. From the Hodge-Tate sequence one deduces that $Q$ carries a filtration with non-zero graded pieces $\wh{\omega}^{\otimes i}$ for $i=4-k,6-k,...,k-4$ (up to twists). Using this filtration and the sequences (\ref{eq:omega}) with $k-2$ replaced by $i=4-k,6-k,...,k-4$ together with the usual Eichler-Shimura isomorphism for quaternionic Shimura curves one sees that $H^{1}_{\proet}(\X_{\Cp},\wh{\V}_{k})$ and $H^{1}_{\proet}(\X_{\Cp}, Q)$ do not have any Hecke eigenvalues in common. Thus, looking at the second exact sequence above, we see that any generalized Hecke eigenvector in $H^{1}_{\proet}(\X_{\Cp}, {\rm Ker}_{k})$ that is not killed by the map $H^{1}_{\proet}(\X_{\Cp}, {\rm Ker}_{k}) \ra H^{1}_{\proet}(\X_{\Cp},\wh{\V}_{k})$ must come from $H^{1}_{\proet}(\X_{\Cp}, \wh{\omega}^{\otimes 2-k}(k-2))$.

\medskip

\noindent Hence we have an exact sequence
$$ H^{1}_{\proet}(\X_{\Cp}, \wh{\omega}^{\otimes 2-k}(k-2)) \ra H^{1}_{\proet}(\X_{\Cp},\wh{\V}_{k}) \ra H^{1}_{\proet}(\X_{\Cp}, \wh{\omega}^{\otimes k-2}). $$
By the exact sequence (\ref{eq:omega}) and looking at Hecke eigenvalues again we get an exact sequence
$$ H^{1}_{\proet}(\X_{\Cp}, \wh{\omega}^{\otimes 2-k}(k-2)) \ra H^{1}_{\proet}(\X_{\Cp},\wh{\V}_{k}) \ra H^{0}_{\et}(\X_{\Cp}, \omega^{\otimes k-2}\otimes \Omega_{\X_{\Cp}}^{1})(-1). $$
We now apply a similar argument to $ H^{1}_{\proet}(\X_{\Cp}, \wh{\omega}^{\otimes 2-k}(k-2))$. Replacing $k-2$ with $2-k$ in (\ref{eq:omega}) we have the exact sequence
$$ 0 \ra H^{1}_{\et}(\X_{\Cp}, \omega^{\otimes 2-k}) \ra H^{1}_{\proet}(\X_{\Cp}, \wh{\omega}^{\otimes 2-k}) \ra H^{0}_{\et}(\X_{\Cp}, \omega^{\otimes 2-k}\otimes \Omega_{\X_{\Cp}}^{1})(-1) \ra 0.$$
Arguing with Hecke eigenvalues as above and using Serre duality we get a sequence
\begin{equation}\label{eq:vk}
0 \ra H^{1}_{\et}(\X_{\Cp}, \omega^{\otimes 2-k}(k-2)) \ra H^{1}_{\proet}(\X_{\Cp},\wh{\V}_{k}) \ra H^{0}_{\et}(\X_{\Cp}, \omega^{\otimes k-2}\otimes \Omega_{\X_{\Cp}}^{1})(-1) \ra 0
\end{equation}
which is exact in the middle. By Serre duality 
$$H^{1}_{\et}(\X_{\Cp}, \omega^{\otimes 2-k}(k-2)) \cong H^{0}_{\et}(\X_{\Cp}, \omega^{\otimes k-2}\otimes \Omega_{\X_{\Cp}}^{1})(k-2).$$
Counting dimensions, we see that (\ref{eq:vk}) is a short exact sequence as desired.
\end{proof}

\medskip

We have a diagram
\begin{equation*}
\renewcommand{\labelstyle}{\textstyle}
\xymatrix@R=2pc@C=2pc{ H_{\proet}^{1}(\X_{\Cp},\oo\mbf{D}_{k}^{s}) \ar[r]^{ES _{k}} \ar[d] & H_{\proet}^{1}(\X_{w,\Cp},\wh{\omega}_{k,w}^{\dg}) \ar[r]^{\sim\,\,\,} & H^{0}_{\et}(\X_{w,\Cp},\omega_{k,w}^{\dg}\otimes \Omega_{\X_{w,\Cp}}^{1})(-1) \\
 H^{1}_{\proet}(\X_{\Cp},\wh{\V}_{k}) \ar[r] & H^{1}_{\proet}(\X_{\Cp},\wh{\omega}^{\otimes k-2}) \ar[r] \ar[u] & H_{\et}^{0}(\X_{\Cp},\omega^{\otimes k-2}\otimes \Omega_{\X_{\Cp}}^{1})(-1) \ar[u]
  }
\end{equation*}
which commutes by Remark \ref{factorglobal} and the functoriality of the remaining maps. Note that the right vertical map is injective by analytic continuation. Its image is, by definition, the space of classical modular forms. We then have:

\begin{prop}\label{prop:image}
Let $k\geq 2$. The image of $ES_{k}$ is contained in the space of classical modular forms. Moreover $ES_{k}$ is surjective on slope $\leq h$-parts when $h<k-1$.
\end{prop}

\begin{proof}
The first statement is clear from the diagram. To see the second, note that both the left and right vertical maps in the diagram above are isomorphisms on slope $\leq h$-parts when $h<k-1$ by the control theorems of Stevens and Coleman. To conclude, use that the map $H^{1}(\X_{\Cp},\wh{\V}_{k}) \ra H_{\et}^{0}(\X_{\Cp},\omega^{\otimes k-2}\otimes \Omega_{\X_{\Cp}}^{1})(-1)$ is surjective by Proposition \ref{Faltings ES}.
\end{proof}

Note that the kernel of $ES_{k}$ is big: it is infinite-dimensional with finite codimension.

\medskip

\subsection{Results}

In this section we deduce some properties of our overconvergent Eichler-Shimura map $\mc{ES}$ and the sheaves $\mbf{V}$ and $\mc{M}^{\dg}$. The results are very similar to those of \cite[\S 6]{oes}, but we are able to prove them in a more global form. Let us denote by $\mc{C}^{\sm}$ the smooth locus of $\mc{C}$. It contains the \'etale locus $\mc{C}^{\et}$ of the weight map, and both these loci are Zariski open. Furthermore $\mc{C}^{\et}$ contains the set $\mc{C}^{\nc}$ of non-critical classical crystalline points for which the roots of the $p$-Hecke polynomial are distinct. Let us explicitly record a lemma in rigid geometry that makes arguments involving Zariski density simpler. We will (sometimes implicitly) apply it in the remainder of this section.

\begin{lemm}\label{lemm:qs}
Let $X$ be a smooth rigid space over a non-archimedean field $K$, and assume that $(V_{n})_{n=1}^{\infty}$ is an increasing cover of $X$ by affinoids. Assume that $S\sub X$ is a very Zariski dense set of points. Then we may find an increasing cover $(U_{n})$ of $X$ by affinoids such that $S\cap U_{n}$ is very Zariski dense in $U_{n}$ for all $n$.
\end{lemm}

\begin{proof}
We define $U_{n}$ to be the Zariski closure of $S\cap V_{n}$ in $V_{n}$. By definition $U_{n}$ is a union of components of $V_{n}$, hence affinoid, and it remains to show that $X=\bigcup_{n=1}^{\infty}U_{n}$. For this it suffices to show that, for fixed $n$ and a fixed component $C$ of $V_{n}$, there exists $m\geq n$ such that $C\sub U_{m}$. Fix $n$ and $C$. $C$ is a subset of a (global) component $D$ of $X$, so by Zariski density we pick a point $s\in S\cap D$. If $m\geq n$, let $C_{m}$ denote the component of $V_{m}$ such that $C\sub C_{m}$. Define, for $m \geq n$,
$$ T_{m}=\{ x\in V_{m} \mid \exists k\geq m\, :\, i_{km}(x)\in C_{k} \} $$
where $i_{km}\, :\, V_{m} \ra V_{k}$ is the inclusion map. Then $T_{m}$ is a union of components and we claim that $s\in T_{m}$ for some $m$. To see this, note that $T_{m}=V_{m}\cap T_{m+1}$ for all $m\geq n$, so $T:=\bigcup_{m\geq n} T_{m}$ is an open subset of $X$ with open complement. Hence if $s\notin T$, $s$ could not lie on the same component as $C$. Therefore there is some $m$ for which $s\in T_{m}$, and hence some $k$ for which $s\in C_{k}$, which implies that $C\sub U_{k}$ by the very Zariski density of $S$.
\end{proof}

\begin{rema}\label{rema:qsc}
Note that $\mc{C}^{\sm}$ may be covered by an increasing union of affinoids. To see this note first that this is true for $\mc{C}$ since it is finite over $\mc{W}\times \mb{A}^{1}$, which is quasi-Stein. If we pick an increasing affinoid cover $(V_{n})$ of $\mc{C}$, then $\mc{C}^{\sm}=\bigcup_{n}V_{n}^{\sm}$ and $V_{n}^{\sm}$ is Zariski dense in $V_{n}$, with complement cut out by the functions $f_{n,1},...,f_{n,k_{n}}$ say. We may choose $\epsilon_{n}\in p^{\mb{Q}}$ tending to $0$ such that 
$$ \{ |f_{n,1}|,...,|f_{n,k_{n}}| < \epsilon_{n} \}\cap V_{m} \sub \{  |f_{m,1}|,...,|f_{m,k_{m}}| < 2^{n-m}\epsilon_{m} \} $$
for all $m\leq n$, where we use a general fact that the sets $(\{ |g_{i}| \leq \epsilon \})_{\epsilon\in p^{\mb{Q}}}$ define a cofinal system of neighbourhoods of $\{ |g_{i}|=0 \}$. Then we claim that the sets 
$$ U_{n} = \{  |f_{n,1}|,...,|f_{n,k_{n}}| \geq \epsilon_{n} \} $$
are affinoid and cover $\mc{C}^{\sm}$. Indeed, they cover  $\mc{C}^{\sm}$ by definition, and they are affinoid since they are rational subsets of the $V_{n}$.
\end{rema}

With this we state a simple consequence of Proposition \ref{prop:image}:

\begin{theo}\label{thm:coker}
View $\mc{ES}$ as a morphism of coherent sheaves on $\mc{C}_{\Cp}$. Then $\mc{ES}$ is surjective outside a Zariski-closed subset of $\mc{C}_{\Cp}$ of dimension $0$, and the support of $\coker(\mc{ES})$ is disjoint from the set of non-critical points.
\end{theo}
\begin{proof}
Since $\coker(\mc{ES})$ is coherent its support is Zariski-closed. Hence the first statement follows from the second: if the support is not of dimension $0$, then it must contain some component of $\mc{C}$ and hence non-critical points. To prove the second statement it suffices to prove that $\mc{ES}$ is surjective on fibres at non-critical points. Let $x$ be such a point, with weight $k$. By the computation of the fibres in \S \ref{sec:sheaves} and the computation of $\mc{ES}$ we see that it amounts to the Eichler-Shimura map $ES_{k}$ restricted to maximal semisimple quotients of the generalized eigenspaces for $x$. Since $x$ is non-critical this map is surjective by Proposition \ref{prop:image}. 
\end{proof}

Let us raise the following question:

\medskip

\noindent \textbf{Question}: Is the support of $\coker{\mc{ES}}$ precisely the set of $x$ with weight $k\in \mb{Z}_{\geq 2}$ such that $ES_{k}$ fails to be surjective on generalized eigenspaces for $x$?

\medskip
 
Next let us state some results on the structure of $\mbf{V}$ and $\mc{M}^{\dg}$ over $\mc{C}^{\sm}$. For any $n \in \mb{Z}_{\geq 1}$ we let $\tau(n)$ denote the number of positive divisors of $n$.

\begin{theo} Assume, for simplicity, that that the tame level $K^{p}$ is of the form $\prod_{\ell \nmid \disc(B)}K_{1}(N) \times \prod_{\ell \mid \disc(B)}\oo_{B_{\ell}}^{\times}$ for some $N\geq 3$. Let $M\in \mb{Z}_{\geq 1}$ be a divisor of $N$.
\begin{enumerate}
\item $\mbf{V}$ and $\mc{M}^{\dg}$ are locally free on $\mc{C}^{\sm}$. The rank of $\mbf{V}$ over $\mc{C}_{M-\new}^{\sm}$ is $2\tau(N/M)$ and the rank of $\mc{M}^{\dg}$ over $\mc{C}_{M-\new}^{\sm}$ is $\tau(N/M)$.

\item Let $\ell$ be a prime not dividing $Np\cdot \disc(B)$. Then $\mbf{V}$ is unramified at $\ell$ and the trace of $\Frob_{\ell}$ on $\mbf{V}$ over $\mc{C}^{\sm}_{M-\new}$ is $\tau(N/M)\cdot T_{\ell}$, where we view $T_{\ell}$ as a global function on $\mc{C}$.

\item Let $x \in \mc{C}_{N-\new}^{\sm}(\overline{\mb{Q}}_{p})$. Then the fibre of $\mbf{V}$ at $x$ is the Galois representation attached to $x$ via the theory of pseudorepresentations. 

\end{enumerate}

\end{theo}

\begin{proof}
First we prove (1). By Lemma \ref{lemm:torsionfree} $\mbf{V}$ and $\mc{M}^{\dg}$ are torsion free coherent sheaves over $\mc{C}^{\sm}$, which is smooth of pure dimension $1$. Therefore $\mbf{V}$ and $\mc{M}^{\dg}$ are locally free. To compute the ranks over $\mc{C}_{M-\new}^{\sm}$ it suffices to compute the dimension of the fibres at points in $\mc{C}^{\nc} \cap \mc{C}_{M-\new}^{\sm}$. This is a computation using Atkin-Lehner theory, the control theorems of Stevens resp. Coleman and the classical Eichler-Shimura isomorphism (see e.g \cite[Lemma IV.6.3]{bel}).

\medskip

\noindent To prove (2), note that the statements are true after specializing to points in $\mc{C}^{\nc}$ (using Stevens's control theorem and the well known structure of the Galois representations $H^{1}_{\et}(\X,\mc{V}_{k})$). By Zariski density of $\mc{C}^{\nc}\cap \mc{C}_{M-\new}^{\sm}$ in $\mc{C}^{\sm}_{M-\new}$ we may then conclude. Note that, to check that it is unramified, one may reduce to the affinoid case by Lemma \ref{lemm:qs} and Remark \ref{rema:qsc}, where it becomes a purely ring-theoretic statement. Localizing (ring-theoretically), one may then reduce to the free case, where one can check on matrix coefficients using Zariski density.

\medskip

\noindent Finally, (3) follows from (2). 
\end{proof}

\begin{rema}
The assumption on $K^{p}$ in the above theorem is only to simplify the exposition. The theorem and the definition of the $\mc{C}_{M-\new}$ may be adapted to arbitrary tame levels $K^{p}=\prod_{\ell\neq p}K_{\ell}$ with essentially the same proof, except for the explicit formulas for the ranks.

\medskip

\noindent We remark that the statement that $\mbf{V}$ is unramified at $\ell \nmid Np\cdot \disc(B)$ also follows from the same fact for the $\mbf{V}_{\U}^{s}$ (for $\U$ small open), which in turn follows from the construction upon noting that all members of the tower $(\mc{X}_{K_{p}})_{K_{p}}$ are analytifications of schemes over $\mb{Q}$ with good reduction at $\ell$.
\end{rema}

Finally, we prove that $\mc{ES}$ generically gives us Hodge-Tate filtrations/decompositions on $\mbf{V}$.

\begin{theo}\label{theo:main} We work over $\mc{C}^{\sm}$.
\begin{enumerate}
\item The kernel $\mc{K}$ and image $\mc{I}$ of $\mc{ES}$ are locally projective sheaves of $\oo_{\mc{C}^{\sm}}\ctens_{\Qp}\Cp$-modules, and may also be viewed as locally free sheaves on $\mc{C}^{sm}_{\Cp}$.

\item Let $\epsilon_{\mc{C}^{\sm}}$ be the character of $G_{\Qp}$ defined by the composition
$$G_{\Qp} \overset{\epsilon}{\longrightarrow}  \Zp^{\times} \overset{\chi_{\mc{W}}}{\longrightarrow} \oo_{\mc{W}}^{\times} \longrightarrow  (\oo_{\mc{C}^{\sm}}^{\times}\ctens_{\Qp}\Cp)^{\times} $$
where $\epsilon$ is the $p$-adic cyclotomic character of $G_{\Qp}$ and $\chi_{\mc{W}}$ is the universal character of $\Zp^{\times}$. Then the semilinear action of $G_{\Qp}$ on the module $\mc{K}(\epsilon_{\mc{C}^{\sm}}^{-1})$ is trivial.

\item The exact sequence
$$ 0 \ra \mc{K} \ra \mbf{V}_{\Cp} \ra \mc{I} \ra 0 $$
is locally split. Zariski generically, the splitting may be taken to be equivariant with respect to both the Hecke- and $G_{\Qp}$-actions, and such a splitting is unique.

\end{enumerate}
\end{theo}
\begin{proof}
Part (1) is clear since $\mbf{V}$ and $\mc{M}^{\dg}$ are locally free on $\mc{C}^{\sm}$ and $\oo_{\mc{C}^{\sm}}\ctens_{\Qp}\Cp$ is locally a Dedekind domain. The second statement follows similarly.

\medskip

\noindent For part (2) we use Lemma \ref{lemm:qs} and Remark \ref{rema:qsc} to reduce to the affinoid case with a very Zariski dense set of points in $\mc{C}^{\nc}$. Taking a further cover we may also assume that $\mc{K}$ is free (perhaps after making a finite extension, but this will not effect the rest of the argument). The statement now follows from the family version of Sen theory (\cite{se1}, \cite{se2}) , using the fact that the Sen operator $\phi$ of $\mc{K}(\epsilon_{\mc{C}^{\sm}}^{-1})$ vanishes on points of $\mc{C}^{\nc}$ and hence on $\mc{C}^{\sm}$ by Zariski density of $\mc{C}^{\nc}$ and analyticity of $\phi$. The argument proceeds exactly as the proof of \cite[Theorem 6.1(c)]{oes}, to which we refer for more details.

\medskip

Finally we prove part (3), arguing more or less exactly as in the proof of \cite[Theorem 6.1(d)]{oes}, to which we refer for more details. Once again we may work over an affinoid $U\sub \mc{C}^{\sm}$ with a Zariski dense set of points $U \cap \mc{C}^{\nc}$. We may without loss of generality assume that $\mc{K}$, $\mbf{V}_{\Cp}$ and $\mc{I}$ are free (once again, one might need to make a finite extension, but this does not effect the rest of the argument and we will ignore it). The short exact sequence is an extension which defines an element in $H^{1}(G_{\Qp},\mc{H})$ where
$$\mc{H}:=\Hom_{\oo_{\mc{C}}(U)\ctens_{\Qp}\Cp}(\mc{I},\mc{K}).$$
Let $\phi$ be the Sen operator of $\mc{H}$. By work of Sen it is well-known that $\det(\phi)$ kills $H^{1}(G_{\Qp},\mc{H})$. Specializing at points in $U \cap \mc{C}^{\nc}$ one sees that $\det(\phi)$ does not vanish identically on any component of $U$. Thus, localizing with respect to $\det(\phi)$, we find a Zariski open subset of $U$ over which the extension is split as semilinear $G_{\Qp}$-representations. Moreover, if we further remove the finite set of points whose weight is $-1\in \mb{Z}$, then $\mc{I}$ and $\mc{K}$ have distinct Hodge-Tate weights fibre-wise (by using part (2) and that $\mc{I}$ has constant Hodge-Tate weight $-1$). Thus there can be no non-zero Galois equivariant homomorphisms between $\mc{K}$ and $\mc{I}$. Since the Hecke action commutes with the $G_{\Qp}$-action, this implies that the $G_{\Qp}$-splitting must be Hecke stable as well. 
\end{proof}

\section{Appendix}\label{sec:6}

In this appendix we define a "mixed" completed tensor product that we will use in the main text, and prove some basic properties as well as a few technical results that we need. 

\subsection{Mixed completed tensor products}\label{sec:6.1} Let $K$ be a finite extension of $\Qp$ and let $\oo$ be its ring of integers, with uniformizer $\vp$. Our "mixed completed tensor products" will be denoted by an unadorned $\ctens$. The base field is assumed to be implicit in the notation. In the main text we will always take $K=\Qp$ so this should not cause any problem.  We starting by making the following definition:

\begin{defi} Let $M$ be a topological $\oo$-module.
\begin{enumerate}
\item $M$ is called \emph{linear-topological} if there exists a basis of neighbourhoods of $0$ consisting of $\oo$-submodules.

\item We will say that $M$ is a \emph{profinite flat} $\oo$-module if $M$ is flat over $\oo$ (i.e torsion-free), linear-topological and compact.
\end{enumerate}
\end{defi}

Let us remark that if $M$ is a profinite flat $\oo$-module then the topology on $M$ is profinite, which justifies this terminology. We also remark that such $M$ are exactly the $\oo$-modules which are projective and pseudocompact in the language of \cite[Expos\'e ${\rm VII}_{B}$, \S 0]{SGA3}. Let us recall the following structure theorem for profinite flat $\oo$-modules, specialized from \cite[Expos\'e ${\rm VII}_{B}$, 0.3.8]{SGA3}:

\begin{prop}
M is a profinite flat $\oo$-module if and only if it is isomorphic to $\prod_{i\in I}\oo $ equipped with the product topology, for some set $I$.
\end{prop}

 We note that a set of elements $(e_{i})_{i\in I}$ in $M$ such that $M\cong \prod_{i\in I}\oo e_{i}$ is called a \emph{pseudobasis}. In what follows `$\vp$-adically complete" always means complete and separated.

\begin{defi}\label{defi: ctens}
Let $M$ be a profinite flat $\oo$-module and let $X$ be any $\oo$-module. We define
$$ X \ctens M := \varprojlim _{i} (X \otimes_{\oo} M/I_{i}) $$
where $(I_{i})$ runs through any cofinal set of neighbourhoods of $0$ consisting of $\Zp$-submodules.
\end{defi}

By abstract nonsense this is independent of the choice of system of neighbourhoods of $0$. Note that if $A$ is an $\oo$-algebra and $R$ is a small $\oo$-algebra, then $A \ctens R$ comes with a natural structure of a ring, as is seen by choosing the neighbourhoods to be ideals. To compute $X \ctens M$ in a useful form we fix a pseudobasis $(e_{i})_{i\in I}$ and define, for $J\sub I$ finite and $n\geq 1 $ an integer, $M_{J,n}$ to be the submodule corresponding to
$$ M_{J,n}:= \prod_{i\in J}\vp^{n} \oo e_{i} \times \prod_{i\notin J} \oo e_{i}. $$
This forms a system of neighbourhoods of $0$. Then we have:

\begin{prop}\label{prop: appflat}
The functor $X \mapsto X \ctens M$ is isomorphic to the functor $X \mapsto \prod_{i\in I} \wh{X}$, where $\wh{X}$ is the $\vp$-adic completion of $X$. 
\end{prop}

\begin{proof}
We compute
\begin{eqnarray*}
X \ctens M & = & \varprojlim _{J,n} X \otimes_{\oo} M/M_{J,n}\\
 & = &  \varprojlim _J (\varprojlim _n X \otimes_{\oo} M/M_{J,n}) \\
 & \cong & \varprojlim _J (\varprojlim _n X \otimes_{\oo} \prod_{i\in J} \oo/ \vp^{n}) \\
 & = & \varprojlim _J \left( \prod_{i\in J} (\varprojlim _n X \otimes_{\oo} \oo/\vp^{n}) \right) \\
 & = & \varprojlim _J \left( \prod_{i\in J} \wh{X} \right) \\
 & = & \prod_{i\in I} \wh{X}.\end{eqnarray*} 
\end{proof}

\begin{coro}\label{cor:complete}
We have
\begin{enumerate}
\item $X \mapsto X \ctens M$ is exact on $\vp$-adically complete $X$.

\item If $0 \ra X \ra Y \ra Q \ra 0$ is an exact sequence of $\oo$-modules where $X$ is $\vp$-adically complete and $Q$ is killed by $\vp^{N}$ for some integer $N\geq 0$, then $Y$ is $\vp$-adically complete and the natural map $(X \ctens M)[1/\vp] \ra (Y \ctens M)[1/\vp]$ is an isomorphism.
\end{enumerate}
\end{coro}

\begin{proof}
(1) follows from exactness of products in the category of $\oo$-modules. For (2) note that $Q$ is $\vp$-adically complete and $Q \ctens M$ is killed by $\vp^{N}$, so the second statement will follow from the first and (1). To see that $Y$ is $\vp$-adically complete first note that $\vp^{N}Y \sub X$ (i.e. $Q$ killed by $\vp^{N}$) so 
$$ \vp^{N+k}Y \sub \vp^{k}X \sub \vp^{k}Y $$
for all $k\geq 0$ and hence we can compute $\wh{Y}$ as $\varprojlim_{k} Y/\vp^{k}X$. Thus we get exact sequences
$$ 0 \ra X/\vp^{k}X \ra Y/\vp^{k}X \ra Q \ra 0$$
for all $k$ and hence an exact sequence
$$ 0 \ra X=\wh{X} \ra \wh{Y} \ra Q $$
by left exactness of inverse limits. This sits in a natural diagram
\[
\xymatrix{0 \ar[r] & X \ar[d]\ar[r] & Y\ar[d]\ar[r] & Q\ar[d] \ar[r]& 0\\
0 \ar[r] & X \ar[r] & \wh{Y} \ar[r] & Q}
\] 
with exact rows and the snake lemma now implies that the map $Y\ra \wh{Y}$ is an isomorphism as desired.
\end{proof}

In particular we can now define $- \ctens M$ on $K$-Banach spaces (this is why we are calling it a "mixed" completed tensor product). It is canonically independent of the choice of unit ball, but we also make a more general definition.

\begin{defi}\label{mixtens}
Let $V$ be a $K$-vector space and let $V^{\circ} \sub V$ be an $\oo_{K}$-submodule such that $V^{\circ}[1/\vp]=V$. Then we define 
$$ V \ctens M := (V^{\circ} \ctens M)[1/\vp] $$
where the choice of $V^{\circ}$ is implicit (the choice will in general affect the result).
\end{defi}
If $V$ is naturally a Banach space we will always take $V^{\circ}$ to be an open and bounded $\oo$-submodule, and Corollary \ref{cor:complete}(2) implies that $V \ctens M$ is independent of the choice.

We finish with a lemma needed in the main text.

\begin{lemm}\label{lemm: commute}
Let $S$ be a $K$-Banach algebra and let $V$ be a finite projective $S$-module and $U$ a Banach $S$-module. Let $M$ be a profinite flat $\oo_{K}$-module. Then we have a natural isomorphism $V\otimes_{S}(U \ctens M) \cong (V\otimes_{S}U) \ctens M$. If $V$ and $U$ in addition are Banach $S$-algebras and $M$ is a small $\oo_{K}$-algebra, then this is an isomorphism of rings.
\end{lemm}

\begin{proof}
Pick an open bounded $\oo_{K}$-subalgebra $S_{0}$ of $S$ and pick open bounded $S_{0}$-submodules $V_{0}\sub V$ and $U_{0} \sub U$. Moreover, pick $\oo_{K}$-submodules $(I_{i})$ of $M$ that form a basis of neighbourhoods of $0$. Then we have natural maps
$$ V_{0} \otimes_{S_{0}}( \varprojlim _i (U_{0} \otimes_{\oo_{K}} M/I_{i})) \ra V_{0} \otimes_{S_{0}} U_{0} \otimes_{\oo_{K}} M/I_{i} $$
for all $i$ which give us a natural map
$$ V_{0} \otimes_{S_{0}}(U_{0} \ctens M) \ra (V_{0} \otimes_{S_{0}} U_{0})\ctens M. $$
Now invert $\vp$ to get the desired map. If $V$ and $U$ are Banach $S$-algebras we may choose everything so that $V_{0}$ and $U_{0}$ are subrings, and if $M$ is a small $\oo_{K}$-algebra we may choose the $I_{i}$ to be ideals, hence we see that this is a ring homomorphism. To see that it is an isomorphism note that if $V$ is in addition free we may choose $V_{0}$ to be a finite free $S_{0}$-module. The assertion is then clear. In general, write $V$ as a direct summand of a finite free $S$-module.
\end{proof}

\subsection{Mixed completed tensor products on rigid-analytic varieties}\label{sec:6.2}

Let $X$ be a quasicompact and separated smooth rigid analytic variety over an algebraically closed and complete extension $C$ of $K$, of pure dimension $n$. Here $K$ is a finite extension of $\Qp$ as in the previous subsection and we use the same notation $\oo=\oo_{K},\vp$ et cetera. In this subsection we aim to prove:

\begin{prop}\label{prop:comp-r}
We have a canonical isomorphism
$$R^{i}\nu_{\ast}(\ohat_{X} \ctens M) \cong (R^{i}\nu_{\ast}\ohat_{X})\ctens M$$
of sheaves on $X_{\et}$.
\end{prop}
Here $U \mapsto (\ohat_{X} \ctens M)(U) = \ohat_{X}(U) \ctens M$ for $U\in X_{\proet}$ qcqs defines a sheaf on $X_{\proet}$ and the completed tensor product is computed using $\ohatp_{X}(U) \sub \ohat_{X}(U)$. The former is $\vp$-adically complete by \cite[Lemma 4.2(iii)]{sch2}, so we have exactness of $-\ctens M$ and hence a sheaf. We use ${\rm Im}((R^{i}\nu_{\ast}\ohatp_{X})(U) \ra (R^{i}\nu_{\ast}\ohat_{X})(U))$ to compute the completed tensor product on the right hand side and claim that this is a sheaf, but these statements are not obvious and we will prove them below.

\medskip

We note the following important corollary to the above proposition.

\begin{coro}\label{coro:comp-r}
We have $$R^{i}\nu_{\ast}(\ohat_{X} \ctens M) \cong \Omega^i_{X}(-i)\otimes_{\oo_X}(\oo_{X}\ctens M)$$ as sheaves on $X_{\et}$. \end{coro}

\begin{proof} This follows from Proposition \ref{prop:comp-r}, Lemma \ref{lemm: commute} and the canonical isomorphism $R^{i}\nu_{\ast}(\ohat_{X}) \cong \Omega^{i}_{X}(-i)$ (Proposition 3.23 of \cite{sch1}) together with the finite projectivity of the latter sheaf. \end{proof}

\medskip

We will need several lemmas before we can prove Proposition \ref{prop:comp-r}.

\begin{lemm}\label{lem:sheaf}
The presheaf $U \mapsto H^{i}_{cts}(\Zpn,\oo_{X}^{+}(U))$ on $X_{\et}$ is a sheaf, where $\oo_{X}^{+}(U)$ has its $\vp$-adic topology and the trivial $\Zpn$-action.
\end{lemm}

\begin{proof}
This follows from the computation in \cite[Lemma 5.5]{sch2}; $H^{i}_{cts}(\Zpn,\oo_{X}^{+}(U))$ is isomorphic to $\bigwedge^{i}\oo_{X}^{+}(U)^{n}$ in a way compatible with completed tensor products.
\end{proof}

Put $T^{n}=\Spa(C \langle X_{1}^{\pm 1},...,X_{n}^{\pm 1} \rangle, \OC \langle X_{1}^{\pm 1},...,X_{n}^{\pm 1} \rangle )$. We let $\wt{T}^{n}$ be the inverse limit over $n \geq 0$ of $\Spa (C\langle X^{\pm 1/p^n} _1,..., X^{\pm 1/p^n} _{n} \rangle, \OC \langle X^{\pm 1/p^n} _1,..., X^{\pm 1/p^n} _{n} \rangle)$. This is an affinoid perfectoid (see \cite[Example 4.4]{sch2}).

\begin{lemm}\label{lem:cts}
Let $U\in X_{\et}$ and and assume that there is a map $U \ra T^{n}$ which is a composition of rational embeddings and finite \'etale maps. Pull back the affinoid perfectoid $\Zpn$-cover $\wt{T}^{n}$ to obtain an affinoid perfectoid $\Zpn$-cover $\wt{U}$ of $U$. Then:
\begin{enumerate}

\item There is a canonical injection
$$H^{i}_{cts}(\Zpn, \oo_{X}^{+}(U)) \ra H^{i}(U_{\proet},\ohatp_{X})$$ 
with cokernel killed by $p$.

\smallskip

\item $H^{i}(U_{\proet},\ohat_{X}) \cong H^{i}_{cts}(\Zpn, \oo_{X}(U))$.

\item $H^{i}(U_{\proet},\ohatp_{X} \ctens M)^{a} \cong (H^{i}(U_{\proet},\ohatp_{X}) \ctens M)^{a}$ (where $-^{a}$ denotes the associated almost $\OC$-module).

\item $H^{i}(U_{\proet},\ohat_{X} \ctens M) \cong H^{i}(U_{\proet},\ohat_{X}) \ctens M$.

\end{enumerate}
\end{lemm}

\begin{proof}
We start with (1). The Cartan--Leray spectral sequence (cf. Remark \ref{cartanleray}) gives us
$$ H^{p}_{cts}(\Zpn, H^{q}(\wt{U}_{\proet}, \oo_{X}^{+}/\vp^{m}))\, \implies \, H^{p+q}(U_{\proet}, \oo_{X}^{+}/\vp^{m}) $$
for all $m$. Since the $H^{q}(\wt{U}_{\proet}, \oo^{+}_{X}/\vp^{m})$ are almost zero for all $q\geq 1$ (\cite[Lemma 4.10]{sch2}, see the proof of (i) and (v) ) we conclude that
$$ H^{p}(U_{\proet}, \oo_{X}^{+}/\vp^{m})^{a} = H^{p}(\Zpn, (\oo_{X}^{+}/\vp^{m})(\wt{U}))^{a}. $$
Next, note that the inverse system
$$ (H^{p}(\Zpn, (\oo_{X}^{+}/\vp^{m})(\wt{U}))_{m} $$
has surjective transition maps by the proof of \cite[Lemma 5.5]{sch2}. Since the inverse system $(\oo_{X}^{+}/\vp^{m})_{m}$ has almost vanishing higher inverse limits on the pro-\'etale site (by an easy application of the almost version of \cite[Lemma 3.18]{sch2}) we conclude that there is an almost isomorphism
$$ \varprojlim H^{p}(\Zpn, (\oo_{X}^{+}/\vp^{m})(\wt{U})^{a} \cong H^{p}(U_{\proet},\wh{\oo}_{X}^{+})^{a}. $$
By \cite[Lemma 5.5]{sch2} we have
$$  \bigwedge^{p}(\oo_{X}^{+}/\vp^{m})(U)^{n} \cong H^{p}(\Zpn, (\oo_{X}^{+}/\vp^{m})(U)) \ra H^{p}(\Zpn, (\oo_{X}^{+}/\vp^{m})(\wt{U})) $$
with cokernel killed by $p$, compatibly in $m$; in fact the proof shows that the injection is split. It remains to verify that $\varprojlim H^{q}(\Zpn, (\oo_{X}^{+}/\vp^{m})(U)) \cong H_{cts}^{q}(\Zpn, \oo_{X}^{+}(U))$. This is standard; one may e.g. argue as in the proof of \cite[Theorem 2.5.7]{nsw} (note that the relevant $\varprojlim^{1}$ vanishes by the isomorphism above). This finishes the proof of (1).
\medskip

For (2), use (1) and invert $\vp$ - this commutes with taking cohomology by quasicompactness.

\medskip

For (3) we view $- \ctens M$ as the functor of taking self products over some index set $I$. We remark that $H^{i}(U_{\proet},\ohatp_{X})^{a}$ is $\vp$-adically complete by part (1), \cite[Lemma 5.5]{sch2} and the almost version of Corollary \ref{cor:complete}(3). Thus we are left with verifying that cohomology for $\ohatp_{X}$ almost commutes with arbitrary products. Writing $\prod_{i\in I}$ as $\varprojlim _J \prod_{i\in J}$ for all $J\sub I$ finite we are left with showing that the relevant higher inverse limits vanish, which is easy using the Mittag-Leffler condition and \cite[Lemma 3.18]{sch2}.

\medskip

Finally we invert $\vp$ in (3) to deduce (4).
\end{proof}

\begin{lemm}\label{lem:push}
The sheaf $R^{i}\nu_{\ast}\ohat_{X}$ is equal to the sheaf $U \mapsto H^{i}_{cts}(\Zpn, \oo_{X}(U))$, and $(R^{i}\nu_{\ast}\ohat_{X}) \ctens M$ may be computed using $H^{i}_{cts}(\Zpn, \oo_{X}^{+}(U)) \sub R^{i}\nu_{\ast}\ohat_{X}(U)$.
\end{lemm}

\begin{proof}
We may work \'etale locally on $X$, so assume without loss of generality that there is a map $X \ra T^{n}$ which is a composition of rational embeddings and finite \'etale maps. By abstract nonsense $R^{i}\nu_{\ast}\ohat_{X}$ is the sheafification of the presheaf $U \mapsto H^{i}(U_{\proet},\ohat_{X})$. Let us write $\mc{F}$ for this presheaf.  By part (1) of the previous Lemma there is an exact sequence 
$$ 0 \ra  H^{i}_{cts}(\Zpn, \oo_{X}^{+}(U)) \ra H^{i}(U_{\proet}, \ohatp_{X}) \ra  Q^{i}(U) \ra 0 $$
where $Q^{i}(U)$ is simply defined as the quotient. This is  an exact sequence of presheaves on $\Xetw$. Moreover $Q^{i}(U)$ is killed by $p$ for all $U$. Write $\mc{G}$ for the presheaf $U \mapsto H^{i}_{cts}(\Zpn, \oo_{X}^{+}(U))$, $\mc{F}^{+}$ for $U \mapsto H^{i}(U_{\proet}, \ohatp_{X})$ and $\mc{Q}$ for $U \mapsto Q^{i}(U)$. Then the exact sequence above sheafifies to
$$ 0 \ra \mc{G} \ra (\mc{F}^{+})^{sh}=R^{i}\nu_{\ast}\ohatp_{X} \ra \mc{Q}^{sh} \ra 0$$
( $-^{sh}$ for sheafification) since $\mc{G}$ is already a sheaf by Lemma \ref{lem:sheaf}, and $\mc{Q}^{sh}$ is killed by $p$ since $\mc{Q}$ is. Inverting $\vp$ we get that $H^{i}_{cts}(\Zpn, \oo_{X}^{+}(U))[1/\vp] \cong R^{i}\nu_{\ast}\ohat_{X}(U)$. For the other assertion of the lemma, note that ${\rm Im}(R^{i}\nu_{\ast}\ohatp_{X}(U) \ra R^{i}\nu_{\ast}\ohat_{X}(U))$ coincides with $H^{i}_{cts}(\Zpn, \oo_{X}^{+}(U))$ under the identification $H^{i}_{cts}(\Zpn, \oo_{X}^{+}(U))[1/\vp] \cong R^{i}\nu_{\ast}\ohat_{X}(U)$, and is therefore a sheaf (as was asserted in Proposition \ref{prop:comp-r}).
\end{proof}

\medskip

\noindent \textbf{Proof of Proposition \ref{prop:comp-r}}: Lemma \ref{lem:push} gives us that
$$ ((R^{i}\nu_{\ast}\ohat_{X}) \ctens M)(U)   =   (H^{i}_{cts}(\Zpn, \oo_{X}^{+}(U))\ctens M)[1/\vp]. $$
On the other hand $R^{i}\nu_{\ast}(\ohat_{X} \ctens M)$ is the sheaf associated with the presheaf 
$$U \mapsto H^{i}(U_{\proet},\ohat_{X} \ctens M)$$ 
by abstract nonsense. Working locally again, assume that there is a map $U \ra T^{n}$ as above. We may then compute
\begin{eqnarray*}
H^{i}(U_{\proet},\ohat_{X} \ctens M) & = & H^{i}(U_{\proet},\ohat_{X}) \ctens M \\
 & = &  ({\rm Im}(H^{i}(U_{\proet},\ohatp_{X}) \ra H^{i}(U_{\proet},\ohat_{X})) \ctens M)[1/\vp] \\
 & = &  (H^{i}_{cts}(\Zpn, \oo_{X}^{+}(U))\ctens M)[1/\vp] \end{eqnarray*}
using Lemma \ref{lem:cts} (4) and the proof of Lemma \ref{lem:push}. Therefore $U \mapsto H^{i}(U_{\proet},\ohat_{X} \ctens M)$ is actually defines sheaf when restricting to the basis of such $U$ (since $U \mapsto H^{i}_{cts}(\Zpn, \oo_{X}^{+}(U)) $ is a sheaf), so it must equal $R^{i}\nu_{\ast}(\ohat_{X} \ctens M)$ on these $U$. This gives the desired isomorphism. \qed

\subsection{Analogues of the theorems of Tate and Kiehl}\label{sec:6.3}

Let $X$ be a quasicompact and quasiseparated rigid analytic variety over some complete non-archimedean field $K$ of \emph{characteristic zero}. For simplicity let us work with the site given by the basis of the topology consisting of quasicompact open subsets together with finite covers (we will specialize this further eventually). Given a small $\Zp$-algebra $R$, we may form the sheaf $\mc{R}(U)=\oo_{X}(U) \ctens R$ on $X$ ($U$ is qc open). In this section we are going to prove analogues of the theorems of Tate and Kiehl for $\oo_{X}$ for $\mc{R}$ under the assumption that $X$ is affinoid and that $K$ is discretely valued. The assumption that $K$ is discretely valued makes it easy to prove various flatness assertions. We suspect that this assumption can be dropped with some more work but we will not need the extra generality. We will however prove a few statements that we need without this assumption on $K$. For simplicity we work with the topology on $X$ given by rational subsets and finite covers. We start with some basic properties.

\begin{lemm}\label{lemm:flat} Let $X$ be affinoid and assume that $K$ is discretely valued.
\begin{enumerate}

\item For every rational $U\sub X$, $\mc{R}(U)$ is Noetherian.

\item If $V\sub U \sub X$ are rational, then the map $\mc{R}(U) \ra \mc{R}(V)$ is flat.

\item If $U$ is rational and $(U_{i})_{i}$ is a finite rational cover of $U$, then the natural map $\mc{R}(U) \ra \prod_{i}\mc{R}(U_{i})$ is faithfully flat.
\end{enumerate}
\end{lemm}

\begin{proof}
To prove (1), choose surjections $\oo_{K} \langle X_{1},...,X_{m} \rangle \surj \oo_{X}(U)$, $\Zp[[T_{1},...,T_{n}]] \surj  R$ and use the image of $\Zp \langle X_{1},...,X_{m} \rangle$ to compute $\oo_{X}(U) \ctens R$ (this is valid by the open mapping theorem). Then we get a surjection
$$ \oo_{K}\langle X_{1},...,X_{m} \rangle\ctens \Zp[[T_{1},...,T_{n}]] \surj \oo_{X}(U) \ctens R $$
and by direct computation the left hand side is $(\oo_{K} \langle X_{1},...,X_{m} \rangle [[T_{1},...,T_{n}]])[1/p]$, which is Noetherian (since $K$ is discretely valued), hence the right hand side is Noetherian.

\medskip

Next we prove (2). Pick functions $f_{1},...,f_{n},g$ such that $V=\{ |f_{1}|,...,|f_{n}|\leq |g| \neq 0 \}$. We factor $\mc{R}(U) \ra \mc{R}(V)$ as
$$ \oo_{X}(U) \ctens R \ra \oo_{X}(U)[1/g] \otimes_{\oo_{X}(U)} (\oo_{X}(U) \ctens R) \ra \oo_{X}(V) \ctens R. $$
The first morphism is flat since $\oo_{X}(U) \ra \oo_{X}(U)[1/g]$ is, so it suffices to show that the second morphism is flat. It is obtained from
$$ \oo_{X}^{\circ}(U)[f_{1}/g,...,f_{n}/g] \otimes_{\oo_{X}^{\circ}(U)} (\oo_{X}^{\circ}(U) \ctens R) \ra \oo_{X}^{\circ}(V) \ctens R $$
by inverting $p$, so it suffices to show that this map is flat. But the right hand side is the completion of the left hand side with respect to the ideal generated by the images of the generators of the maximal ideal of $R$. Since the ring on the left hand side is Noetherian (it is finitely generated over $\oo_{X}^{\circ}(U) \ctens R$, which is Noetherian by the same proof as in (1)), we conclude that the map is flat.

\medskip

For (3) we use a more geometric argument. By part (2) the map
$$ \oo_{X}(U) \ctens R \ra \prod_{i} \oo_{X}(U_{i}) \ctens R $$
is flat. To see that it is faithfully flat, it suffices to show that it induces a surjection on spectra of \emph{maximal} ideals, by \cite[\S\S I.3.5, Proposition 9]{bou}. To simplify notation, put $S=\oo_{X}(U)$ and $S_{i}=\oo_{X}(U_{i})$. Since we are only interested in maximal ideals, we will momentarily work with classical rigid spaces rather than adic spaces in this proof (only). $R$, $S^{\circ}$ and the $S^{\circ}_{i}$ are formally of finite type over $\oo_{K}$, and the unadorned completed tensor products $S^{\circ} \ctens  R$ etc. are completed tensor products over $\oo_{K}$ in the category of topological $\oo_{K}$-algebras formally of finite type. Recall Berthelot's generic fibre functor $(-)^{rig}$ (see e.g. \cite[\S 7]{dj}). By \cite[Proposition 7.2.4(g)]{dj} the generic fibre functor commutes with fibre products, so the morphism
$$ \bigsqcup_{i} U_{i} \times_{\Sp K} \Spf(R)^{rig} =\bigsqcup_{i}\Spf (S_{i}^{\circ} \ctens R)^{rig} \ra \Spf (S^{\circ} \ctens R)^{rig} = U \times_{\Sp K} \Spf(R)^{rig} $$
is surjective (it is a finite open cover). By \cite[Lemma 7.1.9]{dj} we deduce that we have a surjection
$$ \bigsqcup_{i}\MaxSpec (S_{i} \ctens R) = \MaxSpec \left( \prod_{i} S_{i} \ctens R \right) \ra \MaxSpec (S\ctens R), $$
which is what we wanted to prove.
\end{proof}

Having established this Lemma we now view $\mc{R}$ as a sheaf of Banach algebras on $X$ by picking the unit ball of $\mc{R}(U)$ to be $\mc{R}^{\circ}(U)=\oo_{X}^{\circ}(U) \ctens R$. Then all restriction maps are norm-decreasing, hence continuous, and if we view $R$ as a profinite flat $\Zp$-module and choose a pseudobasis with index set $I$, the Banach space structure on $\mc{R}(U)$ is the same as the natural Banach space structure on the set of bounded sequences in $\oo_{X}(U)$ indexed by $I$ (of course the multiplication is very different from the pointwise multiplication on the latter). We will follow \cite[\S 5]{aw} (which in turn follows \cite[\S 4.5]{fvp}) closely in the proof of our analogue of Kiehl's theorem.

\begin{defi}
Let $U\sub X$ be rational and let $M$ be a finitely generated $\mc{R}(U)$-module. For every $V\sub U$ rational, we define
$$ \Loc(M)(V) := \mc{R}(V) \otimes_{\mc{R}(U)} M. $$ 
This is a presheaf of $\mc{R}$-modules of $U$.
\end{defi}

\begin{prop}[Tate's theorem] \label{prop:thmB}
 For every rational $U\sub X$, $\check{H}^{i}(\mf{U},\mc{R})=0$ for all $i\geq 1$ and all finite rational covers $\mf{U}$ of $U$, i.e. the augmented \v{C}ech complex for $\mf{U}$ is exact. As a consequence, $H^{i}(U,\mc{R})=0$ for all $i\geq 1$.
\end{prop}

\begin{proof}
The second statement is a direct consequence of the first by a well known theorem of Cartan. For the first, pick a pseudobasis $(e_{i})_{i\in I}$ of $R$, considered as a profinite flat $\Zp$-module. For a cover $\mf{U}$ as above, write $C^{\bullet}_{aug}(\mf{U},-)$ for the augmented \v{C}ech complex functor. Then, using the pseudobasis, we have 
$$C^{\bullet}_{aug}(\mf{U},\mc{R})=C^{\bullet}_{aug}(\mf{U},\mc{R}^{\circ})[1/p] \cong \left( \prod_{i\in I} C^{\bullet}_{aug}(\mf{U},\oo_{X}^{\circ}) \right)[1/p]. $$  
Note that the cohomology groups $\check{H}^{i}(\mf{U},\oo_{X}^{\circ})$ are bounded $p$-torsion for $i\geq 1$ by Tate's theorem and the open mapping theorem, hence by exactness of products and the above displayed equation so are the cohomology groups $\check{H}^{i}(\mf{U}, \mc{R}^{\circ})$ for $i\geq 1$. Thus $\check{H}^{i}(\mf{U},\mc{R})$ vanish as desired.
\end{proof}

\begin{prop} \label{prop:loc}
Assume that $K$ is discretely valued. Then functor $\Loc$ defines a full exact embedding of abelian categories from the category of finitely generated $\mc{R}(U)$-modules into the category of $\mc{R}$-modules on $U$. Moreover, $\check{H}^{i}(\mf{U},\Loc(M))=0$ and $H^{i}(U,\Loc(M))=0$ for all $i\geq 1$ and all finite rational covers $\mf{U}$ of $U$ for all finitely generated $\mc{R}(U)$-modules $M$. 
If $K$ is not discretely valued, $\Loc(M)$ is still a sheaf with vanishing higher cohomology groups if $M$ is a finitely generated projective module. 
\end{prop}

\begin{proof}
We start with the sheaf/vanishing assertions. Again, the statement about derived functor cohomology follows from that of \v{C}ech cohomology by Cartan's theorem, and the assertion about \v{C}ech cohomology implies that $\Loc(M)$ is  a sheaf. Pick a cover $\mf{U}$. Then by Proposition \ref{prop:thmB} the complex $C^{\bullet}_{aug}(\mf{U},\mc{R})$ is exact. If $M$ is projective, then it is flat and hence $C^{\bullet}_{aug}(\mf{U},\mc{R})\otimes_{\mc{R}(U)}M = C^{\bullet}_{aug}(\mf{U},\Loc(M))$ is exact. If $K$ is discretely valued, then by the flatness properties of $\mc{R}$ all terms are flat $\mc{R}(U)$-modules and we may deduce that $C^{\bullet}_{aug}(\mf{U},\mc{R})\otimes_{\mc{R}(U)}M = C^{\bullet}_{aug}(\mf{U},\Loc(M))$ is exact for arbitrary finitely generated $M$. In either case, it follows that $\Loc(M)$ is a sheaf with vanishing higher \v{C}ech cohomology. 

\medskip

Now let $K$ be discretely valued. To see that the functor is fully faithful, note that $\Loc(M)$ is generated by global sections as an $\mc{R}$-module. For exactness, one checks that if $f\, :\, M \ra N$ is a morphism of finitely generated $\mc{R}(U)$-modules, then $\ker(\Loc(f))=\Loc(\ker(f))$ and similarly for images and cokernels (in particular, images and cokernels in this sub-abelian category turn out to be equal to the presheaf images and cokernels).
\end{proof}

\begin{rema}
 Propositions \ref{prop:thmB} and \ref{prop:loc} also hold for the \'etale site of $X$, with the same proofs with obvious modifications.
\end{rema}

We may now introduce coherent $\mc{R}$-modules. 
\begin{defi}
Assume that $K$ is discretely valued. For a cover $\mf{U}$ of $U$ as above a sheaf $\mc{F}$ of $\mc{R}$-modules on $U$ is called \emph{$\mf{U}$-coherent} if for every $U_{i}\in \mf{U}$ there exists an $\mc{R}(U_{i})$-module $M_{i}$ such that $\mc{F}|_{U_{i}}\cong \Loc(M_{i})$. 

\medskip

$\mc{F}$ is said to be \emph{coherent} if it is $\mf{U}$-coherent for some $\mf{U}$.
\end{defi}
The equivalence of this definition and the usual definition is standard, cf. e.g. \cite[\S 5.2]{aw}. 

\medskip

Before proving Kiehl's theorem we need a simple lemma. Pick $f\in \oo_{X}(X)$  and consider the standard cover of $X$ consisting of $X(f)=\{ |f|\leq 1 \}$ and $X(1/f)=\{ |f|\geq 1 \}$. We denote the intersection by $X(f,1/f)$. We let $s$ denote the restriction map $\oo_{X}(X(1/f)) \ra \oo_{X}(X(f,1/f))$.

\begin{lemm}\label{lem:dense}
With notation as above:
\begin{enumerate}
\item  The image of $s$ is dense.

\item The image of the restriction map $\mc{R}(X(1/f)) \ra \mc{R}(X(f,1/f))$ is dense for the natural Banach algebra structure on the target.
\end{enumerate}

\end{lemm}

\begin{proof}
Statement (1) is well known, but since there seems to be a little confusion in the literature let us give the one-sentence proof: $s$ is the completion of the identity map on $\oo_{X}(X)[1/f]$ with respect to the group topology generated by $(\vp^{n}\oo_{X}^{\circ}(X)[1/f])_{n\geq 0}$ on the source and the group topology generated by $(\vp^{n}\oo_{X}^{\circ}(X)[f,1/f])_{n\geq 0}$ on the target.

\medskip

To prove (2), pick a pseudobasis $(e_{i})_{i\in I}$ for $R$. Then the map is the natural map from bounded $I$-sequences in $\oo_{X}(X(1/f))$ to bounded $I$-sequences in $\oo_{X}(X(f,1/f))$, which has dense image since $s$ does.
\end{proof}

We may now follow the proof of Kiehl's theorem from \cite[\S 5.3-5.5]{aw}. One firstly establishes an auxiliary lemma (\cite[Lemma 5.3]{aw}), which only uses the Banach algebra structure and vanishing of cohomology (in our setting this is Proposition \ref{prop:thmB}). Using this auxiliary lemma one proves \cite[Theorem 5.4]{aw} on surjections of certain maps (upon noting that, in their notation, it is erroneously asserted in \S 5.3 that $s_{1}$ rather than $s_{2}$ has dense image, see Lemma \ref{lem:dense}). Then Ardakov and Wadsley establish a weak form of Kiehl's theorem for the covering $\{X(f),X(1/f)\}$ (\cite[Corollary 5.4]{aw}), using \cite[Theorem 5.4]{aw} and properties of $\Loc$ functor, which is Proposition \ref{prop:loc} in our setting. From this, \cite[Theorem 5.5]{aw} follows as well. Following their arguments, we get

\begin{theo}[Kiehl's Theorem]\label{theo: thmA} Assume that $K$ is discretely valued and let $\mc{F}$ be a coherent $\mc{R}$-module on $X$. Then there exists a finitely generated $\mc{R}(X)$-module $M$ such that $\mc{F}\cong \Loc(M)$. Moreover, if $\mc{F}$ is locally projective (i.e. there is a cover $\mf{U}$ such that $\mc{F}(U)$ is projective for all $U\in \mf{U}$), then $M$ is projective.
\end{theo}

\begin{proof}
For the last part, use that $\mc{R}(X) \ra \prod_{U\in \mf{U}}\mc{R}(U)$ is faithfully flat and that projectivity for finitely presented modules may be checked after a faithfully flat base extension.
\end{proof}

We record a simple base change lemma.

\begin{lemm}\label{lemm: bc}
Assume that $K$ is discretely valued and that $C$ is the completion of an algebraic closure $\overline{K}$ of $K$. Let $X/K$ be an affinoid rigid space and let $M$ be a finitely generated projective $\mc{R}(X)$ module. Let $X_{C}$ be the base change of $X$ to $C$ and let $f\, :\, X_{C} \ra X$ be the natural map, and let $\mc{R}_{C}$ be the sheaf $U \mapsto \oo_{X_{C}}\ctens R$ on $X_{C}$. Then we have natural isomorphisms $M\otimes_{\mc{R}(X)}\mc{R}_{C}(X)\cong M\otimes_{K} C$ and
$$ f^{-1}(\Loc(M)) \otimes_{f^{-1}\mc{R}} \mc{R}_{C} \cong \Loc(M\otimes_{K}C). $$
\end{lemm} 

\begin{proof}
The proof of the first assertion is similar to the proof of Lemma \ref{lemm: commute}. Using this, the second assertion follows straight from the definitions upon noting that $f$ is open, which makes it easy to compute $f^{-1}$.
\end{proof}

\begin{rema}
When working with affinoid weights in the main part of the paper we will often need analogues of the results of this section (as well as Lemma \ref{lemm: commute}) when one replaces $-\ctens R$ by $-\wh{\otimes}_{\Qp}S$ in the definition of $\mc{R}$, where $S$ is a reduced $\Qp$-Banach algebra of topologically of finite type. In this case these results are classical results in rigid geometry, or straightforward consequences of such results.
\end{rema}

\subsection{Quotients of rigid spaces by finite groups.}\label{sec:6.4} In this section we show that if $X$ is a rigid space and $G$ is a finite group acting on $X$ from the right such that $X$ has a cover by $G$-stable affinoid open subsets, then the quotient $X/G$ exists in the category of rigid spaces. If $X=\Spa(A,A^{\circ})$ is affinoid, then the quotient $X/G:=\Spa(A^{G},(A^{G})^{\circ})$ exists in the category of \emph{affinoid} rigid spaces by \cite[\S 6.3.3 Proposition 2]{bgr}. The purpose of this section is to show that this construction glues. In fact, we show that the quotient is a quotient in Huber's ambient category $\ms{V}$ (\cite[\S 2]{hub}), hence in the category of adic spaces. This is presumably well known, but we have not been able to find a reference.

\medskip

We start with some general discussion on quotients. Let $\underline{X}=(X,\oo_{X},(v_{x})_{x\in X})\in \ms{V}$ and let $G$ be a finite group acting on $\underline{X}$ from the right. Consider the quotient topological space $Y=X/G$ and let $\pi\, :\, X \ra Y$ be the quotient map. We may consider the sheaf of rings $\oo_{Y}:=(\pi_{\ast}\oo_{X})^{G}$ on $Y$. For $y\in Y$ we pick $x\in X$ such that $\pi(x)=y$ and let $v_{y}$ denote the valuation on $\oo_{Y,y}$ given by composing $v_{x}$ with the natural map $\oo_{Y,y} \ra \oo_{X,x}$. This is independent of the choice of $x$ (and representatives for the valuations $v_{x}$). Then we have:

\begin{lemm}\label{lemm: abstq}
With notation as above, the sheaf $\oo_{Y}:=(\pi_{\ast}\oo_{X})^{G}$ is a sheaf of complete topological rings on $Y$ and $(Y, \oo_{Y}, (v_{y})_{y\in Y})$ is in $\ms{V}$ and is the categorical quotient of $X$ by $G$ in $\ms{V}$.
\end{lemm}

\begin{proof}
We prove that $\oo_{Y}$ is a sheaf of complete topological rings, the rest is then immediate from the definitions. First, note that $\oo_{Y}$ is a sheaf since invariants are left exact. Second, note that taking invariants preserve completeness (for example, argue with nets). Finally, taking invariants preserves topological embeddings, so $\oo_{Y}$ is a sheaf of complete topological rings.
\end{proof}

Now let $K$ be a complete non-archimedean field with ring of integers $\oo_{K}$, and fix an element $\vp \in \oo_K$ with $0<|\vp|<1$. Let $B$ be a $K$-algebra of topologically finite type, and assume that $B$ carries a left action over $K$ by a finite group $G$. Then $A:=B^{G}$ is a $K$-algebra of topologically finite type and the inclusion $A \ra B$ is a finite homomorphism by \cite[\S 6.3.3 Proposition 2]{bgr}. We wish to show that $Y:=\Spa(A,A^{\circ})$ is the quotient of $X:=\Spa(B,B^{\circ})$ in $\ms{V}$ by verifying the conditions of Lemma \ref{lemm: abstq}.

\begin{prop} \label{prop: qspace}
$Y=X/G$ as topological spaces.
\end{prop}

\begin{proof}
On topological spaces $\pi$ clearly factors through $X/G$. $\pi$ is closed since it is finite. We know from commutative algebra (the going-up theorem) that ${\rm MaxSpec}(B) \surj {\rm MaxSpec}(A)$ as sets, hence $\pi$ is surjective on classical points and is therefore surjective since it is closed and classical points are dense. Since $\pi$ is closed and surjective it then suffices to show that the fibers of $\pi$ are $G$-orbits, i.e. that $Y=X/G$ as sets. 

\medskip
To show this, we first note that $\pi$ fits into a $G$-equivariant commutative diagram
\[
\xymatrix{X\ar[r]^{\pi}\ar[d] & Y\ar[d]\\
\mathrm{Spv}(B)\ar[r]^{\pi'} & \mathrm{Spv}(A)
}
\]
with injective vertical arrows; here, following Huber, $\mathrm{Spv}(R)$ denotes the set of all equivalence classes of valuations on a given ring $R$.  It now suffices to show that the fibers of $\pi'$ are $G$-orbits.  To see this, choose any valuation $y \in \mathrm{Spv}(A)$, with support $\mathfrak{p}=\mathfrak{p}_y=\Ker y \in \Spec(A)$. By \cite[\S\S V.2.2 Theorem 2(i)]{bou}, $G$ acts transitively on the set of supports $\mathfrak{q}_x \in \Spec(B)$ of valuations $x\in \pi'^{-1}(y)$.  Therefore, choosing any $\mathfrak{q}=\mathfrak{q}_x$ of this form, we need to show that the stabilizer $G_{\mathfrak{q}}$ of $\mathfrak{q} \in \Spec(B)$ acts transitively on the set of valuations $x \in \mathrm{Spv}(\mathrm{Frac}(B/\mathfrak{q}))$ extending the given valuation $y \in \mathrm{Spv}(\mathrm{Frac}(A/\mathfrak{p}))$. By \cite[\S\S V.2.2 Theorem 2(ii)]{bou}, ${\rm Frac}(B/\mf{q})$ is a finite normal extension of ${\rm Frac}(A/\mf{p})$ and $G_{\mf{q}}$ surjects onto ${\rm Aut}({\rm Frac}(B/\mf{q})/{\rm Frac}(A/\mf{p}))$, so this follows from \cite[\S VI.7 Corollary 3]{zs}.

\end{proof}
\begin{coro} \label{coro: affinoidq}
$Y=(X/G, (\pi_{\ast}\oo_{X})^{G}, (v_{xG}))$ and is therefore the quotient of $\Spa(B,B^{\circ})$ by $G$ in $\ms{V}$.
\end{coro}

\begin{proof}
We know that $Y=X/G$ on topological spaces and it suffices to verify that $\oo_{Y}=(\pi_{\ast}\oo_{X})^{G}$ since then the equality of the valuations follows directly. There is a natural map $\oo_{Y} \ra (\pi_{\ast}\oo_{X})^{G}$ and hence it suffices to prove that this is an isomorphism on rational subsets $\{ |f_{1}|,...,|f_{n}| \leq |h|\neq 0 \}$ on $Y$, i.e. that 
$$ B \langle f_{1}/h,...,f_{n}/h \rangle^{G} = B^{G} \langle f_{1}/h,...,f_{n}/h \rangle .  $$

To prove this, note that we have an exact sequence\[
0\to B^{G}\to B\overset{a\mapsto(ga-a)_{g\in G}}{\longrightarrow}\prod_{g\in G}B_{(g)}\]
of finite $B^{G}$-modules. Note also that $B^{G}\langle \tfrac{f_{1}}{h},\dots,\tfrac{f_{n}}{h}\rangle $
is flat over $B^{G}$ and that \[
B\langle \tfrac{f_{1}}{h},\dots,\tfrac{f_{n}}{h}\rangle =B\widehat{\otimes}_{B^{G}}B^{G}\langle \tfrac{f_{1}}{h},\dots,\tfrac{f_{n}}{h}\rangle =B\otimes_{B^{G}}B^{G}\langle \tfrac{f_{1}}{h},\dots,\tfrac{f_{n}}{h}\rangle ,\]
since $B$ is module-finite over $B^{G}$. Then applying $-\otimes_{B^{G}}B^{G}\langle \tfrac{f_{1}}{h},\dots,\tfrac{f_{n}}{h}\rangle $
to the exact sequence above gives\[
0\to B^{G}\langle \tfrac{f_{1}}{h},\dots,\tfrac{f_{n}}{h}\rangle \to B\langle \tfrac{f_{1}}{h},\dots,\tfrac{f_{n}}{h}\rangle \overset{a\mapsto(ga-a)_{g\in G}}{\longrightarrow}\prod_{g\in G}B\langle \tfrac{f_{1}}{h},\dots,\tfrac{f_{n}}{h}\rangle _{(g)}.\]
Since the kernel of the third arrow here is $B\langle \tfrac{f_{1}}{h},\dots,\tfrac{f_{n}}{h}\rangle ^{G}$,
we get $$B^{G}\langle \tfrac{f_{1}}{h},\dots,\tfrac{f_{n}}{h}\rangle =B\langle \tfrac{f_{1}}{h},\dots,\tfrac{f_{n}}{h}\rangle ^{G}$$
as desired.
\end{proof}
 
\begin{coro} \label{coro: quotients}
Let $X$ be a rigid space over $K$ with a right action of a finite group $G$ and assume that there is a cover of $X$ by $G$-stable affinoids. Then the quotient $X/G$ of $X$ by $G$ exists in $\ms{V}$ and is a rigid space. Moreover, the natural map $X \ra X/G$ is finite and if $X$ is affinoid then $X/G$ is affinoid. 
\end{coro}

\begin{proof}
Consider $(X/G, (\pi_{\ast}\oo_{X})^{G}, (v_{xG}))$ where $\pi\, :\, X \ra X/G$ is the quotient map on topological spaces. If this is a rigid space then the remaining statements hold by Lemma \ref{lemm: abstq} and Corollary \ref{coro: affinoidq}. To see that it is a rigid space let $(U_{i})_{i\in I}$ be a cover of $G$-stable affinoids. By $G$-stability $\pi^{-1}(\pi(U_{i}))=U_{i}$ for all $i$ and hence $\pi(U_{i})=U_{i}/G$ is open for all $i$. By Corollary \ref{coro: affinoidq} $(U_{i}/G, (\pi_{\ast}\oo_{U_{i}})^{G}, (v_{xG}))=(X/G, (\pi_{\ast}\oo_{X})^{G}, (v_{xG}))|_{U_{i}}$ is an affinoid rigid space for all $i$, and we are done.
\end{proof}

\end{document}